\documentclass[11pt]{article}
\usepackage{graphicx,color,amsfonts,amsmath,amssymb,enumerate}
 \usepackage{url,fancyhdr,indentfirst}
   \usepackage{amsthm,natbib,comment,bm}
    \usepackage[colorlinks=true,citecolor=blue]{hyperref}
\usepackage{dsfont}
\usepackage{mathrsfs}
\usepackage{tikz-qtree}
\usepackage{xr}
\usepackage{subfigure}
\usepackage{threeparttable}
\usepackage[title]{appendix}

\def\d{\mathrm{d}}

\newcommand{\E}{\mathbb{E}}
\newcommand{\R}{\mathbb{R}}

\newcommand{\N}{\mathbb{N}}
\newcommand{\p}{\mathbb{P}}

\newcommand{\id}{\mathds{1}}

\newcommand{\esssup}{\mathrm{ess\mbox{-}sup}}
\newcommand{\essinf}{\mathrm{ess\mbox{-}inf}}
\renewcommand{\ge}{\geqslant}
\renewcommand{\le}{\leqslant}
\renewcommand{\geq}{\geqslant}
\renewcommand{\leq}{\leqslant}
\renewcommand{\epsilon}{\varepsilon}

\theoremstyle{plain}
\newtheorem{theorem}{Theorem}
\newtheorem{corollary}{Corollary}
\newtheorem{lemma}{Lemma}
\newtheorem{proposition}{Proposition}

\theoremstyle{definition}
\newtheorem{definition}{Definition}
\newtheorem{example}{Example}

\theoremstyle{remark}

\theoremstyle{definition}

\renewcommand{\cite}{\citet}
\renewcommand{\cdots}{\dots}

\usepackage{tikz}

\usepackage[onehalfspacing]{setspace}

\begin{document}

\title{A General Wasserstein Framework for Data-driven Distributionally Robust Optimization: Tractability and Applications}

 \date{\today}

\author{ Jonathan Yu-Meng Li$^{\dag}$, \,  Tiantian Mao$^{\dag\dag}$\\
\\
$^{\dag}$ Telfer School of Management\\
University of Ottawa, Ottawa, ON, Canada\\
\\
$^{\dag\dag}$ Department of Statistics and Finance, School of Management \\
University of Science and Technology of China\\
 Hefei, Anhui, China}

\maketitle

\begin{abstract}
Data-driven distributionally robust optimization is a recently emerging paradigm aimed at finding a solution that is driven by sample data but is protected against sampling errors. An increasingly popular approach, known as Wasserstein distributionally robust optimization (DRO), achieves this by applying the Wasserstein metric to construct a ball  centred at the empirical distribution and finding a solution that performs well against the most adversarial distribution from the ball. In this paper, we present a general framework for studying different choices of a Wasserstein metric and point out the limitation of the existing choices. In particular, while choosing a Wasserstein metric of a higher order is desirable from a data-driven perspective, given its less conservative nature, such a choice comes with a high price from a robustness perspective - it is no longer applicable to many heavy-tailed distributions of practical concern. We show that this seemingly inevitable trade-off can be resolved by our framework, where a new class of Wasserstein metrics, called coherent Wasserstein metrics, is introduced. Like Wasserstein DRO, distributionally robust optimization using the coherent Wasserstein metrics, termed generalized Wasserstein distributionally robust optimization (GW-DRO), has all the desirable performance guarantees: finite-sample guarantee, asymptotic consistency, and computational tractability. The worst-case expectation problem in GW-DRO is in general a nonconvex optimization problem, yet we provide new analysis to prove its tractability without relying on the common duality scheme. Our framework, as shown in this paper, offers a fruitful opportunity to design novel Wasserstein DRO models that can be applied in various contexts such as operations management, finance, and machine learning.
\end{abstract}

\section{Introduction}
Data-driven problems arise from many operations research and machine learning applications where a stochastic optimization problem needs to be solved using sample data drawn from a probability distribution of interest. The goal is to find a solution that performs well in out-of-sample tests against the distribution underlying the stochastic optimization problem. These problems are challenging to solve, because firstly the use of sample data to represent a distribution is prone to sampling errors, and secondly the underlying data-generating distribution in most real-life applications is fundamentally unknown. The increasing availability of large data sets in recent years has renewed the interest of exploring how to best exploit sample data to obtain solutions with favourable out-of-sample performances. One prominent idea is to find a solution that can perform well against distributions that are, in some sense, close to the empirical distribution constructed from the sample data. This, in principle, allows the solution to maximally leverage the information contained in sample data regarding the underlying data-generating distribution while at the same time ensuring the solution does not overly rely on sample data, i.e. avoids overfitting.

An increasingly popular framework to implement this idea is data-driven distributionally robust optimization (DD-DRO). It seeks a solution that performs the best with respect to the most adversarial distribution from a set of distributions, known as ambiguity set, that are close to the empirical distribution according to some predefined metric. The choice of a proper metric is crucial in DD-DRO. Ambiguity sets constructed from different metrics could contain distributions with distinctly different structural properties.  A well-known example is ambiguity sets defined based on the Kullback-Leibler divergence metrics (\cite{KL51}), which contain only discrete distributions whose support is limited to, i.e. a subset of, the support of the empirical distribution (\cite{BDDM13}, \cite{HH13}, \cite{S17}). Solutions of DD-DRO that adopts such ambiguity sets may not generalize well to situations where the underlying distribution has a more general support structure, e.g. taking values other than the observed samples. Another metric that has now been more widely applied in DD-DRO is the Wasserstein metric (\cite{KR58}). Ambiguity sets defined based on the metric contain distributions with a more general distributional structure, including both discrete and continuous distributions. DD-DRO adopting such ambiguity sets, also known as Wasserstein distributionally robust optimization (\cite{EK18}), is attractive in that its solution has potential to generalize well against various forms of distributions that may arise from practical applications. The Wasserstein metric consists of  a family of metrics in different orders, namely the type-$p$ Wasserstein metric, $p \in [1,\infty]$, each defined based on a transportation cost function with a different power order. The type-1 Wasserstein metric, extensively studied in \cite{EK18}, is so far the most popular choice in the applications of Wasserstein DRO. \cite{KENS97} and \cite{GK22} provide a comprehensive study of Wasserstein DRO for the type-$p$ Wasserstein metric of a higher order, $p >1$. Several other works (\cite{BKM19}, \cite{GCK20}, \cite{SKE19}, \cite{SNVD20}, \cite{CBM18}) have explored its applications in machine learning and related data-driven problems. It is known that an ambiguity set defined based on the type-$p$ Wasserstein metric, called type-$p$ Wasserstein ball, would contain only distributions that have finite $p^{{\rm th}}$-order moments (\cite{V08} p.95). With the same radius, a type-$p$ Wasserstein ball is strictly smaller than a type-$q$ Wasserstein ball, where $p>q$, and contains distributions that concentrate more heavily around the sample data. A Wasserstein ball of a higher order is thus less conservative or can be considered more data-driven. The type-$\infty$ Wasserstein ball, in particular, is the most data-driven in that the distributions from the ball concentrate fully in a bounded neighbourhood of sample data.

One can see that when it comes to the choice of a Wasserstein metric, a metric with a lower order would appeal to those who seek an ambiguity set that offers protection against a more adversarial form of distributions, whereas a metric with a higher order would appeal to those who seek an ambiguity set that better exploits the sample data. In many real-life applications, however, a pursuit of both, i.e. exploiting well the information from the sample data yet without dismissing the possibility that the underlying distribution may take an extreme form, is of necessity. For instance, in the context of financial portfolio management, a portfolio needs to be optimized using much of the information from market data but without dismissing the possibility that the return distribution may have a heavy tail, i.e. non-negligible weights on rarely occurring events. We point out first in this paper that the Wasserstein metric as it stands, i.e. the family of type-$p$ Wasserstein metrics, $p \in [1,\infty]$, cannot accommodate this simultaneous pursuit.  This is because ambiguity sets constructed from a  Wasserstein metric with a higher order, e.g. $p\geq 2$, inevitably exclude heavy-tailed distributions\footnote{In this paper, heavy-tailed distributions refer to distributions with finite mean but without finite variance.} (\cite{BAE22}, \cite{DF06}).  One may view this as a trade-off between the pursuit of robustness and data-drivenness when it comes to the choice of a Wasserstein metric. While this trade-off may appear inevitable, i.e. a less conservative choice would dismiss any heavy-tailed distribution, the primary goal of this paper is to present alternative families of Wasserstein metrics that could resolve, or at least lessen, this trade-off. We present a general framework that formalizes this simultaneous pursuit of robustness and data-drivenness in terms of the choice of a Wasserstein metric, and identify a large class of Wasserstein metrics, termed coherent Wasserstein metrics, that allows for exploring this pursuit.

The class of coherent Wasserstein metrics is motivated by the attempt to generalize the type-1 and the type-$\infty$ Wasserstein metric  from a new perspective.  An observation can be made that the type-$\infty$ Wasserstein metric can be viewed as a risk-averse counterpart of the type-1 Wasserstein metric, which replaces the expectation operator of the latter (see \eqref{eq:dwasser} in Section \ref{sec:2}) with the worst-case risk measure, i.e. $\esssup$, to summarize a transportation cost distribution (\cite{KR58}, \cite{RR98}, \cite{V08}). Coherent Wasserstein metrics generalize this observation by adopting a general class of risk measures, namely coherent risk measures (\cite{ADEH99}, \cite{D02}), to summarize a transportation cost distribution, which consists of the expectation and the worst-case risk measure as special cases. Coherent Wasserstein metrics can be interpreted also as general risk-averse formulations of the optimal transport problem arising from the classical definition of the Wasserstein metric (\cite{V08}).   We show that coherent Wasserstein metrics provide a powerful means to reconcile the type-1 and the type-$\infty$ Wasserstein metrics  and offer the opportunity to identify new families of Wasserstein metrics motivated by the popularity of risk measures such as Conditional value-at-risk (CVaR) (\cite{AT02}, \cite{KSZ19}, \cite{EPRWB14}) and expectiles (\cite{BKMG14}, \cite{BB17}, \cite{G11}). Like the Wasserstein metric, coherent Wasserstein metrics are theoretically sound in that they satisfy all necessary properties of a distance metric. We call the resulting DD-DRO formulation generalized Wasserstein distributionally robust optimization (GW-DRO) and show that GW-DRO generally satisfies the desirable conditions posed for DD-DRO, namely finite-sample guarantee, asymptotic consistency, and computational tractability (see Section \ref{sdddro} for detailed definitions).

From an optimization perspective, GW-DRO represents a new class of DRO problems that are distinctly different from existing DRO problems in two aspects. Firstly, the worst-case expectation problem embedded in GW-DRO has a nonlinear constraint in distribution, whereas existing worst-case expectation problems in DRO such as Wasserstein DRO generally have linear constraints in distributions, taking the form of moment constraints. Secondly, as shown in this paper, the worst-case expectation problem of GW-DRO is in general a nonconvex optimization problem in distribution, which to the best of our knowledge has not been studied in the DRO literature. These differences are significant, which render existing DRO analysis no longer applicable to studying the tractability of GW-DRO.
In this paper, we take a different approach to studying the tractability of the worst-case expectation problems. Instead of relying on the common analysis starting from a dual problem formulation, we tackle directly the worst-case expectation problem from a primal perspective, and show how it can be reduced to a finite-dimensional optimization problem. Leveraging this finite-dimensional result, we then show how GW-DRO problems can be tractably solved as convex programs. Our approach of tackling the primal problem first and then deriving more tractable formulations is novel, which opens the door for solving a more general class of DRO problems. As a by-product, it provides an alternative, possibly simpler, way to derive the tractable formulation for Wasserstein DRO.

 In addition to general tractability results, we provide also in-depth analysis of GW-DRO by focusing on two important instances of Wasserstein balls, defined by CVaR- and expectile-based Wasserstein metrics. We show that the worst-case expectation problems in these instances can be solved in closed-form when the loss function is convex Lipschitz continuous and the support set is unconstrained. These closed-form solutions are highly interpretable and structurally comparable to the solutions for the type-1 Wasserstein DRO. In particular, we show that the solutions for the case of expectile-based Wasserstein ball, CVaR-based Wasserstein ball, and type-1 Wasserstein ball are closely related in that they exhibit an ``inclusion" relationship with the first being most general. This applies also to their respective worst-case distributions, with the first having the most flexible, or the richest, worst-case distribution structure. The closed-form solutions, when applied to contexts such as machine learning, could be interpreted also from a regularization perspective (c.f. \cite{KENS97}). For instance, GW-DRO in these applications, when adopting CVaR-based Wasserstein balls, boils down to aggregating different regularized empirical minimization problems into a single minimization problem, and thus could be viewed as an ensemble of regularized models.

In the following, we summarize the key contributions of this paper.
\begin{enumerate}
\item We propose, in the spirit of data-driven distributionally robust optimization, a theoretically sound framework for studying different families of Wasserstein metrics. The framework sheds light on the potential limitation of the existing family of Wasserstein metrics and offers guidance to discover new families of Wasserstein metrics better suited for designing a richer, yet not only conservative, form of ambiguity sets.

\item We introduce a new family of coherent Wasserstein metrics and show that the corresponding distributionally robust optimization models, i.e. GW-DRO, enjoy all the desirable properties of a data-driven model for solving a stochastic program, namely finite-sample guarantee, asymptotic consistency, and tractability.

\item We provide a new systematic approach to studying the tractability of distributionally robust optimization problems without relying on the common duality scheme. It allows for tackling non-convex worst-case expectation problems naturally arising from GW-DRO and proving their tractability with the discovery of hidden convexity.

\item We present the application of GW-DRO to operations, finance, and machine learning problems. In particular, we show that in many of these problems a deep connection can be drawn between Wasserstein DRO and GW-DRO. Most notably, while the former is known to have a regularization interpretation in applications such as machine learning, the latter offers an even richer, and more novel, regularization interpretation.
\end{enumerate}

\section{Wasserstein data-driven distributionally robust optimization}\label{sec:2}

As the basic setup, we denote a decision vector by $x \in \mathbb{R}^{n}$, a random vector of interest by $\xi \sim \mathbb{P}$, supported on a convex set $\Xi \subseteq \mathbb{R}^{m}$, i.e. $\mathbb{P}(\xi \in \Xi)=1$,  and a loss function $h: \mathbb{R}^{n} \times \mathbb{R}^{m}   \rightarrow \mathbb{R}$
by $h(x,\xi)$, which depends on a made decision $x$ and the realization of the random vector $\xi$.  In many practical problems of interest, one seeks to find a decision $x$ that minimizes the expected loss $\mathbb{E}^{\mathbb{P}}[h(x, \xi)]$, i.e. solving
\begin{equation}\label{eq-220324-2}
J^{\star}:=\inf _{x \in \mathbb{X}}\left\{\mathbb{E}^{\mathbb{P}}[h(x, \xi)]=\int_{\Xi} h(x, \xi) \mathbb{P}(\mathrm{d} \xi)\right\},
\end{equation}
where $\mathbb{X}$ denotes a feasible set of solutions.

Data-driven optimization refers to finding a solution to the above problem when the distribution $\mathbb{P}$ can only be partially observed through a finite set of data $\widehat{\xi_1},..., \widehat{\xi}_N$ sampled independently from the distribution. One common data-driven method is to directly replace the distribution $\mathbb{P}$ with the empirical distribution
$\widehat{\mathbb{P}}_{N}:=\frac{1}{N} \sum_{i=1}^{N} \delta_{ \widehat{\xi_i} }$ and solve instead the following optimization problem
\begin{equation}\label{eq-220324-3}
\widehat{J}_{{\rm SAA}}:=\inf _{x \in \mathbb{X}}\left\{\mathbb{E}^{\widehat{\mathbb{P}}_{N}}[h(x, \xi)]=\frac{1}{N}\sum_{i=1}^N h(x, \widehat{\xi}_i)\right\}.
\end{equation}
This method, also known as sample average approximation (SAA), is susceptible to sampling errors and suffers from the issue of the optimizer's curse (bias), i.e. disappointing out-of-sample performances (\cite{KENS97}). As a remedy, data-driven distributionally robust optimization (DD-DRO) was proposed as a new data-driven method, which offers a solution that mitigates the adverse impact of sampling errors by solving the following minimax optimization problem
\begin{equation}\label{eq:wcprob}
\widehat{J}_{N}:=\inf _{x \in \mathbb{X}} \sup _{\mathbb{P} \in \mathbb{B}(\widehat{\mathbb{P}}_{N} )} \mathbb{E}^{\mathbb{P}}[h(x, \xi)].
\end{equation}
The set $\mathbb{B}(\widehat{\mathbb{P}}_{N})$, known as ambiguity set, is a set constructed based on the empirical distribution $\widehat{\mathbb{P}}_{N}$, which contains the unknown distribution $\mathbb{P}$ in \eqref{eq-220324-2} with high probability. A solution generated from \eqref{eq:wcprob} is robust against sampling errors in that it is guaranteed to perform the best with respect to the most adversarial distribution from the set $\mathbb{B}(\widehat{\mathbb{P}}_{N})$.
A natural construction of the set $\mathbb{B}(\widehat{\mathbb{P}}_{N})$ takes the general form of
\begin{equation} \label{general_ball}
\mathbb{B}_{\epsilon}^{d}\left(\widehat{\mathbb{P}}_{N}\right):=\left\{\mathbb{P} \in \mathcal{M}(\Xi): d\left(\widehat{\mathbb{P}}_{N}, \mathbb{P}\right) \leq \varepsilon\right\},
\end{equation}
where  $\mathcal M(\Xi)$ denotes the set of all distributions supported on $\Xi$, $d(\mathbb{P}_1, \mathbb{P}_2)$ stands for a probability metric that measures the distance between any two distributions $\mathbb{P}_1, \mathbb{P}_2 \in \mathcal{M}(\Xi)$, and $\varepsilon$ refers to the radius of the ball centred at the empirical distribution $\widehat{\mathbb{P}}_{N}$.

The quality of the solution generated from \eqref{eq:wcprob} depends critically on the structure of the ambiguity set $\mathbb{B}_{\epsilon}^{d}(\widehat{\mathbb{P}}_{N})$, which in turn depends on the choice of the probability metric $d$. Among several proposed probability metrics, the (type-1) Wasserstein metric (\cite{KR58})
\begin{align} \label{eq:dwasser}
d_{\mathrm{W}}\left(\mathbb{P}_{1}, \mathbb{P}_{2}\right):=\inf \left\{\E^{\Pi}\left[\left\|\xi_{1}-\xi_{2}\right\| \right]\left|\begin{array}{l} \Pi \text { is a joint distribution of } \xi_{1} \text { and } \xi_{2} \\ \text { with marginals } \mathbb{P}_{1} \text { and } \mathbb{P}_{2}, \text { respectively }\end{array}\right.\right\},
\end{align}
has stood out as a popular choice, given its applicability to a large class of distributions, i.e. any distributions
$\mathbb{P}_{1}, \mathbb{P}_{2} \in \mathcal{M}(\Xi)$ that have finite first moments. In particular, it allows for constructing a ball
$$
\mathbb{B}_{\varepsilon}^{\mathrm{W}} \left(\widehat{\mathbb{P}}_{N}\right):=\left\{\mathbb{P} \in \mathcal{M}(\Xi): d_{\mathrm{W}}\left(\widehat{\mathbb{P}}_{N}, \mathbb{P}\right) \leq \varepsilon\right\}
$$
that contains a rich set of distributions.


The ball $\mathbb{B}_{\varepsilon}^{\mathrm{W}} (\widehat{\mathbb{P}}_{N})$ is advantageous from a robustness perspective, i.e. containing various forms of distributions, but its flip side is less mentioned in the literature of Wasserstein DRO. Namely, it may contain overly-disperse distributions that differ too noticeably from the empirical distribution and thus be considered overly-conservative. As a useful contrast to highlight the limitation of the ball constructed from the type-1 Wasserstein metric, let us consider the following variant of Wasserstein metric, known as the type-$\infty$ Wasserstein metric:
\begin{align} \label{eq:dw}
d_{\infty}\left(\mathbb{P}_{1}, \mathbb{P}_{2}\right)
:=\inf \left\{\esssup^\Pi \left\|\xi_{1}-\xi_{2}\right\| \left|\begin{array}{l} \Pi \text { is a joint distribution of } \xi_{1} \text { and } \xi_{2} \\ \text { with marginals } \mathbb{P}_{1} \text { and } \mathbb{P}_{2}, \text { respectively }\end{array}\right.\right\}.
\end{align}
Its induced ball
$$
\mathbb{B}_{\varepsilon}^{{\rm wc}}\left(\widehat{\mathbb{P}}_{N}\right):=\left\{\mathbb{P} \in \mathcal{M}(\Xi): d_{\infty}\left(\widehat{\mathbb{P}}_{N}, \mathbb{P}\right) \leq \varepsilon\right\}
$$
would contain only distributions that fully concentrate in a neighbourhood of samples $\widehat{\xi_1},..., \widehat{\xi}_N$, bounded by $\varepsilon$, and thus resemble to a greater extent the empirical distribution. The ball constructed from the type-$\infty$ Wasserstein metric thus has the merit of data-drivenness. The price to pay to adopt the type-$\infty$ Wasserstein metric is high, nonetheless, from a robustness perspective, as the metric is only applicable to distributions with bounded support.


One can see that the two Wasserstein metrics $d_{\mathrm{W}}$ and $d_{\infty}$ essentially differ in how they summarize the distribution of $\left\|\xi_{1}-\xi_{2}\right\|$. To formalize this point, we call a random variable $X$ a transportation cost random variable from $\mathbb{P}_1$ to $\mathbb{P}_2$ if  there exist $ {\xi}_1 \sim \mathbb{P}_1,\;\; \xi_2 \sim \mathbb{P}_2$ such that $X\stackrel{\rm d} = \|\xi_1-\xi_2\|$. Let $\rho$ denote a real-valued function that maps a random variable $X$ to a real value. In the case of type-1 Wasserstein metric $d_{\mathrm{W}}$, we have $\rho:= \mathbb{E}$, whereas in the case of type-$\infty$ Wasserstein metric we have $\rho:= \esssup$. The type-1 Wasserstein metric $d_{\mathrm{W}}$ could induce an overly-conservative ball $\mathbb{B}_{\varepsilon}^{\mathrm{W}}(\widehat{\mathbb{P}}_{N})$, because the expectation $\mathbb{E}$ is  indistinguishable for deviations of $X$ at different quantiles, whereas the type-$\infty$ Wasserstein metric induces a ball $\mathbb{B}_{\varepsilon}^{{\rm wc}}(\widehat{\mathbb{P}}_{N})$ that can contain only distributions with bounded support, because $\textrm{esssup}$, as the worst-case risk measure, is the strongest tail measure.

Taking this perspective, we seek to identify in this paper a new class of Wasserstein metrics that can reconcile the type-1 and type-$\infty$ Wasserstein metrics so that these metrics can be well justified from both robustness and data-drivenness perspective. We formalize this pursuit in the next section, where a new class of Wasserstein metrics, called coherent Wasserstein metrics, will be introduced.

\subsection{Coherent Wasserstein metrics} \label{cwm}
We begin by defining $\{\rho_{\alpha}\}_{\alpha\in A}$ as a class of real-valued functions used to summarize the distribution of a transportation cost random variable, where $A$ is an index set. The induced Wasserstein distance between two distributions $\mathbb{P}_{1} $ and $\mathbb{P}_{2}$ is defined as
 \begin{equation}\label{eq:distance}
d_{\rho_{\alpha}}(\mathbb{P}_1,\mathbb{P}_2) :=\inf \left\{\rho_\alpha^{\Pi}(\|\xi_{1} - \xi_{2}\|)\left|\begin{array}{l} \Pi \text { is a joint distribution of } \xi_{1} \text { and } \xi_{2} \\ \text { with marginals } \mathbb{P}_{1} \text { and } \mathbb{P}_{2}, \text { respectively }\end{array}\right.\right\},
\end{equation}
and a ball of radius $\varepsilon$ centred at the  empirical distribution $\widehat{\mathbb{P}}_{N}$ can be defined accordingly as
$$
\mathbb{B}^{\rho_\alpha}_{\varepsilon}\left(\widehat{\mathbb{P}}_{N}\right):=\left\{\mathbb{P} \in \mathcal{M}(\Xi): d_{\rho_{\alpha}}\left(\widehat{\mathbb{P}}_{N}, \mathbb{P}\right) \leq \varepsilon\right\}.
$$

The novelty of our framework lies in taking a set perspective, i.e. $\alpha \in A$, to study properties that a whole family of Wasserstein metrics $\{d_{\rho_{\alpha}}\}_{\alpha \in A}$ should satisfy, rather than considering each metric separately. This perspective, which is largely missing in the literature of Wasserstein distributionally robust optimization is essential, we believe, when it comes to studying the choice of a Wasserstein metric.  Built upon the observation made about the type-1 and type-$\infty$ Wasserstein metrics $d_{\mathrm{W}}$ and $d_{\infty}$, namely that the former is advantageous from a robustness perspective whereas the latter is advantageous from a data-driven perspective, we define the following two desirable properties for a family of metrics $\{d_{\rho_{\alpha}}\}_{\alpha \in A}$. These two properties capture the simultaneous pursuit of robustness and data-drivenness underlying the philosophy of data-driven distributionally robust optimization.

 \begin{itemize}
 \item [(i)]{\bf (Robustness)} A family of metrics $\{d_{\rho_{\alpha}}\}_{\alpha \in A}$ is said to have the property of robustness if for each $\alpha\in A$, $\rho_\alpha$ is well-defined (takes finite value) for any transportation cost random variable $X$ that has finite first moment, i.e. $L^1$ random variables. {Any   distribution with finite mean is contained in a Wasserstein ball $\mathbb{B}^{\rho_\alpha}_{\varepsilon} (\widehat{\mathbb{P}}_{N} )$ for some $\varepsilon >0$.}

\item [(ii)] {\bf (Data-drivenness)}  A family of metrics $\{d_{\rho_{\alpha}}\}_{\alpha \in A}$ is said to have the property of data-drivenness if there exists a sequence of indices $\alpha_n\in A$, $n\in\N$, such that  $\rho_{\alpha_n}$  converges to the worst-case risk measure $\esssup$. The Wasserstein ball $\mathbb{B}^{\rho_{\alpha_n}}_{\varepsilon}(\widehat{\mathbb{P}}_{N} )$ converges to $\mathbb{B}_{\varepsilon}^{{\rm wc}}(\widehat{\mathbb{P}}_{N} )$, as $n \rightarrow \infty$.
 \end{itemize}

These two properties together ensure that a family of metrics $\{d_{\rho_{\alpha}}\}_{\alpha \in A}$ is rich enough to, on the one hand, accommodate distributions with a more adversarial form, e.g. heavy-tailed distributions, like the type-1 Wasserstein metric, and on the other hand be used to approximate the functionality of the type-$\infty$ Wasserstein metric. Clearly, the singleton $\{d_{\mathrm{W}}\}$ satisfies robustness but not data-drivenness, whereas the singleton $\{d_{\infty}\}$ satisfies data-drivenness but not robustness.

\begin{definition}
We call a family of metrics $\{d_{\rho_{\alpha}}\}_{\alpha \in A}$ data-driven distributionally robust Wasserstein {\bf (DD-DRW)} metrics if they satisfy both the properties of robustness and data-drivenness.
\end{definition}

When $\rho_p(X)=\E[X^p]^{1/p}$, $p\in [1,\infty)$,  the induced distance is the Wasserstein metric of order $p$. It is clear that the family $\{d_{\rho_{p}}\}_{p \in [1,\infty)}$ is not {\bf DD-DRW}, because it satisfies data-drivenness, i.e.
$\rho_p(X)$ converges to $\esssup (X)$ as $p\to \infty$ but not robustness, i.e. the ambiguity set $\mathbb{B}^{\rho_\alpha}_{\varepsilon} (\widehat{\mathbb{P}}_{N} )$ fails to account for heavy-tailed distributions for some $p>1$. This points out the potential limitation of applying the family of $p^{{\rm th}}$-order Wasserstein metrics. Namely, the price that needs to be paid to construct a less conservative ambiguity set $\mathbb{B}^{\rho_\alpha}_{\varepsilon} (\widehat{\mathbb{P}}_{N} )$  is high from a distributionally robust perspective -- one has to forgo any heavy-tailed distribution of practical interest.

It is natural to wonder if the limitation of the family $\{d_{\rho_{p}}\}_{p \in [1,\infty)}$ lies in its use of $p^{{\rm th}}$-order power function. We show below that the limitation comes more fundamentally from the use of expected functionals to summarize the transportation cost random variable $X$.

\begin{proposition} \label{pro_eu}
The family of metrics $\{d_{\rho_{\alpha}}\}_{\alpha \in A}$, where $\rho_\alpha(X)= \ell^{-1}_\alpha(\E[\ell_\alpha(X)])$\footnote{For a non-decreasing function $\ell$, its inverse function is defined as  $\ell^{-1}(x) =\inf\{y: \ell(y)\ge x\}$.}, $\alpha\in A$ and $\ell_\alpha$ is increasing convex function and $\ell_\alpha(0)=0$, $\alpha\in A$, is not {\bf DD-DRW}.
\end{proposition}

We now introduce a new class of Wasserstein metrics, called coherent Wasserstein metrics, that generalize the type-1 and type-$\infty$ Wasserstein metrics from a risk measure perspective.
\begin{definition} (Coherent Wasserstein metrics)
A metric $d_{\rho}(\cdot,\cdot):\mathcal M^2\to\R_+$ is called a coherent Wasserstein metric if it takes the form of
\begin{equation}\label{eq:dis1}
d_{\rho}( \mathbb{P}_{1}, \mathbb{P}_{2}) :=\inf \left\{\rho^{\Pi}(\|\xi_{1} - \xi_{2}\|)\left|\begin{array}{l} \Pi \text { is a joint distribution of } \xi_{1} \text { and } \xi_{2} \\ \text { with marginals } \mathbb{P}_{1} \text { and } \mathbb{P}_{2}, \text { respectively }\end{array}\right.\right\},
\end{equation}
where $\rho$ is a law-invariant coherent risk measure, i.e. satisfying $\rho(0)=0$ and the following properties:

(translation invariance) $\rho(X+c) = \rho(X)+c$ for any $c\geq 0$,

(monotonicity) $\rho(X_1) \geq \rho(X_2)$ for any $X_1 \geq X_2$,

(subadditvity) $\rho(X_1 + X_2) \leq \rho(X_1) + \rho(X_2)$,

(positive homogeneity) $\rho(cX)=c\rho(X)$ for any $c>0$,

(law invariance) $\rho(X_1)=\rho(X_2)$ for any $X_1 \stackrel{\rm d} = X_2$.

\end{definition}
The use of a law-invariant coherent risk measure $\rho$ is motivated by its well-established properties in the literature of risk measures (\cite{ADEH99}, \cite{K01}) and that it naturally includes the expectation $\mathbb{E}$ and the worst-case measure $\esssup$, as limiting cases. Coherent Wasserstein metrics can be viewed as natural risk-averse formulations of the classical optimal transport problem (\cite{V08}).  We show firstly that coherent Wasserstein metrics, like the classical Wasserstein metrics, are valid distance metrics.

\begin{proposition}\label{pr:1-sec1}
Any coherent Wasserstein metric $d_{\rho}(\cdot,\cdot):\mathcal M(\Xi)\times \mathcal M(\Xi)\to \R_+$ satisfies the following properties of a distance metric
\begin{itemize}
\item [(i)]  (Identity of indiscernibles) $d_{\rho}(\mathbb{P}_{1}, \mathbb{P}_{2})=0$ if and only if $\mathbb{P}_{1}= \mathbb{P}_{2}$.
\item [(ii)] (Symmetry) $d_{\rho}(\mathbb{P}_{1}, \mathbb{P}_{2})=d_\alpha(\mathbb{P}_{2}, \mathbb{P}_{1})$ for any $\mathbb{P}_{1}, \mathbb{P}_{2}$.
\item [(iii)] (Triangle inequality) $d_{\rho}\left(\mathbb{P}_{1}, \mathbb{P}_{2} \right) + d_{\rho}\left(\mathbb{P}_{2}, \mathbb{P}_{3} \right) \ge  d_{\rho}\left(\mathbb{P}_{1},  \mathbb{P}_{3} \right)$ for any $\mathbb{P}_{1}$, $\mathbb{P}_{2}$,  $\mathbb{P}_{3}$.
\item [(iv)] (Non-negativity) $d_{\rho}(\mathbb{P}_{1}, \mathbb{P}_{2})\ge 0$ for any $\mathbb{P}_{1}, \mathbb{P}_{2}$.
\end{itemize}
\end{proposition}

It turns out that coherent Wasserstein metrics offer the needed flexibility for building a family of metrics $\{d_{\rho_{\alpha}}\}_{\alpha \in  A}$ satisfying the property of {\bf DD-DRW}. We highlight below that a family of metrics $\{d_{\rho_{\alpha}}\}_{\alpha \in  A}$ composed of coherent Wasserstein metrics naturally satisfies the property under very mild conditions. Recall that
$\operatorname{VaR}_{\alpha}(X)$ is the Value-at-Risk of $X$ at level $\alpha$ defined by
 $$
\operatorname{VaR}_{\alpha}(X)  =F^{-1}(\alpha) = \inf\{x: F(x)\ge \alpha\}, ~~~X\sim F,
$$
and a function $g: [0,1] \rightarrow [0,1]$ is called a distortion function if it is increasing and satisfies $g(0)=0, g(1)=1$. We denote  the left-derivative function of $g$ by $g'$.
\begin{proposition}
\label{prop:3insec2}
A family of metrics $\{d_{\rho_{\alpha}}\}_{\alpha \in  A}$, where $d_{\rho_{\alpha}}$ is a coherent Wasserstein metric,  satisfies {\bf DD-DRW} if and only if for every $\alpha \in A$, $\rho_\alpha$ can be represented by
\begin{equation}\label{eq:convex-kusuoka}
\rho_\alpha(X)=\sup _{g \in {\mathcal H}_{\rho_\alpha} } \int_{0}^{1} \operatorname{VaR}_{\alpha}(X)  \mathrm{d} g(\alpha),
\end{equation}
where ${\mathcal H}_{\rho_\alpha}$ is a subset of convex distortion functions satisfying
$c_{\alpha} := \sup_{g \in {\mathcal H}_{\rho_ \alpha}} \left\|g^{\prime}\right\|_{\infty} < \infty,$
and $\exists \, \alpha_n\in A,$ $n \in \N$ such that $c_{\alpha_n} \rightarrow \infty$ as $n\rightarrow \infty$.
 \end{proposition}

The representation \eqref{eq:convex-kusuoka} is known as the dual representation of a law-invariant coherent risk measure. The property of {\bf DD-DRW}, particularly the {\bf robustness} condition, boils down to requiring first the existence of such a representation. This is a very mild condition in that any lower-semicontinuous coherent risk measure is known to have a dual representation (see, e.g. \cite{K01}, \cite{JST06}   and \cite{R13}). This observation, more importantly, reveals that to build a family of metrics $\{d_{\rho_{\alpha}}\}_{\alpha \in  A}$ satisfying {\bf DD-DRW}, a generalization of Wasserstein metrics from a dual perspective, i.e. coherent Wasserstein metrics, is critical. This is in shape contrast to a generalization from an expected functional perspective, i.e. Proposition \ref{pro_eu}. The property of {\bf DD-DRW} further requires that the set ${\mathcal H}_{\rho_\alpha}$ in \eqref{eq:convex-kusuoka} contains only
Lipschitz continuous distortion functions  with uniformly bounded Lipschitz constants, i.e. bounded by $c_{\alpha}$, and that there exists a sequence of $\sup_{g \in {\mathcal H}_{\rho_\alpha}} \left\|g^{\prime}\right\|_{\infty}$ in $ \alpha\in A$ converges to the infinity. It is not hard to identify families of coherent risk measures satisfying these conditions. In particular, we identify the following two families of coherent Wasserstein metrics, defined through Conditional Value-at-Risk (CVaR) and expectiles, that satisfy {\bf DD-DRW}. CVaR and expectiles are two most popular risk measures proposed as convex substitutes for the traditional risk measure, Value-at-Risk (VaR).  In the remainder of this paper, we will pay particular attention to these two families of coherent Wasserstein metrics to demonstrate the practical value of our new framework.



\subsubsection*{CVaR-Wasserstein Metric}
Take  $\rho$ as CVaR at level $\alpha \in [0,1)$, i.e.,
$$\rho(X)={\rm CVaR}_\alpha(X) = \frac1{1-\alpha}\int_\alpha^1 {\rm VaR}_u(X)\d u , ~~\alpha\in [0,1).$$  We obtain the  following metric
\begin{align} \label{eq:dcvar}
d_{\rm CVaR_\alpha}\left(\mathbb{P}_{1}, \mathbb{P}_{2}\right)
:=\inf \left\{ {\rm CVaR}_\alpha^{\Pi}(\left\|\xi_{1}-\xi_{2}\right\|) :  \Pi \in  \Pi (\mathbb{P}_{1},\mathbb{P}_{2})\right\}
\end{align}
and the {\rm CVaR}-Wasserstein ball
$$
\mathbb{B}_{(1),\varepsilon}^{\alpha}\left(\widehat{\mathbb{P}}_{N}\right):=\left\{\mathbb{P} \in \mathcal{M}(\Xi): d_{\rm CVaR_\alpha}\left(\widehat{\mathbb{P}}_{N}, \mathbb{P}\right) \leq \varepsilon\right\},\; \alpha \in [0,1).
$$



\subsubsection*{Expectile-Wasserstein Metric}
Recall that the expectile $e_{\alpha}(X)$ of $X$ at level $\alpha \in[0,1)$ is defined as the unique solution to
$$
\alpha \mathbb{E}\left[(X-x)_{+}\right]=(1-\alpha) \mathbb{E}\left[(X-x)_{-}\right],
$$
where $a_+=\max\{a,0\}$ and $a_-=\max\{-a,0\}$. $e_{\alpha}(X)$ is coherent for any $\alpha \in [1/2,1)$, reduces to the mean $\E$ when $\alpha=1/2$ and converges to the worst-case risk measure $\esssup$ as $\alpha \rightarrow 1$.

Taking $\rho$ as expectile $e_{\alpha}(X)$ at level $\alpha \in [1/2,1)$, we have the following metric
\begin{align}\label{def-exp}
d_{e_{\alpha}}\left(\mathbb{P}_{1}, \mathbb{P}_{2}\right):=\inf \left\{e_{\alpha}^{\Pi}\left(\left\|\xi_{1}-\xi_{2}\right\|\right): \Pi \in \Pi (\mathbb{P}_{1},\mathbb{P}_{2}) \right\}
\end{align}
and the expectile-Wasserstein ball
$$
\mathbb{B}_{(2),\varepsilon}^{\alpha}\left(\widehat{\mathbb{P}}_{N}\right):=\left\{\mathbb{P} \in \mathcal{M}(\Xi):  d_{e_{\alpha}}\left(\widehat{\mathbb{P}}_{N}, \mathbb{P}\right) \leq \varepsilon\right\},\; \alpha \in [1/2,1).
$$
We close this section by providing a simple demonstration of how the family of {\rm CVaR}-Wasserstein metrics allows for constructing Wasserstein balls that can, on the one hand, contain heavy-tailed distributions of practical interest and on the other hand converge to the type-$\infty$ Wasserstein ball, as $\alpha \rightarrow 1$.  It is worth noting that by adopting a family of coherent Wasserstein metrics  $\{d_{\rho_{\alpha}}\}_{\alpha \in  A}$, the property of {\bf robustness} in fact implies that for any $\varepsilon>0$, the ambiguity set $\mathbb{B}^{\rho_\alpha}_{\varepsilon}(\widehat{\mathbb{P}}_{N} )$ always contains a heavy-tailed distribution.

\begin{example}
A function $F_{\gamma,\beta}$ with  $\gamma,\beta>0$ is called a Pareto distribution if
$$F_{\gamma,\beta} (x)=1-\left( 1+ \frac{x }\beta\right)^{-\gamma},~~x\ge 0.$$
Suppose that $\widehat{\mathbb P}=\delta_{0}$, i.e. a point mass at $0$. Let us define the following two sets. The first is based on the $p^{{\rm th}}$-order Wasserstein metric for some $p\geq2$, whereas the second is based on the CVaR-Wasserstein metric for some $\alpha \in [0,1)$
$$
\mathbb B_1(p)=\{ F_{\gamma,\beta} : d_{W_p}(F_{\gamma,\beta} ,\widehat{\mathbb P} )\le \epsilon,\; \gamma,\beta>0\},\;\; p \geq 2,$$
$$\mathbb B_2(\alpha)=\{ F_{\gamma,\beta} : d_{\rm CVaR_\alpha}(F_{\gamma,\beta} ,\widehat{\mathbb P} )\le \epsilon,\;\gamma,\beta>0\},\;\; \alpha \in [0,1).
$$

Figure \ref{worst-dis} demonstrates Pareto distributions with different $\gamma$ that are feasible to the CVaR-Wasserstein ball $\mathbb B_2(\alpha)$ for $\alpha=0$ (the left figure) and for $\alpha = 0.99$ (the right figure). Note first that none of the heavy-tailed distributions in the figures are feasible to the type-$p$ Wasserstein ball $\mathbb B_1(p)$, $p\geq 2$, since
for any $p > \gamma$, $d_{W_p}(F_{\gamma,\beta} ,\widehat{\mathbb P} )=\infty$, and thus, $F_{\gamma,\beta}\not \in \mathbb B_1(p)$. In contrast, for any $\alpha\in [0,1)$ and $\gamma>1$, there always exists $\beta$ such that $ F_{\gamma,\beta} \in \mathbb B_2(\alpha).$\footnote{
For $\alpha\in(0,1)$, it holds that
\begin{align*}
 d_{\rm CVaR_\alpha}(F_{\gamma,\beta} ,\widehat{\mathbb P} )
   & = \frac1{1-\alpha} \int_\alpha^1  \left[\beta(1-u)^{-1/\gamma}-\beta \right] d u  =  \frac{\gamma\beta}{\gamma-1}(1-\alpha)^{-1/\gamma}-\beta.
\end{align*}
This implies
  $$
 \mathbb B_2(\alpha)=\left\{ F_{\gamma,\beta} : \frac{\gamma}{\gamma-1}(1-\alpha)^{-1/\gamma}-1 \le \frac\epsilon\beta,\;\gamma,\beta>0\right\}.
$$
Note that  $\lim_{\beta\to0} \epsilon/\beta=\infty$ which implies for any $\gamma>1$, there exists $\beta>0$ small enough such that $\frac{\gamma}{\gamma-1}(1-\alpha)^{-1/\gamma}-1 \le \epsilon/\beta$.
So, for each $\gamma>1$, there exists $\beta>0$ such that $F_{\gamma,\beta} \in  \mathbb B_2(\alpha)$.}
Moreover, comparing the feasible Pareto distributions between the two figures, one can see that the Pareto distributions in the right figure (the case $\alpha=0.99$) concentrate significantly around the sample point $0$ while retaining ``a bit of" heavy tail. This showcases how the family of CVaR-Wasserstein allows for the simultaneous pursuit of robustness and data-drivenness. In the case $\alpha=0$ (the left figure), the CVaR-Wasserstein metric reduces to the type-1 Wasserstein metric and one can see from the figure that the feasible Pareto distributions disperse to the right noticeably away from the sample point, which shows the conservative nature of the type-1 Wasserstein metric.

\begin{figure}[h]
\begin{center}
\includegraphics[scale=0.45]{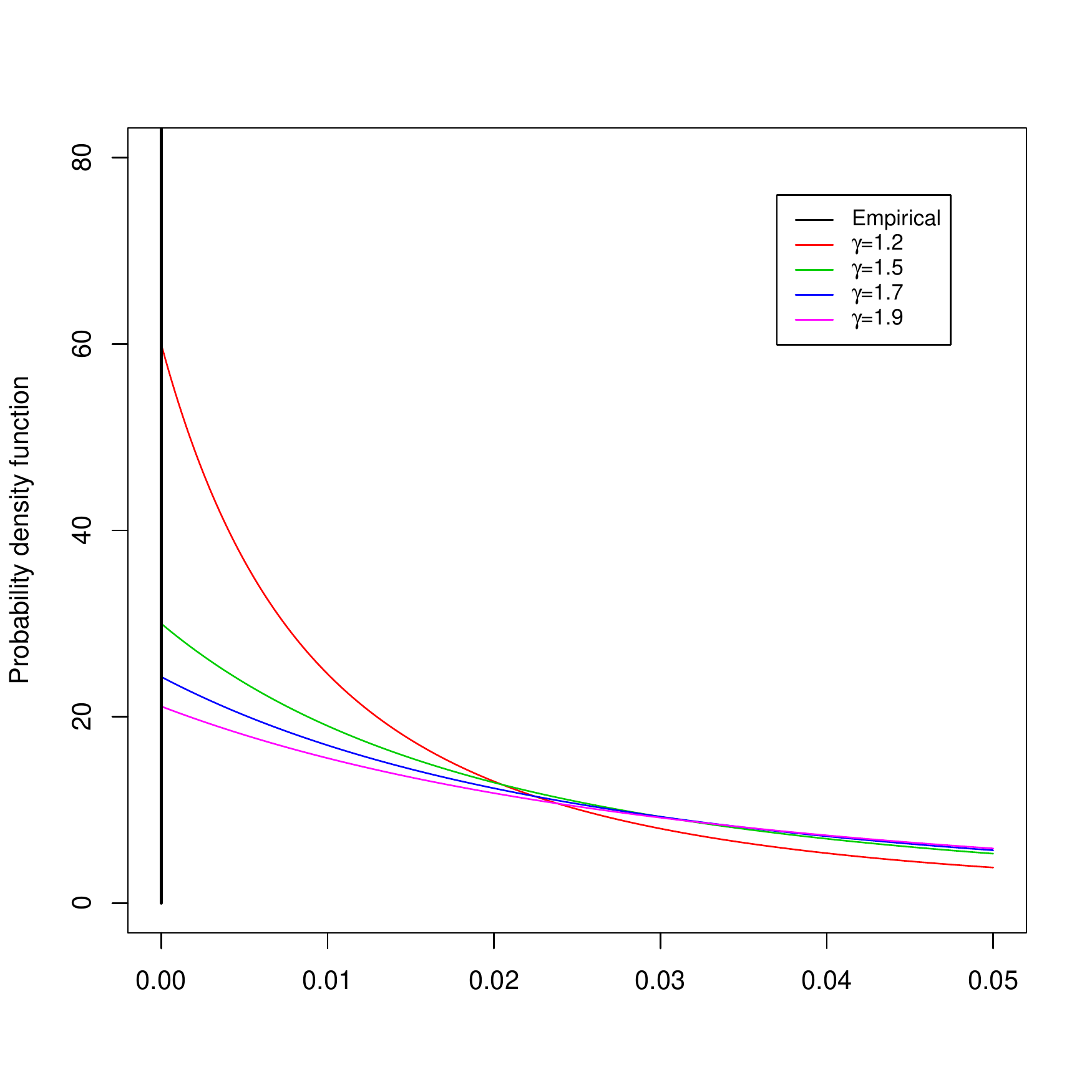}
\includegraphics[scale=0.45]{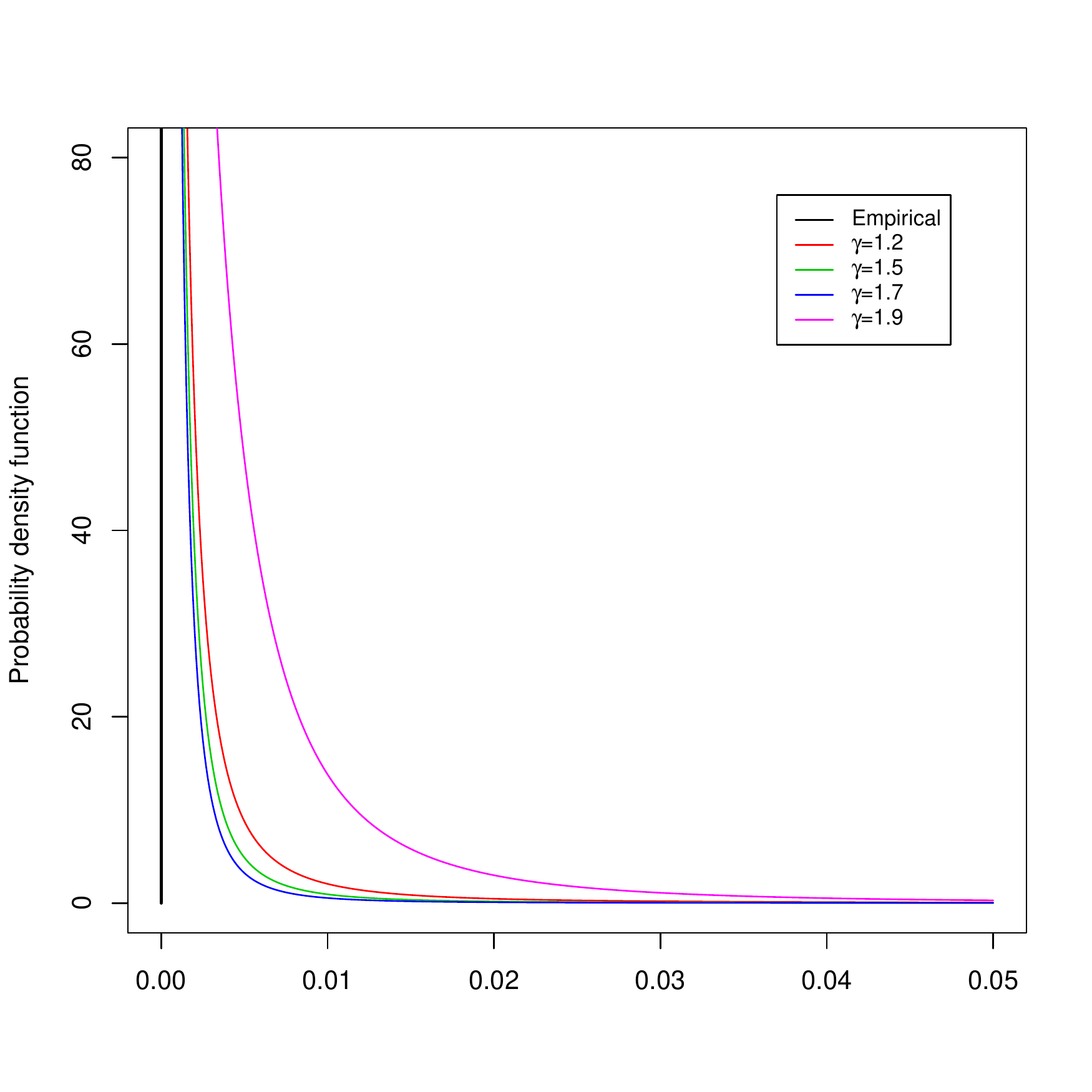}
\caption{Feasible distributions in $ \mathbb B_2(\alpha)$ with $\epsilon=0.1$ for  $\alpha=0$ (left) and  $\alpha=0.99$ (right).}
\label{worst-dis}
\end{center}
\end{figure}
\end{example}

\subsection{Generalized Wasserstein distributionally robust optimization} \label{sdddro}

We call the data-driven distributionally robust model \eqref{eq:wcprob} with an ambiguity set
 defined based on a coherent Wasserstein metric $d_{\rho_{\alpha}}$, i.e.
 $\mathbb{B}(\widehat{\mathbb{P}}_{N}):= \mathbb{B}^{\rho_\alpha}_{\varepsilon}\left(\widehat{\mathbb{P}}_{N}\right)$,
generalized Wasserstein distributionally robust optimization (GW-DRO) model. From this point on, we let $\widehat{J}_{N}$ and $\widehat{x}_{N}$ denote respectively the optimal value and the optimal solution to the GW-DRO model.

We will demonstrate throughout this paper that GW-DRO, like Wasserstein DRO (\cite{EK18}), has all the desirable properties of a data-driven model for solving the stochastic program \eqref{eq-220324-2}, namely finite-sample guarantee, asymptotic consistency, and tractability. The first refers to the guarantee that the out-of-sample performance of $\widehat{x}_{N}$ can be bounded, with some confidence level, by the optimal value $\widehat{J}_{N}$ when the ambiguity set $\mathbb{B}^{\rho_\alpha}_{\varepsilon}\left(\widehat{\mathbb{P}}_{N}\right)$ is properly calibrated, the second refers to assurance that $\widehat{J}_{N}$ and $\widehat{x}_{N}$ would converge respectively to the optimal value and solution to the nominal problem \eqref{eq-220324-2} as $N \rightarrow \infty$, and the third refers to the computational tractability of solving the minimax problem \eqref{eq:wcprob} for many loss functions $h(x,\xi)$ and sets $\mathbb{X}$.

 We provide precise statements regarding the first two properties below. In particular, we highlight that GW-DRO enjoys these two properties under a rather mild condition on the risk measure $\rho_{\alpha}$.

\begin{proposition} \label{pr-3finitesample} (Finite sample guarantee) Let $\rho_{\alpha}$ denote a  risk measure with the representation \eqref{eq:convex-kusuoka} satisfying
$ g'(1)\le c, \forall g\in\mathcal H_{\rho_\alpha} $ for some $c\in \R$, and $\mathbb{P}$ be a light-tailed distribution, i.e. satisfying
$
A:=\mathbb{E}^{\mathbb{P}}\left[\exp \left(\|\xi\|^{a}\right)\right]=\int_{\Xi} \exp \left(\|\xi\|^{a}\right) \mathbb{P}(\mathrm{d} \xi)<\infty,
$
for some $a>1$.

Assume that $\widehat{J}_{N}$ and $\widehat{x}_{N}$ represent the optimal value and an optimizer of the distributionally robust program \eqref{eq:wcprob} with an ambiguity set $\mathbb{B}(\widehat{\mathbb{P}}_{N} )=\mathbb{B}^{\rho_{\alpha}}_{\varepsilon_N(\eta)}\left(\widehat{\mathbb{P}}_{N}\right)$, $N\in\N$, for some $\eta \in(0,1)$, where
   \begin{equation*}
\varepsilon_N(\eta)= 
 \begin{cases}
c\varepsilon_{0}^{1 /  m_2}, & \text { if }   \varepsilon_{0}\le c, \\
c\varepsilon_{0}^{1 / a}, & \text { if } \varepsilon_{0}> c,\end{cases} ~  ~~\varepsilon_{0}  = \frac{\log \left(c_{1} /\eta\right)}{c_{2} N},
\end{equation*}
for some constants $c_1,c_2$ only depending on   $a$, $A$, $m$ and $m_2=\max\{m,2\}$. Then, it holds the finite sample guarantee
\begin{equation}\label{eq-fsg}
\mathbb{P}^{N}\left\{\widehat{\Xi}_{N}: \mathbb{E}^{\mathbb{P}}\left[h\left(\widehat{x}_{N}, \xi\right)\right] \leq \widehat{J}_{N}\right\} \geq 1-\eta.
\end{equation}
\end{proposition}

\begin{proposition} \label{prop:3} (Asymptotic consistency) Under the condition of Proposition \ref{pr-3finitesample}, let $\varepsilon_{N}= \left( \frac{k_N }N\right)^{1/m_2}$, $ N \in \mathbb{N}$ where   $ k_N/N^\delta \to 0 $ and $ \log N/k_N\to 0$ as $N\to\infty$ for some $\delta<1$,  $m_2=\max\{m, 2\}$, and assume that $\widehat{J}_{N}$ and $\widehat{x}_{N}$ represent the optimal value and an optimizer of the distributionally robust program \eqref{eq:wcprob} with an ambiguity set $\mathbb{B}(\widehat{\mathbb{P}}_{N} )=\mathbb{B}^{\rho_{\alpha}}_{\varepsilon_N}\left(\widehat{\mathbb{P}}_{N}\right)$, $N\in\N$.
\begin{itemize}
\item [(i)] If $h(x, \xi)$ is upper semicontinuous in $\xi$ and there exists $L \geq 0$ with $|h(x$, $\xi)| \leq$ $L(1+\|\xi\|)$ for all $x \in \mathbb{X}$ and $\xi \in \Xi$, then $\mathbb{P}^{\infty}$-almost surely we have $\widehat{J}_{N} \downarrow J^{\star}$ as $N \rightarrow \infty$ where $J^{\star}$ is the optimal value of (\ref{eq-220324-2}).
\item [(ii)]
 If the assumptions of assertion (i) hold, $\mathbb{X}$ is closed, and $h(x, \xi)$ is lower semicontinuous in $x$ for every $\xi \in \Xi$, then any accumulation point of $\left\{\widehat{x}_{N}\right\}_{N \in \mathbb{N}}$ is $\mathbb{P}^{\infty}$-almost surely an optimal solution for (\ref{eq-220324-2}).
 \end{itemize}
\end{proposition}

These two guarantees are qualitatively identical to those of the Wasserstein DRO (c.f. \cite{EK18}), except further parametrized by the exact specification of the risk measure $\rho_{\alpha}$.  Both guarantees require only that any distortion function $g$ that may be invoked by (the dual representation of) the risk measure $\rho_{\alpha}$ has a bounded density at the worst-case value, i.e. the condition $g'(1)\le c, \forall g\in\mathcal H_{\rho_\alpha}$.  One can observe by Proposition \ref{prop:3insec2} that, somewhat interestingly, any coherent Wasserstein metric $d_{\rho_{\alpha}}$ chosen from a {\bf DD-DRW} family $\{d_{\rho_{\alpha}}\}_{\alpha \in A}$ would naturally meet this requirement. {\bf DD-DRW} families of metrics are thus of rather convenient choices.


Lastly, with regard to the tractability of GW-DRO, we will demonstrate in Sections \ref{sec:4-concave} and \ref{sec:4convex} that the minimax problem \eqref{eq:wcprob} can often be solved as finite-dimensional convex programs for many loss functions $h(x,\xi)$ arising from practical applications. In particular, we consider cases where the loss function $h(x,\xi)$ is either a general concave or convex function in $\xi$, and as motivating examples we present first in the next section a number of applications.

\subsection{Illustrative examples} \label{illexample}
The first application is the classical two-stage planning problem with recourse decisions, which is particularly common in operations management contexts. The loss function $h(x,\xi)$ in this application could either be concave or convex in $\xi$, depending on the exact setting of the second-stage problem.

\subsubsection*{Example (i): Two-stage problems with recourse}
Let $x_0$ denote the first stage, or ``here-and-now", decision that needs to be made before the realization of a random vector $\xi \sim \mathbb{P}$. In the case where the distribution $\mathbb{P}$ is unknown, the following two-stage distributionally robust linear program naturally arises
$$ \min_{x_0 \in \mathbb{X}} c^\top x_0 + \sup _{\mathbb{P} \in \mathbb{B}(\widehat{\mathbb{P}}_{N} )} \mathbb{E}^{\mathbb{P}}[Q(x_0,\xi)], $$
where $Q(x_0,\xi)$ is the recourse function capturing the optimal value of the recourse problem. It can take either the formulation of
$$ Q(x_0,\xi) = \min_{x_1} \left\{x_1^\top \bar{Q} \xi \;|\; Tx_0+Wx_1 \geq h \right\}, $$
where the objective of the recourse problem is uncertain due to $\xi$, or the formulation of
$$ Q(x_0,\xi) = \min_{x_1} \left\{q^\top x_1 \;|\; Tx_0 + Wx_1 \geq h+H\xi  \right\}, $$
where the ``right-hand-side" of the recourse problem is uncertain. Applications of these two cases can be found, for instance, in \cite{BDNT10}. Clearly, the first case corresponds to a loss function $h(x,\xi)$ in \eqref{eq-220324-2} that is concave in $\xi$, whereas the second case corresponds to a loss function $h(x,\xi)$ that is convex in $\xi$.

The next application is a problem of fundamental interest in finance.

\subsubsection*{Example (ii): Portfolio optimization} Let $\xi \sim \mathbb{P}$ denote a random vector of returns from $n$ different financial assets. The problem of robust portfolio optimization is a widely studied topic (see, e.g. \cite{DY10}), where a portfolio vector $x \in \mathbb{R}^n$ needs to be sought that maximizes the worst-case utility  subject to investment constraints captured by $\mathbb{X}$
$$ \max_{x \in \mathbb{X}} \inf _{\mathbb{P} \in \mathbb{B}(\widehat{\mathbb{P}}_{N} )}   \mathbb{E}^{\mathbb{P}}[u( \xi^\top x)]. $$
The utility function $u:\mathbb{R}\rightarrow \mathbb{R}$ is non-decreasing and is often assumed to be concave so to capture the risk-aversion attitude of an investor. This corresponds to a loss function $h(x,\xi):= - u(\xi^\top x)$ in \eqref{eq-220324-2} that is convex in $\xi$.

The last application is motivated by the recent surge of interest in statistical learning.

\subsubsection*{Example (iii): Machine learning} In supervised learning, a random vector $\xi:=(\xi^x,\xi^y) \sim \mathbb{P}$ represents an input-output pair and the goal is to seek a predictor (function) $f(\xi^x;\beta)$ parameterized by $\beta$ that best maps a given input value $\xi^x$ to a predicted output value. The issue of sampling errors, i.e. the uncertainty of $\mathbb{P}$, has motivated the recent study of the following distributionally robust statistical learning problem
$$ \min_{\beta \in {\cal B}} \sup _{\mathbb{P} \in \mathbb{B}(\widehat{\mathbb{P}}_{N} )}
\mathbb{E}^\mathbb{P}[ \hat{\ell}(f(\xi^x;\beta),\xi^y)], $$
where $\hat{\ell}$ is a function capturing losses incurred from prediction errors. We make the common assumption of a linear predictor, i.e. $f(\xi^x;\beta) = \beta^\top \xi^x$ for some $\beta \in {\cal B}$. In the case of a regression problem, i.e.
$\xi^y \in \mathbb{R}$, one can set
$$ \hat{\ell}(f(\xi^x;\beta),\xi^y) := \ell(\beta^\top \xi^x - \xi^y),\;\;\text{for some } \ell:\mathbb{R}\rightarrow \mathbb{R}_+$$
and
$ \| \xi_1-\xi_2 \| : =\| (\xi_1^x,\xi_1^y) - (\xi_2^x,\xi_2^y) \|$ (in \eqref{eq:distance}),
whereas in the case of a classification problem, i.e, $\xi^y \in \{-1,+1\}$, one can set
$$  \hat{\ell}(f(\xi^x;\beta),\xi^y) := \ell(\xi^y \cdot \beta^\top \xi^x),\;\;\text{for some non-increasing } \ell:\mathbb{R}\rightarrow \mathbb{R}_+ ,$$
and $ \| \xi_1-\xi_2 \| : = \| \xi_1^x - \xi_2^x\| + \mathbb{I}(\xi_1^y-\xi_2^y)$ (in \eqref{eq:distance}) where $\mathbb{I}(s) = 0$ if $s=0$ and $\mathbb{I}(s)=\infty$ otherwise. The function $\ell$ chosen in most machine learning methods is a convex function (see e.g. \cite{SKE19}). The case of regression would thus correspond to a loss function $h(x,\xi)$ in \eqref{eq-220324-2} that is convex in $\xi$. In the case of classification, since the chosen norm $\| \xi_1-\xi_2 \|$ assumes the cost of perturbing an output is infinitely large, any distribution $\mathbb{P} \in \mathbb{B}(\widehat{\mathbb{P}}_{N})$ would differ from the empirical distribution $\widehat{\mathbb{P}}_{N}$ only along the input space. We thus need only the observation that the loss function is convex in input variable $\xi^x$.



\section{Solving GW-DRO with concave loss functions}\label{sec:4-concave}
To study the tractability of solving GW-DRO, we focus first on the inner maximization problem of \eqref{eq:wcprob} -- the worst-case expectation problem. For notational convenience, we suppress the decision variable $x$ in $ \sup _{\mathbb{P} \in \mathbb{B}_{\rho,\varepsilon} (\widehat{\mathbb{P}}_{N} )} \mathbb{E}^{\mathbb{P}}[h(x, \xi)]$ and write the problem as
\begin{align} \label{eq:main}
\sup_{\mathbb{P} \in \mathcal{M}(\Xi)}~~ &\E^{\mathbb P}[\ell(\xi)]\\
{\rm subject~ to}~~& d_{\rho} \left(\widehat{\mathbb{P}}_{N} , \mathbb{P}\right)\leq \epsilon.\nonumber
\end{align}
Throughout this section, we consider the case where the loss function $\ell$ is concave in $\xi$. Following the definition of coherent Wasserstein metrics, the above problem can be stated also in terms of the joint distributions $\Pi$
\begin{align} \label{eq:main}
\sup_{\Pi }~~ &\E^{\Pi}[\ell(\xi)]\\
{\rm subject~ to}~~& \rho^{\Pi}(\|\widehat{\xi}-\xi\|) \leq \epsilon, \nonumber \\
                          ~~&\Pi \in \mathcal M(\Xi^2) \text { is a joint distribution of } \widehat{\xi} \text { and } \xi \nonumber\\
                          ~~& \text { with marginals } \widehat{\mathbb{P}}_{N} \text { and } \mathbb{P}, \text { respectively}. \nonumber
\end{align}
The above problem is distinctly different from and structurally more involved than the worst-case expectation problem from Wasserstein DRO in that it is a nonlinear optimization problem over the variable $\Pi$ due to the nonlinearity of the function
$\rho^\Pi$.
An even more fundamental challenge of solving the problem \eqref{eq:main} lies in the following observation.

\begin{proposition}\label{prop:6-nonconvex}
The feasible set of $\Pi$ in \eqref{eq:main} may be nonconvex, and thus the worst-case expectation problem \eqref{eq:main} is a nonconvex optimization problem in general.
\end{proposition}
\begin{proof}
We show this by considering a representative class of metrics, CVaR-Wasserstein metrics.  For $\rho={\rm CVaR}_{\alpha}$ and $(1-\alpha)>1/N$, let  
$$
\Pi_1= \frac1N\sum_{i=1}^N \delta_{(\widehat{\xi}_i,\widehat{\xi}_i+\epsilon e)}~~~{\rm and}~~~\Pi_2=\frac1N\sum_{i=2}^{N} \delta_{(\widehat{\xi}_i,\widehat{\xi}_i)} +~\frac1N \delta_{\left(\widehat{\xi}_1,\widehat{\xi}_1+ N\epsilon e (1-\alpha)\right) },
$$
where $e\in\R^m$ satisfies $\|e\|=1$.  One can verify that ${\rm CVaR}_{\alpha}^{\Pi_{i}}(\|\widehat{\xi}-\xi\|)=\epsilon$, for $i=1,2$, that is, $\Pi_1$ and $\Pi_2$ are feasible solutions of the problem \eqref{eq:main}.
Denote by $\Pi_{\lambda} = (1-\lambda )\Pi_1 +\lambda\Pi_2$ for $\lambda\in (0,1)$. We have
$$
{\rm CVaR}_{\alpha}^{\Pi_{\lambda}}(\|\widehat{\xi}-\xi\|) =  \epsilon \frac{(1-\alpha)\lambda  + [1-\alpha- \lambda /N]  }{1-\alpha} = \epsilon   + \epsilon\lambda\left( 1-\frac{1}{N(1-\alpha)}  \right) >\epsilon,
$$
which means $\Pi_\lambda$ is not a feasible solution of the problem \eqref{eq:main}. We therefore conclude that  problem \eqref{eq:main} is a nonconvex optimization problem in general.
\end{proof}

The problem \eqref{eq:main} is thus an {\it infinite-dimensional nonconvex optimization problem} in general that is not amenable to convex analysis.  The non-convexity of \eqref{eq:main}, as highlighted in the above proposition, naturally arises from coherent risk measures $\rho^{\Pi}$ that are concave in distributions. These risk measures, such as CVaR and its extensions, could assign a higher risk value to a mixture distribution, i.e. a convex combination of two arbitrary distributions, reflecting the uncertainty resulting from the mixture. Because of the nonconvexity, the common strategy of analyzing the tractability of worst-case expectation problems, which starts from studying first their dual problems, is no longer applicable to \eqref{eq:main}.
In this paper, we show how to bypass this difficulty by analyzing the worst-case expectation problem \eqref{eq:main} first from a primal perspective, i.e. solving the problem \eqref{eq:main} directly. Our analysis reveals that somewhat surprisingly, the problem  \eqref{eq:main} in its most general form can always be reduced to a structurally simple finite-dimensional convex optimization problem.

\begin{theorem} \label{thm1}
If $\ell$ is concave, then the worst-case expectation problem \eqref{eq:main} is equivalent to
\begin{align} \label{eq:main-eq1-7}
\sup_{y_1,\ldots,y_N\in\R^m}~~ &  \frac1N\sum_{i=1}^N \ell(y_i) \\
{\rm subject~ to}~~&  \rho^{\Pi}(\|\widehat{\xi}-\xi\|)\leq \epsilon,~~  \Pi((\widehat{\xi},\xi)=(\widehat{\xi}_i,y_i))=\frac1N,~y_i\in\Xi,~i=1,\ldots,N.\nonumber
\end{align}
\end{theorem}

That is, the worst-case distribution to the problem \eqref{eq:main-eq1-7} always takes the form of
$$\Pi = \frac1N\sum_{i=1}^N \delta_{ (\widehat{\xi}_i,y_i^*) },~~~{\rm i.e.}~~~\mathbb P = \frac1N\sum_{i=1}^N \delta_{ y_i^* }, $$
where $y_i^*$, $i=1,...,N$, is the optimal solution to the problem \eqref{eq:main-eq1-7}. The theorem is general in that it does not rely on any assumption of the function form of $\rho$, other than the general property of coherent risk measures. The feasible set of $y_i$, $i=1,...,N$, is clearly a convex set, 
since  for any two feasible solutions $y_i^{(1)},  y_i^{(2)} \in \Xi$, $i=1,...,N$, and their convex combination $y_i^{(3)}=\lambda y_i^{(1)} + (1-\lambda)y_i^{(2)}$, $i=1,\ldots,N$, $\lambda\in [0,1],$ we have
\begin{align*}
 \rho(\|\widehat{\xi}-\xi_{y^{(3)}}\|) 
 &\le   \rho(\lambda\|\widehat{\xi}-  \xi_{y^{(1)}}\| +(1-\lambda)\|\widehat{\xi}-\xi_{y^{(2)}}\|) \le \lambda  \rho (\|\widehat{\xi}-\xi_{y^{(1)}}\|) + (1-\lambda)\rho (\|\widehat{\xi}-\xi_{y^{(2)}}\|)\le \epsilon,
\end{align*}
where  $\xi_{y}$ denotes a random variable such that $(\widehat{\xi},\xi_{y})$ has the distribution $\Pi_y =\frac1N\sum_{i=1}^N \delta_{(\widehat{\xi}_i,y_i)}$.
The theorem lays an important basis for studying the tractability of GW-DRO in general and can be extended further, as shown in the next section, to even more general class of loss functions $\ell$.

Theorem \ref{thm1} can be further exploited to identify the dual of the worst-case expectation problem \eqref{eq:main}. Namely, the convexity of the reduced problem \eqref{eq:main-eq1-7} renders it now amenable to analysis using convex duality. We show in the following how the dual problem can be obtained also as a finite-dimensional convex minimization problem.  This alternative formulation can be conveniently integrated with the outer minimization problem in \eqref{eq:wcprob} so to solve the overall problem of GW-DRO  \eqref{eq:wcprob} as a single convex minimization problem.
\begin{corollary} \label{dual_form}
If $\ell$ is concave, then the problem \eqref{eq:main} is equivalent to
\begin{align}
\inf_{\lambda, p, s, z_i, v_i} ~~& \lambda \epsilon + \frac{1}{N}\sum_{i=1}^N s_i & \nonumber \\
{\rm subject\;to} ~~&  [-\ell]^* (z_i - \nu_i) + \sigma_{\Xi}(\nu_i) - z_i^\top \widehat{\xi}_i \leq s_i, &  i=1,...,N,  \label{da1}\\
 &  \| z_i \|_{*} \leq p_i,     &i=1,...,N, \label{da2}\\
 &  \sum_{i=1}^N p_i = \lambda, & \label{da3} \\
 &  \frac{p}{\lambda}\in A_\rho, & \nonumber
\end{align}
where $\lambda \in \mathbb{R}, p\in \mathbb{R}^N, s \in \mathbb{R}^N, z_i\in \mathbb{R}^m, v_i\in \mathbb{R}^m$, $\sigma_{\Xi}$ is the support function of $\Xi$,
$A_\rho$ is a subset of a probability simplex, defined by
$$ A_\rho = \left\{y \in \R^N: \exists Z \sim \frac{1}{N}\sum_{i=1}^N \delta_{y_i},\; Z \in {\cal Z}_{\rho}   \right\},$$
and ${\cal Z}_{\rho}$ denotes the risk envelope of $\rho$, i.e.
$ \mathcal Z_\rho =\{ Z\ge 0: \E[Z]=1, \E[ZX]\le \rho(X)~{\rm for ~all}~X\}$. It is defined that $\frac{p}{0} \notin  A_{\rho}$ for any $p \neq {\bf 0}$ and $\frac{\bf{0}}{0} \in A_{\rho}$.
\end{corollary}

We now demonstrate how Corollary \ref{dual_form} can be applied to solve GW-DRO problems when the ambiguity set $\mathbb{B}(\widehat{\mathbb{P}}_{N})$ in \eqref{eq:wcprob} takes either the form of CVaR-Wasserstein ambiguity ball $\mathbb{B}_{(1),\varepsilon}^{\alpha}\left(\widehat{\mathbb{P}}_{N}\right)$ or
expectile-Wasserstein ball $\mathbb{B}_{(2),\varepsilon}^{\alpha}\left(\widehat{\mathbb{P}}_{N}\right)$.

\begin{example} (CVaR-Wasserstein) Note that ${\rm CVaR}_\alpha$ can be represented as (Theorem 4.52 of \cite{FS16})
\[
{\rm CVaR}_\alpha(X)  = \sup_{ Z \in \mathcal Z_{\rm CVaR_\alpha}} \E [ZX]
~~~{\rm with}~~\mathcal Z_{\rm CVaR_\alpha}=\left\{Z\ge 0:\E[Z]=1, Z\le \frac1{1-\alpha}\right\}.
\]
Hence, we have   $A_{\rm CVaR_\alpha}=\{y\in\R^N_+:\sum_{i=1}^Ny_i=1,  y_i\le 1/(1-\alpha)\}$.
Thus, following Corollary \ref{dual_form}, the problem \eqref{eq:wcprob} with $\mathbb{B}(\widehat{\mathbb{P}}_{N}) = \mathbb{B}_{(1),\varepsilon}^{\alpha}\left(\widehat{\mathbb{P}}_{N}\right)$ can be solved by
\begin{align*}
\inf_{\lambda, p, s, z_i, v_i}~~ &  \lambda \epsilon + \frac{1}{N}\sum_{i=1}^N s_i \\
{\rm subject~ to}~~& \eqref{da1}- \eqref{da3},\; \;\; p_i \leq \frac{\lambda}{1-\alpha},\;\; i=1,...,N.
\end{align*}
\end{example}
\begin{example} (Expectile-Wasserstein)
 Note that expectile $e_\alpha$  can be represented as    (Proposition 8 of \cite{BKMG14})
\[
e_\alpha(X) = \sup_{Z\in \mathcal Z_{e_\alpha}}\E[XZ]~~~{\rm with}~~\mathcal Z_{  e_\alpha}=\left\{Z\ge 0:\E[Z]=1, \frac{\esssup Z}{\essinf Z}\le \frac\alpha{1-\alpha}\right\}.
\]
%
%
Hence, we have  $A_{e_\alpha}=\{y\in\R^N_+:\sum_{i=1}^Ny_i=1,  \max y/\min y\le \alpha/(1-\alpha)\}$. Thus, following Corollary \ref{dual_form}, the problem \eqref{eq:wcprob} with $\mathbb{B}(\widehat{\mathbb{P}}_{N}) = \mathbb{B}_{(2),\varepsilon}^{\alpha}\left(\widehat{\mathbb{P}}_{N}\right)$ can be solved by
\begin{align*}
\inf_{\lambda, p, s, z_i, v_i}~~ &  \lambda \epsilon + \frac{1}{N}\sum_{i=1}^N s_i,  \\
{\rm subject~ to}~~& \eqref{da1}- \eqref{da3},\; \; \; p_i \leq \frac{\alpha}{1-\alpha} p_j,\;\; i, j=1,...,N. \\
\end{align*}
\end{example}

\section{Solving GW-DRO with convex loss functions}\label{sec:4convex}
Similar to the analysis presented in the previous section, we study the case of convex loss functions by investigating first how the worst-case expectation problem may be solved from a primal perspective. We show below that under the piecewise linear assumption, the worst-case expectation problem can also be reduced to a finite-dimensional optimization problem for any ambiguity set defined based on coherent Wasserstein metrics.
\begin{theorem}\label{th:main-2}
In the case where the loss function $\ell$ is convex piecewise linear, i.e. $\ell=\max_{k=1,\ldots,K} \ell_k$, where $\ell_k$, $k=1,\ldots,K$, are linear loss functions, the worst-case expectation problem
\eqref{eq:main} is equivalent to
\begin{align} \label{eq:gene-2-31}
\sup_{p_{ij},\xi_{ij}}~~ &\frac1N\sum_{i=1}^N\sum_{j=1}^K p_{ij} \ell_j(\xi_{ij}) \\
{\rm subject~ to}~~& \rho^{\Pi}(\|\widehat{\xi}-\xi\|) \le \epsilon,\label{keyeq}\\
& \Pi((\widehat{\xi},\xi)= (\widehat{\xi}_i,\xi_{ij}))=p_{ij}\ge 0, ~\xi_{ij}\in \Xi,~\forall~ i,j,\nonumber\\
&  \sum_{j=1}^Kp_{ij} =1,~~i=1,\ldots,N.\nonumber
\end{align}
\end{theorem}

The above result reveals the feasibility of solving the problem also as a finite-dimensional optimization problem for any ambiguity sets defined based on coherent Wasserstein metrics. The optimization problem
\eqref{eq:gene-2-31}, as it stands, differs from the finite-dimensional optimization problem \eqref{eq:main-eq1-7} presented in the previous section in that it further requires determining the probability $p_{ij}$ for each support $\xi_{ij}$, potentially rendering the problem \eqref{eq:gene-2-31} nonconvex. The tractability of the optimization problem \eqref{eq:gene-2-31} now depends more heavily on the exact specification of $\rho$ and needs to be studied on a case-by-case basis. In the remainder of this section, we will focus on studying the worst-case expectation problem \eqref{eq:main} in greater detail for CVaR-Wasserstien ambiguity sets $\mathbb{B}_{(1),\varepsilon}^{\alpha}\left(\widehat{\mathbb{P}}_{N}\right)$ and expectile-Wasserstein ambiguity sets $\mathbb{B}_{(2),\varepsilon}^{\alpha}\left(\widehat{\mathbb{P}}_{N}\right)$.

\subsection{CVaR-Wasserstien ambiguity sets}
By the well-known representation of CVaR (\cite{RU02}), ${\rm CVaR}_\alpha (X) =\inf_{t} \{t+\frac1{1-\alpha} \E[(X-t)_+]$, the problem \eqref{eq:gene-2-31} can be explicitly written as
\begin{align}\label{eq:gene-4}
\sup_{t,p_{ij},\xi_{ij}}~~ &\frac1N\sum_{i=1}^N\sum_{j=1}^K p_{ij} \ell_j(\xi_{ij}) \\
{\rm subject~ to}~~& t+\frac1{1-\alpha}\frac1N\sum_{i=1}^N\sum_{j=1}^K p_{ij} (\|\xi_{ij} -\widehat{\xi}_i\|-t)_+\leq \epsilon,\nonumber\\
& \sum_{j=1}^Kp_{ij} =1,~~i=1,\ldots,N,~~p_{ij}\ge 0,  ~\xi_{ij}\in \Xi,~\forall~ i,j.\nonumber
\end{align}
The above problem is complicated by the need to handle the $\alpha$-quantile variable $t$, a source of nonconvexity to the problem. Despite this nonconvexity, we show in this section that the above problem admits a more tractable reformulation in the case where the support set  $\Xi = \mathbb{R}^m$. The reformulation not only enables us to demonstrate the tractability of the worst-case expectation problem $\sup _{\mathbb{P} \in \mathbb{B}_{(1),\varepsilon}^{\alpha} (\widehat{\mathbb{P}}_{N} )} \mathbb{E}^{\mathbb{P}}[\ell(\xi)]$ for any Lipschitz continuous convex function $\ell$, but also reveals a deep connection between
the worst-case expectation problem $\sup _{\mathbb{P} \in \mathbb{B}_{(1),\varepsilon}^{\alpha} (\widehat{\mathbb{P}}_{N} )} \mathbb{E}^{\mathbb{P}}[\ell(\xi)]$ and the worst-case expectation problem formulated based on Wasserstein ambiguity sets, i.e. $\sup _{\mathbb{P} \in \mathbb{B}^{\rm W}_{\varepsilon} (\widehat{\mathbb{P}}_{N} )} \mathbb{E}^{\mathbb{P}}[\ell(\xi)]$.


\begin{theorem} \label{main_cvx}
In the case where the loss function $\ell$ is a convex function satisfying $$L:=\max_{x\in\R^m}\|\partial \ell(x)\|_*<\infty$$ and $\Xi = \mathbb{R}^m$, where $\|\cdot\|_*$ is the dual norm of $\|\cdot\|$, the worst-case expectation problem
\begin{equation} \label{wcvar}
\sup _{\mathbb{P} \in \mathbb{B}_{(1),\varepsilon}^{\alpha} (\widehat{\mathbb{P}}_{N} )} \mathbb{E}^{\mathbb{P}}[\ell(\xi)]
\end{equation}
is equivalent to
\begin{equation} \label{connection}
\max \left\{\sup _{\mathbb{P} \in \mathbb{B}_{\varepsilon}^{{\rm wc}} (\widehat{\mathbb{P}}_{N} )} \mathbb{E}^{\mathbb{P}}[\ell(\xi)], \sup _{\mathbb{P} \in \mathbb{B}^{\rm W}_{(1-\alpha)\varepsilon} (\widehat{\mathbb{P}}_{N} )} \mathbb{E}^{\mathbb{P}}[\ell(\xi)] \right\},
\end{equation}
where
\begin{equation} \label{worst_exp}
\sup _{\mathbb{P} \in \mathbb{B}_{\varepsilon}^{{\rm wc}} (\widehat{\mathbb{P}}_{N} )} \mathbb{E}^{\mathbb{P}}[\ell(\xi)] = \frac1N\sum_{i=1}^N  \max_{e: \|e\|=1} \ell(\widehat{\xi}_i+\epsilon e),
\end{equation}
and
\begin{equation} \label{wass_exp}
\sup _{\mathbb{P} \in \mathbb{B}^{\rm W}_{(1-\alpha)\varepsilon} (\widehat{\mathbb{P}}_{N} )} \mathbb{E}^{\mathbb{P}}[\ell(\xi)]
 = \frac1N\sum_{i=1}^N \ell(\widehat{\xi}_i) +  L(1-\alpha)\epsilon.
\end{equation}
\end{theorem}

The above result is striking for several reasons. First, it reveals via \eqref{connection} an elegant connection between the worst-case expectation problem \eqref{wcvar} and two other worst-case expectation problems, one formulated based on the worst-case ambiguity set $\mathbb{B}_{\varepsilon}^{{\rm wc}}\left(\widehat{\mathbb{P}}_{N}\right)$ and the other formulated based on the classical Wasserstein ambiguity set $\mathbb{B}^{\rm W}_{(1-\alpha)\varepsilon} (\widehat{\mathbb{P}}_{N} )$ with the radius $\epsilon$ scaled by $1-\alpha$. These two latter problems can be solved respectively by a structurally simpler maximization problem \eqref{worst_exp} and in closed-form \eqref{wass_exp}. Second, \eqref{connection} is surprising because it implies that there is no loss of generality to reduce the set $\mathbb{B}_{(1),\varepsilon}^{\alpha} (\widehat{\mathbb{P}}_{N} )$ to the set
\begin{equation} \label{reducedset}
\overline{\mathbb{B}}_{\epsilon}^{\alpha}:=\mathbb{B}^{\rm W}_{(1-\alpha)\varepsilon}  (\widehat{\mathbb{P}}_{N} )\cup  \mathbb{B}_{\varepsilon}^{{\rm wc}} (\widehat{\mathbb{P}}_{N} )
\end{equation}
for solving the worst-case expectation problem. The latter is, however, a considerably smaller set of the former. To see this, note that the following set inclusion relationships follow straightforwardly the fact that $\mathbb{E}[\xi] \leq {\rm CVaR}_{\alpha}[\xi] \leq \frac{1}{(1-\alpha)}\mathbb{E}[\xi]$ for any nonnegative random variable $\xi$ and $\alpha \in (0,1)$.
\begin{equation} \label{setinclusion}
\mathbb{B}^{\rm W}_{(1-\alpha)\varepsilon} (\widehat{\mathbb{P}}_{N} )\subset \mathbb{B}_{(1),\varepsilon}^{\alpha} (\widehat{\mathbb{P}}_{N} ) \subset \mathbb{B}^{\rm W}_{\varepsilon}(\widehat{\mathbb{P}}_{N} ).
\end{equation}
As $\alpha$ increases, the set $\mathbb{B}^{\rm W}_{(1-\alpha)\varepsilon} (\widehat{\mathbb{P}}_{N})$ converges towards $\widehat{\mathbb{P}}_{N}$ and thus is significantly smaller than $\mathbb{B}_{(1),\varepsilon}^{\alpha} (\widehat{\mathbb{P}}_{N} )$.
The inclusion relationship $\mathbb{B}_{\varepsilon}^{{\rm wc}} (\widehat{\mathbb{P}}_{N}) \subset  \mathbb{B}^{\alpha}_{(1),\varepsilon}(\widehat{\mathbb{P}}_{N} ) $ trivially holds, and the set $\mathbb{B}_{\varepsilon}^{{\rm wc}} (\widehat{\mathbb{P}}_{N})$ is considerably smaller than $\mathbb{B}^{\alpha}_{(1),\varepsilon}(\widehat{\mathbb{P}}_{N} )$ in that $\mathbb{B}_{\varepsilon}^{{\rm wc}} (\widehat{\mathbb{P}}_{N})$ contains only distributions whose support is uniformly bounded from the support of $\widehat{\mathbb{P}}_{N}$ by $\epsilon$. Figure \ref{setrelation} demonstrates these inclusion relationships and highlights the considerable reduction from $\mathbb{B}^{\alpha}_{(1),\varepsilon}(\widehat{\mathbb{P}}_{N} )$ to $\mathbb{B}^{\rm W}_{(1-\alpha)\varepsilon}  (\widehat{\mathbb{P}}_{N} )\cup  \mathbb{B}_{\varepsilon}^{{\rm wc}} (\widehat{\mathbb{P}}_{N} )$.


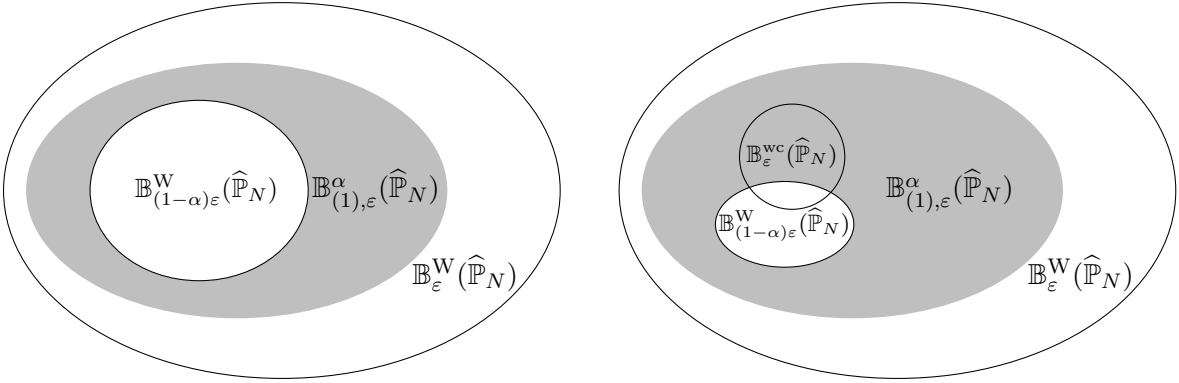
\begin{figure}[htbp]
\begin{center}
\begin{tikzpicture}
\fill[gray!50] (0.4,0)  circle (2.8 and 1.7);
\draw (1,0)  circle (3.7 and 2.5);
\fill[white] (-0.1,0)  circle (1.45 and 1.2);
\draw (-0.1,0)  circle (1.45 and 1.2);
 \node at (3.44, -1.1) {$\mathbb{B}^{\rm W}_{\varepsilon}(\widehat{\mathbb{P}}_{N} )$};
        \node at (2.25,0) {$\mathbb{B}^\alpha_{(1),\varepsilon}(\widehat{\mathbb{P}}_{N} )$};
                \node at (0,0) {\small{$\mathbb{B}^{\rm W}_{(1-\alpha)\varepsilon} (\widehat{\mathbb{P}}_{N} )$}};
\end{tikzpicture}
~~~~
\begin{tikzpicture}
\fill[gray!50] (0.4,0)  circle (2.8 and 1.7);
\draw (1,0)  circle (3.7 and 2.5);
\fill[white] (-0.5,-0.45)  circle (0.92 and 0.57);
\draw (-0.5,-0.45)  circle (0.92 and 0.57);
\draw (-0.4,0.45)  circle (0.7);
 \node at (3.44, -1.1) {$\mathbb{B}^{\rm W}_{\varepsilon}(\widehat{\mathbb{P}}_{N} )$};
        \node at (1.7,0) {$\mathbb{B}^\alpha_{(1),\varepsilon}(\widehat{\mathbb{P}}_{N} )$};
                \node at (-0.5,-0.45) {\footnotesize{$\mathbb{B}^{\rm W}_{(1-\alpha)\varepsilon} (\widehat{\mathbb{P}}_{N} )$}};
                \node at (-0.4,0.5) {\footnotesize{$\mathbb{B}_{\varepsilon}^{{\rm wc}} (\widehat{\mathbb{P}}_{N})$}};
\end{tikzpicture}
\caption{Relationships among the ambiguity sets}
\label{setrelation}
\end{center}
\end{figure}

The figure and the set inclusion relationships \eqref{setinclusion} also demonstrate that while the worst-case expectation problem
$\sup _{\mathbb{P} \in \mathbb{B}^{\alpha}_{(1),\varepsilon}(\widehat{\mathbb{P}}_{N} )} \mathbb{E}^{\mathbb{P}}[\ell(\xi)]$ can be bounded above and below respectively by  $\sup _{\mathbb{P} \in \mathbb{B}^{\rm W}_{\varepsilon} (\widehat{\mathbb{P}}_{N} )} \mathbb{E}^{\mathbb{P}}[\ell(\xi)]$ and $\sup _{\mathbb{P} \in \mathbb{B}^{\rm W}_{(1-\alpha)\varepsilon} (\widehat{\mathbb{P}}_{N} )} \mathbb{E}^{\mathbb{P}}[\ell(\xi)]$, neither of the two can be used as a reasonable proxy to $\sup _{\mathbb{P} \in \mathbb{B}^{\alpha}_{(1),\varepsilon}(\widehat{\mathbb{P}}_{N} )} \mathbb{E}^{\mathbb{P}}[\ell(\xi)]$. The problem $\sup _{\mathbb{P} \in \mathbb{B}^{\rm W}_{\varepsilon} (\widehat{\mathbb{P}}_{N} )} \mathbb{E}^{\mathbb{P}}[\ell(\xi)]$ is overly-conservative, i.e. less data-driven, whereas the problem $\sup _{\mathbb{P} \in \mathbb{B}^{\rm W}_{(1-\alpha)\varepsilon} (\widehat{\mathbb{P}}_{N} )} \mathbb{E}^{\mathbb{P}}[\ell(\xi)]$ accounts for too little, or almost none, uncertainty, i.e. non-robust, as $\alpha \rightarrow 1$. Taking this perspective, we can further see how \eqref{connection} sheds light on the underlying mechanism of the worst-case expectation problem $\sup _{\mathbb{P} \in \mathbb{B}^{\alpha}_{(1),\varepsilon} (\widehat{\mathbb{P}}_{N} )} \mathbb{E}^{\mathbb{P}}[\ell(\xi)]$ to offer data-driven evaluation of expected cost while maintaining some guaranteed level of robustness. Specifically, as $\alpha \rightarrow 1$, i.e. increasingly data-driven, the set $\mathbb{B}^{\rm W}_{(1-\alpha)\varepsilon} (\widehat{\mathbb{P}}_{N} )$ in the reduced set  $\overline{\mathbb{B}}_{\epsilon}^{\alpha}$  (in \eqref{reducedset}) would shrink towards $\widehat{\mathbb{P}}_{N}$, while at the same time the worst-case ambiguity set $\mathbb{B}_{\varepsilon}^{{\rm wc}} (\widehat{\mathbb{P}}_{N})$ in $\overline{\mathbb{B}}_{\epsilon}^{\alpha}$ guarantees the {\it minimum} level of robustness by taking into account any distribution with support maximally deviating by $\epsilon$. The set $\mathbb{B}^{\rm W}_{(1-\alpha)\varepsilon} (\widehat{\mathbb{P}}_{N} )$ can turn to be a set providing additional robustness when $\alpha \rightarrow 0$, in which case
\begin{equation} \label{wost2}
\sup _{\mathbb{P} \in \mathbb{B}_{\varepsilon}^{{\rm wc}} (\widehat{\mathbb{P}}_{N} )} \mathbb{E}^{\mathbb{P}}[\ell(\xi)]  < \sup _{\mathbb{P} \in \mathbb{B}^{\rm W}_{(1-\alpha)\varepsilon} (\widehat{\mathbb{P}}_{N} )} \mathbb{E}^{\mathbb{P}}[\ell(\xi)].
\end{equation}

Third, this connection via \eqref{connection} implies the following intriguing observation of the worst-case distributions for the worst-case expectation problem \eqref{wcvar}.
\begin{corollary} \label{wcdis} Under the condition of Theorem \ref{main_cvx}, if $ \max_{\|e\|\le 1} \ell(\widehat{\xi}_i+\epsilon e) > \ell(\widehat{\xi}_i)$ for some $\widehat{\xi}_i$, then there always exists $\alpha \in (0,1)$ such that the worst-case expectation problem \eqref{wcvar} is attainable, i.e. the worst-case distribution exists.
\end{corollary}

This is in sharp contrast to the worst-case distributions for the worst-case expectation problem
$\sup _{\mathbb{P} \in \mathbb{B}^{\rm W}_{\varepsilon} (\widehat{\mathbb{P}}_{N} )} \mathbb{E}^{\mathbb{P}}[\ell(\xi)]$ formulated based on the Wasserstein ambiguity sets $\mathbb{B}^{\rm W}_{\varepsilon} (\widehat{\mathbb{P}}_{N} )$. Namely, as shown in Proposition \ref{thm:7} in the appendix, the worst-case distributions for the latter generally do not exist. The finding
 in Corollary \ref{wcdis} is difficult to identify directly from the property of the CVaR metric. It might be tempting to suppose that the worst-case distributions for \eqref{wcvar} may not exist for any $\alpha \in [0,1)$, given that the CVaR metric is the conditional analog of the expectation for any $\alpha <1$.

Finally, we demonstrate how Theorem \ref{main_cvx} can be applied to solve problems such as portfolio optimization and machine learning mentioned in Section \ref{illexample}. Let us first make the following observation.
\begin{corollary} \label{reduced1}
In the case where $\ell(\xi) = f(x^\top \xi)$ and $f$ is a Lipschitz continuous convex function in $\mathbb{R}$, the worst-case expectation problem \eqref{wcvar} can be solved by
$$
\min_{x \in \mathbb{X}}
\max \left\{
 \frac1N\sum_{i=1}^N f(x^\top \widehat{\xi}_i) +  {\rm Lip}(f)\|x\|_{*}(1-\alpha)\epsilon,
  \frac1N\sum_{i=1}^N \max\left\{ f_1 (x^\top \widehat{\xi}_i - \epsilon\|x\|_{*}), ~f_2 (x^\top  \widehat{\xi}_i +\epsilon\|x\|_{*} ) \right\}
\right\},
$$
where $f_1(t)= f(t)1_{\{t<t_0\}}+ f(t_0)1_{\{t\ge t_0\}} $ and $f_2(t) = f(t_0)1_{\{t\le t_0\}}+ f(t)1_{\{t> t_0\}} $ and $t_0\in [-\infty,\infty]$ is such that $f$ is decreasing on $(-\infty,t_0)$ and increasing on $(t_0,\infty)$.
\end{corollary}

\begin{example} \label{applications} {(continue Example (ii), Section \ref{illexample})}
Assuming the utility function $u$ is Lipschitz continuous, we can apply Corollary \ref{reduced1} to the robust portfolio optimization problem
$$ \min_{x \in \mathbb{X}} \sup _{\mathbb{P} \in \mathbb{B}_{(1),\varepsilon}^{\alpha}(\widehat{\mathbb{P}}_{N} )}   \mathbb{E}^{\mathbb{P}}[-u( \xi^\top x)]$$
by setting $f=-u$. We obtain the following convex program

$$
\min_{x \in \mathbb{X}}
\max \left\{
 \frac1N\sum_{i=1}^N -u(x^\top \widehat{\xi}_i) +  {\rm Lip}(u)\|x\|_{*}(1-\alpha)\epsilon,
  \frac1N\sum_{i=1}^N -u (x^\top  \widehat{\xi}_i -\epsilon\|x\|_{*} ) \right\}.
$$
\end{example}

\begin{example} {(continue Example (iii), Section \ref{illexample})} \label{ex-ml} Let us consider first solving the distributionally robust regression problem
$$ \min_{\beta \in {\cal B}} \sup _{\mathbb{P} \in \mathbb{B}_{(1),\varepsilon}^{\alpha}(\widehat{\mathbb{P}}_{N} )}
\mathbb{E}^\mathbb{P}[ \ell(\beta^\top \xi^x - \xi^y) ], $$
where $\ell:\mathbb{R}\rightarrow \mathbb{R}_+$ is convex Lipschitz continuous. The function $\ell$ in regression is generally symmetric with respect to the origin and $\ell(0)=0$. That is, $\ell(t) = \ell_1(t) + \ell_2(t) $, $\ell_1(t) = h(t_-)$ and $\ell_2(t) = h(t_+) $ for some non-decreasing convex function $h: \mathbb{R}_+ \rightarrow \mathbb{R}_+$, $h(0)=0$. Applying Corollary \ref{reduced1}, we arrive at the following convex program
\begin{align*}
\min_{\beta \in {\cal B}}\max & \left\{ \begin{array}{l}
\frac{1}{N}\sum_{i=1}^{N}\ell(\beta^{\top}\widehat{\xi}_{i}^{x}-\widehat{\xi}_{i}^{y})+{\rm Lip}(\ell)\| (\beta,-1)\| _{*}(1-\alpha)\varepsilon,\\
\frac{1}{N}\sum_{i=1}^{N}\max\left\{ \ell_{1}(\beta^{\top}\widehat{\xi}_{i}^{x}-\widehat{\xi}_{i}^{y}-\| (\beta,-1)\| _{*}\varepsilon),\ell_{2}(\beta^{\top}\widehat{\xi}_{i}^{x}-\widehat{\xi}_{i}^{y}+\| (\beta,-1)\| _{*}\varepsilon)\right\}
\end{array}\right\}.
\end{align*}

Next, consider solving the distributionally robust classification problem
$$ \min_{\beta \in {\cal B}} \sup _{\mathbb{P} \in \mathbb{B}_{(1),\varepsilon}^{\alpha}(\widehat{\mathbb{P}}_{N} )}
\mathbb{E}^\mathbb{P}[ \ell(\xi^y \cdot \beta^\top \xi^x) ], $$
where $\ell:\mathbb{R}\rightarrow \mathbb{R}_+$ is non-increasing convex Lipschitz continuous. Applying Corollary \ref{reduced1}, we arrive at the following convex program
\[
\min_{\beta \in {\cal B}}\max\left\{ \frac{1}{N}\sum_{i=1}^{N}\ell(\widehat{\xi}_{i}^{y}\cdot\beta^{\top}\widehat{\xi}_{i}^{x})+{\rm Lip}(\ell) \| \beta\| _{*}(1-\alpha)\varepsilon,\frac{1}{N}\sum_{i=1}^{N}\ell(\widehat{\xi}_{i}^{y}\cdot\beta^{\top}\widehat{\xi}_{i}^{x}-\| \beta\| _{*}\varepsilon)\right\}.
\]
A list of Lipschitz continuous functions $\ell$ in regression and classification can be found for examples in  \cite{SKE19}.
\end{example}

\paragraph{A regularization perspective}
 It is known that applying the type-1 Wasserstein DRO to statistical learning problems such as regression and classification is equivalent to solving a classical regularized empirical risk minimization problem (see e.g. \cite{SKE19}.), i.e.
(a):$\frac1N\sum_{i=1}^N f(x^\top \widehat{\xi}_i) +  {\rm Lip}(f)\|x\|_{*}(1-\alpha)\epsilon$ in Corollary \ref{reduced1}, where the regularization term ${\rm Lip}(f)\|x\|_{*}(1-\alpha)\epsilon$ controls the size of the decision variable $x$. One can see from Corollary \ref{reduced1} that applying GW-DRO to statistical learning problems essentially boils down to aggregating two different forms of regularized empirical problems, i.e. (a) and
(b): $\frac1N\sum_{i=1}^N \max\left\{ f_1 (x^\top  \widehat{\xi}_i - \epsilon\|x\|_{*}), ~f_2 (x^\top  \widehat{\xi}_i +\epsilon\|x\|_{*} ) \right\}$, and GW-DRO determines which regularized problem, (a) or (b), to apply according to which one is more conservative, i.e. the one giving a larger value. Clearly, whether the regularized problem (a) or (b) is more conservative would depend on the  exact value of the decision variable $x$, i.e. the choice of a regularized problem in GW-DRO is decision-dependent. One can observe that the value of (a) would tend to be larger than that of (b), when $\alpha \rightarrow 0$, since the regularization term ${\rm Lip}(f)\|x\|_{*}(1-\alpha)\epsilon$ in (a) would become more dominating. In other words, GW-DRO would behave more similarly as the classical regularized problem as $\alpha$ decreases. This perspective that GW-DRO could serve as an ensemble of different regularized problems appears to be of high novelty. In particular, the idea of aggregating regularized problems by taking the pointwise maximum may offer a means to address the general challenge of regularization scheme selection.

\subsection{Expectile-Wasserstein ambiguity sets}
One may expect that the worst-case expectation problem $\sup _{\mathbb{P} \in \mathbb{B}_{(2),\varepsilon}^{\alpha} (\widehat{\mathbb{P}}_{N} )} \mathbb{E}^{\mathbb{P}}[\ell(\xi)]$ defined over expectile-Wasserstein ambiguity sets is a more challenging, or less tractable, problem than the one studied in the previous section, i.e. $\sup _{\mathbb{P} \in \mathbb{B}_{(1),\varepsilon}^{\alpha} (\widehat{\mathbb{P}}_{N} )} \mathbb{E}^{\mathbb{P}}[\ell(\xi)]$, given the common belief that expectiles are more sophisticated forms of risk measures than CVaR. We show in this section that perhaps quite counter-intuitvely, the former is as tractable as, in fact generally more tractable than, the latter. In particular, in the case where $\rho = e_{\alpha}$, the finite-dimensional problem
 \eqref{eq:gene-2-31}, as shown below, always admits a convex reformulation for any convex support set $\Xi$.
The expectile-Wasserstein ambiguity sets $\mathbb{B}_{(2),\varepsilon}^{\alpha} (\widehat{\mathbb{P}}_{N} )$ could thus be a more appealing choice than the CVaR-Wasserstein ambiguity sets $\mathbb{B}_{(1),\varepsilon}^{\alpha} (\widehat{\mathbb{P}}_{N} )$ when the support set is constrained, i.e. $\Xi \neq \mathbb{R}^m$. We obtain the following two convex optimization reformulations, one in a maximization form and another in a minimization form.
\begin{theorem} \label{ex-thm1}
In the case where the loss function $\ell$ is piecewise linear, taking the form of $\ell(x) = \max_{j=1,...,K} \{a_j^\top x + b_j\}$, the worst-case expectation problem
\begin{equation} \label{wcexptile}
\sup _{\mathbb{P} \in \mathbb{B}_{(2),\varepsilon}^{\alpha} (\widehat{\mathbb{P}}_{N} )} \mathbb{E}^{\mathbb{P}}[\ell(\xi)]
\end{equation}
is equivalent to the following convex maximization problem
\begin{align}
\sup_{p_{ij}\ge0, y_{ij}\in\R^m}~~ &\frac1N\sum_{i=1}^N\sum_{j=1}^K \left(a_j^\top y_{ij} +(a_j^\top \widehat{\xi}_i+b_j) p_{ij} \right) \label{eq:expectile-2} \\
{\rm subject~ to}~~&   \frac1N \sum_{i=1}^{N} \sum_{j=1}^K(\|y_{ij}  \|-\varepsilon p_{ij})_+ +\frac{1-\alpha}{2\alpha - 1} \frac1N \sum_{i=1}^{N} \sum_{j=1}^K\|y_{ij}\|\le \frac{1-\alpha}{2\alpha - 1} \varepsilon ,\nonumber\\
& \sum_{j=1}^Kp_{ij} =1,~~i=1,\ldots,N,  ~ \widehat{\xi}_i+\frac{y_{ij}}{p_{ij}}  \in \Xi,~\forall i,j, \nonumber
\end{align}
where ${y_{ij}}/{p_{ij}}$ reads as $\mathbf{\infty}$ if ${y_{ij}}\neq {\bf 0}$, ${p_{ij}}=0$, and $\mathbf{0}$ if ${y_{ij}}={\bf 0}$ and ${p_{ij}}=0$.
Moreover, the problem is also equivalent to the following convex minimization problem
\begin{align}
\inf_{\lambda,s,u_{ij},v_{ij}} ~~& \lambda\varepsilon+\frac{1}{N}\sum_{i=1}^{N}s_{i}& \label{eq:expectile-4} \\
{\rm subject\;to} ~~& b_{j}+\sigma_{\Xi}(u_{ij}+a_{j})-u_{ij}^{\top}\xi_{i}+\| v_{ij}\| _{*} \varepsilon \leq s_{i} & \forall i,j, \nonumber\\
 & \| v_{ij}\| _{*}\leq\lambda &\forall i,j, \nonumber\\
 & \| u_{ij}+v_{ij}\| _{*}\leq \frac{1-\alpha}{2\alpha - 1} \lambda & \forall i,j, \nonumber
\end{align}
where $\lambda \in \mathbb{R}, s \in \mathbb{R}^N, u_{ij}\in \mathbb{R}^m, v_{ij}\in \mathbb{R}^m$, and $\sigma_{\Xi}$ is the support function of $\Xi$.
\end{theorem}

The formulation \eqref{eq:expectile-2} is obtained from the problem \eqref{eq:gene-2-31} and can be used to compute the worst-case distributions, whereas the formulation \eqref{eq:expectile-4} is the dual of \eqref{eq:expectile-2} and can be used to solve the overall GW-DRO problem as a single minimization problem. The key observation from the above result, particularly the formulation \eqref{eq:expectile-2}, is that while expectiles are similar to CVaR in terms of general functional properties, i.e. both are coherent risk measures that are concave in distributions, the feasible sets of distributions induced from the two, i.e. the feasible set of \eqref{eq:gene-2-31}, have different properties. Namely, the former is convex, whereas the latter is nonconvex. This demonstrates also why exploring different coherent risk measures in defining an ambiguity set can be useful. The above result may appear more limited than the result in the previous section in that the loss function $\ell$ is assumed to be piecewise linear. As another key finding, we show next that in the case where the support set $\Xi$ is unconstrained, i.e. $\Xi = \mathbb{R}^m$, the worst-case expectation problem $\sup _{\mathbb{P} \in \mathbb{B}_{(2),\varepsilon}^{\alpha} (\widehat{\mathbb{P}}_{N} )} \mathbb{E}^{\mathbb{P}}[\ell(\xi)]$ can also be solved more generally for any Lipschitz continuous convex function $\ell$.

\begin{theorem} \label{main_ee}
Under the condition of Theorem \ref{main_cvx}, the worst-case expectation problem
\begin{equation} \label{wcee}
\sup _{\mathbb{P} \in \mathbb{B}_{(2),\varepsilon}^{\alpha} (\widehat{\mathbb{P}}_{N} )} \mathbb{E}^{\mathbb{P}}[\ell(\xi)]
\end{equation}
is equivalent to
\begin{align}\label{eq-expectile-17}
 \frac1N  \sum_{i=1}^N \max\left\{ \ell(\widehat{\xi}_i)+ \beta L \varepsilon ,~\max_{\|e\|=1} \ell(\widehat{\xi}_i+\epsilon e)\right\}
\end{align}
with $\beta=(1-\alpha)/\alpha$.
\end{theorem}

Besides shedding light on the tractability of solving the worst-case expectation problem \eqref{wcee} for a more general class of loss functions, the above result shows, rather unexpectedly, that the problem \eqref{wcee} also admits a structurally simple reformulation, i.e. \eqref{eq-expectile-17}, which is comparable to the reformulation \eqref{connection} in the case of CVaR-Wasserstein ambiguity sets. Different from the reformulation \eqref{connection}, which requires comparing only the sample average of  $\ell(\widehat{\xi}_i) +  L(1-\alpha)\epsilon$ and  the sample average of $\max_{\|e\|=1} \ell(\widehat{\xi}_i+\epsilon e)$, the reformulation \eqref{eq-expectile-17} requires comparing first $\ell(\widehat{\xi}_i)+ \beta L \varepsilon$ and $\max_{\|e\|=1} \ell(\widehat{\xi}_i+\epsilon e)$ with respect to each sample $\widehat{\xi}_i$ before taking the sample average. Overall, the two formulations are very similar in terms of the total number of mathematical operations required for the computations. Some further insight about the reformulation \eqref{eq-expectile-17} can be drawn from the structure of the worst-case distributions to the problem \eqref{wcee}.
Letting $$I:=\left\{i=1,\ldots,N:   \ell(\widehat{\xi}_i)+ \beta L \varepsilon >\max_{\|e\|=1} \ell(\widehat{\xi}_i+\epsilon e) \right\},$$
one can verify that if $I\not=\emptyset$, then the discrete distributions
\begin{equation} \label{wcdisexp}
\mathbb P_n=\frac1N \sum_{i \not\in I}\delta_{\widehat{\xi}_i+ \epsilon e_i} + \frac1N\sum_{i \in I} \left[\left(1-\frac{\epsilon}{\|e_{ni}\|}\right)\delta_{\widehat{\xi}_i} + \frac{\epsilon}{\|e_{ni}\|} \delta_{\widehat{\xi}_i+   e_{ni} } \right],
\end{equation}
where $e_i=\arg\max_{e:\|e\|=1}  \ell(\widehat{\xi}_i+\epsilon e),~i\not \in I$, and $e_{ni}\in\R^m$ satisfies
$
   \lim_{n\to\infty} \frac{\ell(\widehat{\xi}_i+e_{ni}) -\ell(\widehat{\xi}_i)}{ \|e_{ni}\| } = L,
$
$i\in I$, are feasible and asymptotically optimal to the problem \eqref{wcee} as $n \rightarrow \infty$. If $I=\emptyset$, then the worst-case distribution becomes
$$
\mathbb P_{\epsilon}=\frac1N \sum_{i =1}^N \delta_{ \widehat{\xi}_i+ \epsilon e_i }~~{\rm with}~~ e_i=\arg\max_{\|e\|=1} \ell(\widehat{\xi}_i+\epsilon e),~~i=1,\ldots,N.
$$
This second case, in particular, draws the connection between the problem $\sup _{\mathbb{P} \in \mathbb{B}_{(2),\varepsilon}^{\alpha} (\widehat{\mathbb{P}}_{N} )} \mathbb{E}^{\mathbb{P}}[\ell(\xi)]$ and the worst-case expectation problem $\sup _{\mathbb{P} \in \mathbb{B}_{\varepsilon}^{{\rm wc}} (\widehat{\mathbb{P}}_{N} )} \mathbb{E}^{\mathbb{P}}[\ell(\xi)]$, since $\mathbb{P}_{\epsilon}$ is the worst-case distribution of the latter.

To best summarize and illustrate these distributions and their rich structure, we provide in Figure \ref{worst-dis} an example based on three sample points. Each subfigure presents one representative structure of the worst-case distributions (or more precisely, distributions that are asymptotically optimal to the problem \eqref{wcee}). It is most useful to view these distributions by conditioning on each sample. In particular, the conditional distribution is either a two point mass distribution with an arbitrary small weight $p \rightarrow 0$ put on a point that is arbitrary far from a sample and the remaining weight $1-p$ put on the sample, or a distribution concentrated at a single point that is $\epsilon$-away from the sample. The former is illustrated in the figure by placing no boundary (the cylinder) from a sample point, whereas the latter is illustrated by a cylinder with a fixed radius $\epsilon$. These distributions are considerably richer than those of the Wasserstein worst-case expectation problem $\sup _{\mathbb{P} \in \mathbb{B}^{\rm W}_{\varepsilon} (\widehat{\mathbb{P}}_{N} )} \mathbb{E}^{\mathbb{P}}[\ell(\xi)]$ and the CVaR-Wasserstein worst-case expectation problem $\sup _{\mathbb{P} \in \mathbb{B}_{(1),\varepsilon}^{\alpha} (\widehat{\mathbb{P}}_{N} )} \mathbb{E}^{\mathbb{P}}[\ell(\xi)]$ in that

\begin{enumerate}
\item case I corresponds to the worst-case distributions of $\sup _{\mathbb{P} \in \mathbb{B}^{\rm W}_{\varepsilon} (\widehat{\mathbb{P}}_{N} )} \mathbb{E}^{\mathbb{P}}[\ell(\xi)]$, where conditional distributions with respect to all samples are not attainable,
\item case I and IV correspond to the worst-case distributions of $\sup _{\mathbb{P} \in \mathbb{B}_{(1),\varepsilon}^{\alpha} (\widehat{\mathbb{P}}_{N} )} \mathbb{E}^{\mathbb{P}}[\ell(\xi)]$, where conditional distributions with respect to all samples are either all unattainable or all attainable,
\item and case I to IV correspond to the worst-case distributions of $\sup _{\mathbb{P} \in \mathbb{B}_{(2),\varepsilon}^{\alpha} (\widehat{\mathbb{P}}_{N} )} \mathbb{E}^{\mathbb{P}}[\ell(\xi)]$, where the conditional distribution with respect to each sample is either attainable or unattainable.
\end{enumerate}

In short, the structure of the worst-case distributions for the expectile-Wasserstein worst-case expectation problem $\sup _{\mathbb{P} \in \mathbb{B}_{(2),\varepsilon}^{\alpha} (\widehat{\mathbb{P}}_{N} )} \mathbb{E}^{\mathbb{P}}[\ell(\xi)]$ can flexibly vary with respect to each sample.

\begin{figure}[h]
\centering
\subfigure[Case I]{
\includegraphics[scale=0.21]{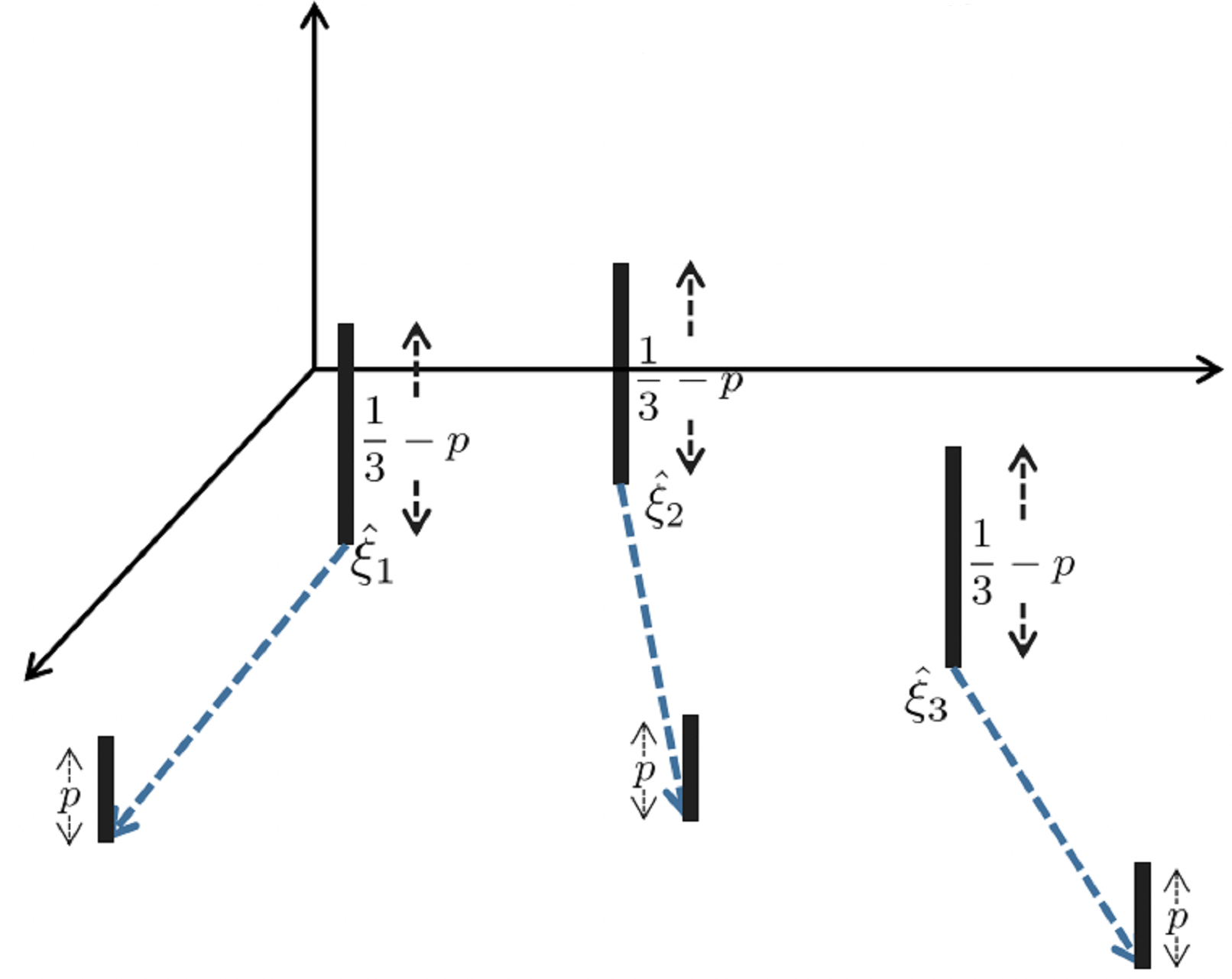}
}
\subfigure[Case II]{
\includegraphics[scale=0.21]{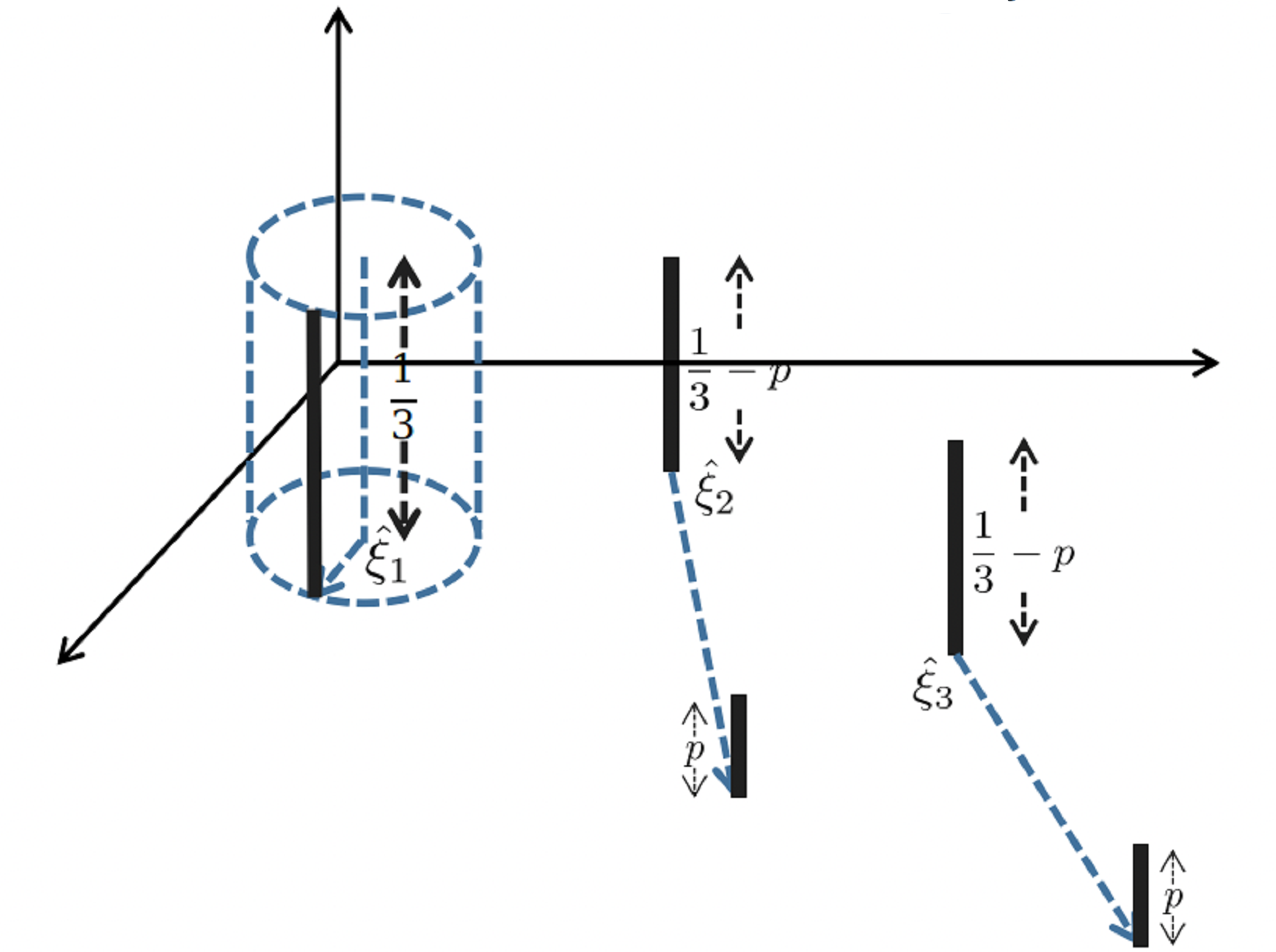}
 }
\subfigure[Case III]{
\includegraphics[scale=0.21]{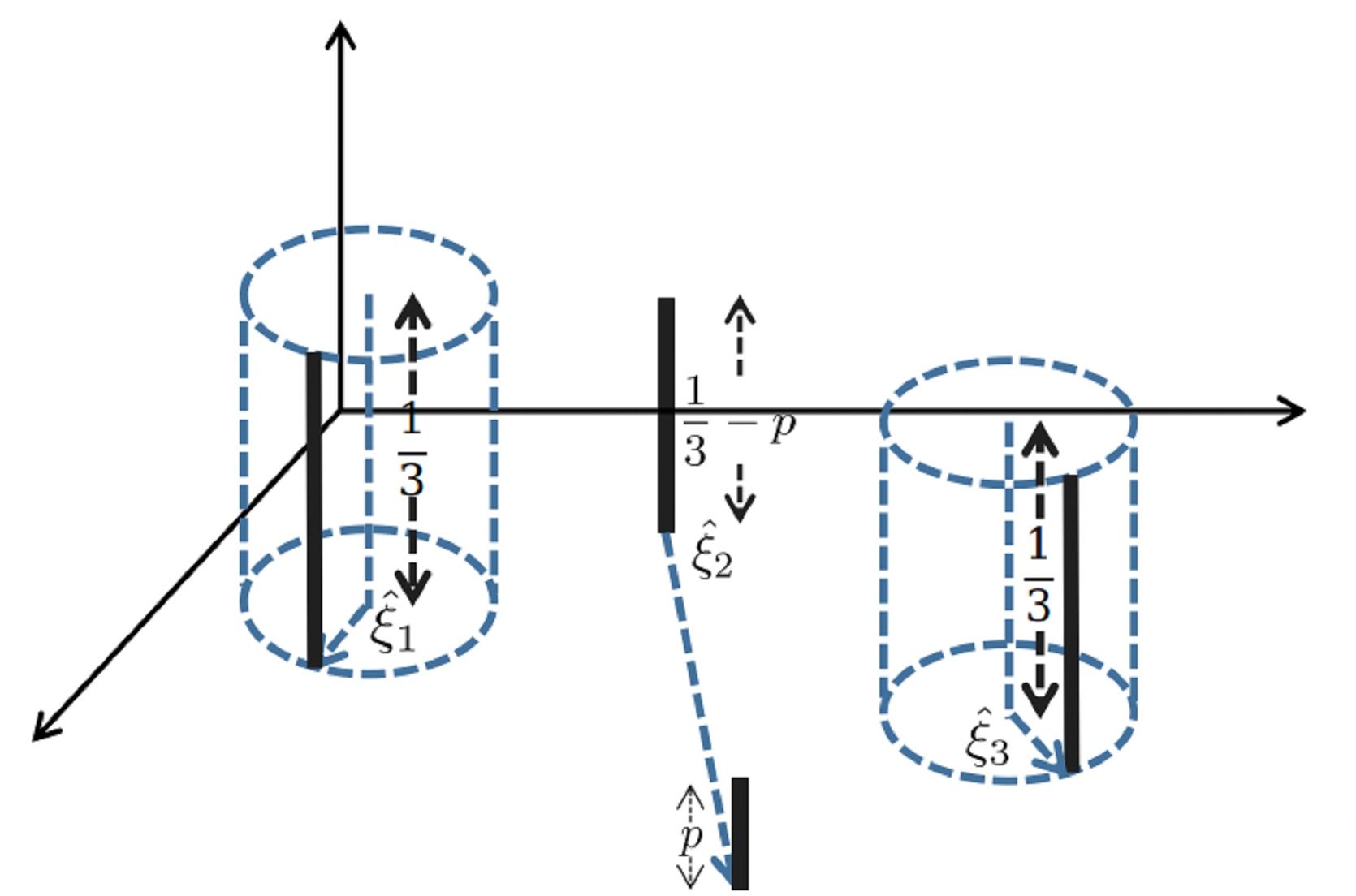}
 }
\subfigure[Case IV]{
\includegraphics[scale=0.18]{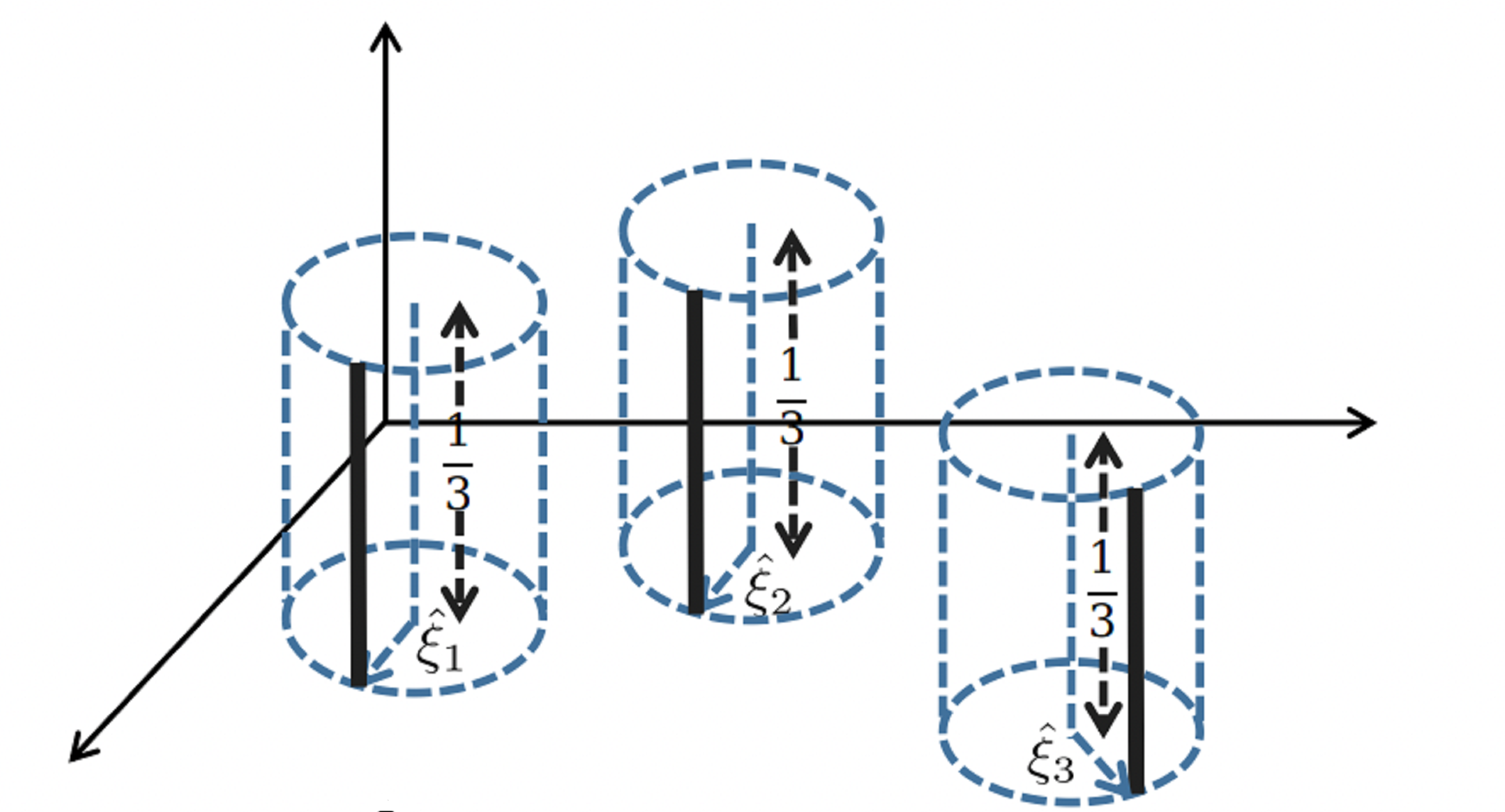}
 }
 \centering
\caption{Worst-case distributions for $\sup _{\mathbb{P} \in \mathbb{B}_{(2),\varepsilon}^{\alpha} (\widehat{\mathbb{P}}_{N} )} \mathbb{E}^{\mathbb{P}}[\ell(\xi)]$}
\end{figure}

With this observation, one can also interpret the parameter $\alpha \in (1/2,1)$ in the expectile-Wasserstein ambiguity sets $\mathbb{B}_{(2),\varepsilon}^{\alpha} (\widehat{\mathbb{P}}_{N} )$ as a parameter that fine tunes the number of sample points contaminated by only bounded perturbations. As $\alpha \rightarrow 1$, the number of such data points increases and the problem becomes increasingly data-driven and less conservative.
The degree of conservativeness, reflected by the worst-case expected value, would more subtly depend on such a structural change of the worst-case distributions, as $\alpha$ varies. This can be seen by comparing the formulation \eqref{eq-expectile-17} for the problem $\sup _{\mathbb{P} \in \mathbb{B}_{(2),\varepsilon}^{\alpha} (\widehat{\mathbb{P}}_{N} )} \mathbb{E}^{\mathbb{P}}[\ell(\xi)]$ and the formulation \eqref{connection} for the problem $\sup _{\mathbb{P} \in \mathbb{B}_{(1),\varepsilon}^{\alpha} (\widehat{\mathbb{P}}_{N} )} \mathbb{E}^{\mathbb{P}}[\ell(\xi)]$. The former depends more nonlinearly on $\alpha$ in a convex fashion, whereas the latter is simply a two-piece linear function in $\alpha$. Figure \ref{optimalvalue} illustrates this difference.

\begin{figure}[h]
\begin{center}
\includegraphics[scale=0.42]{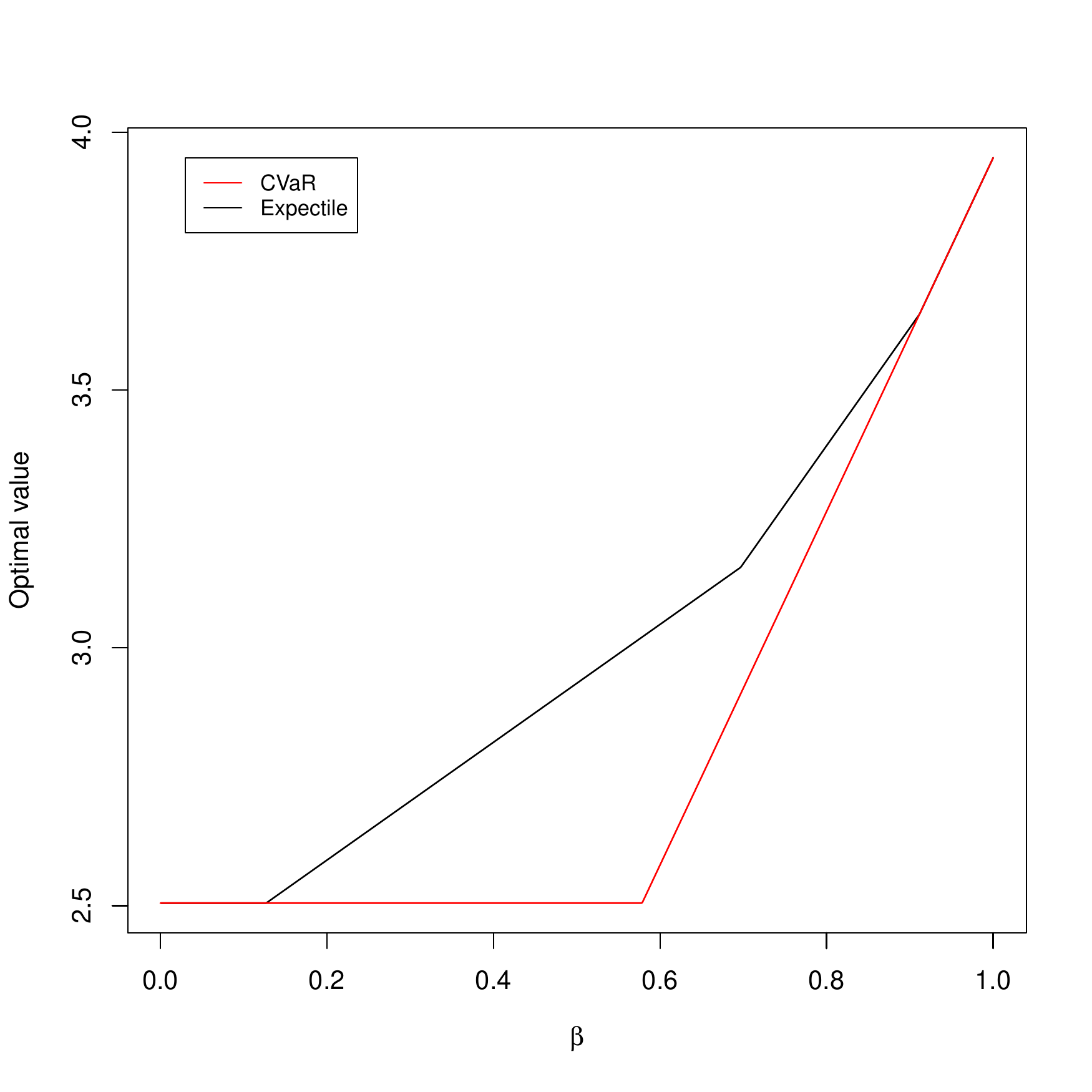}
\includegraphics[scale=0.42]{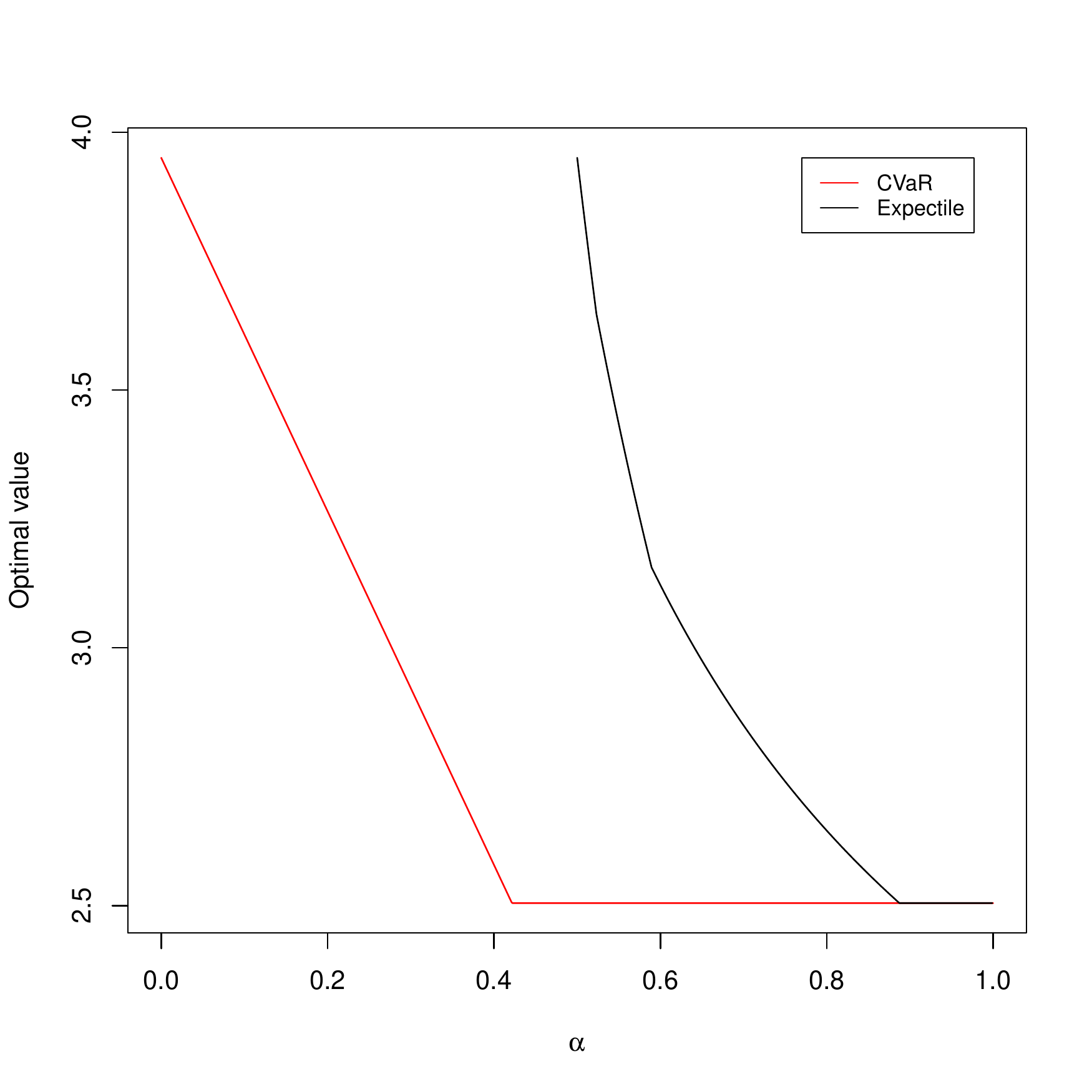}
\caption{Optimal values of $\sup _{\mathbb{P} \in \mathbb{B}_{(1),\varepsilon}^{\alpha} (\widehat{\mathbb{P}}_{N} )} \mathbb{E}^{\mathbb{P}}[f(x^\top\xi)]$ (denoted by CVaR) and $\sup _{\mathbb{P} \in \mathbb{B}_{(2),\varepsilon}^{\alpha} (\widehat{\mathbb{P}}_{N} )} \mathbb{E}^{\mathbb{P}}[f(x^\top\xi)]$ (denoted by expectiles), where $f(t)=t^21_{\{\|t\|\le 1\}} + (7|t|-6)1_{\{\|t\|> 1\}}$ with $x=(1,2,-1)^\top$, $\varepsilon=0.2$, and 
$\widehat{\mathbb P}_N:= \frac13\sum_{i=1}^3\delta_{\widehat{\xi}_i}$ with $\widehat{\xi}_1=(0.2,-0.32,0.5)^\top,$ $
\widehat{\xi}_2=(-0.2,-0.2,0.2)^\top$ and $\widehat{\xi}_3=(0.3,-0.1,-0.1)^\top$.
Left: The lines are with respect to $\beta=\beta_1=1-\alpha$ (CVaR) and $\beta=\beta_2=(1-\alpha)/\alpha$ (Expectile). Right:  The lines are with respect to $\alpha$.  }
\label{optimalvalue}
\end{center}
\end{figure}

As done in the previous section, we identify below the conditions under which there always exists $\alpha < 1$ such that the worst-case distribution exists. It is worth noting that the condition is stronger than the condition identified in Corollary \ref{wcdis}, but the two coincide when $\min_{x\in\R^m}\|\partial \ell(x)\| > 0$, i.e. the loss function $\ell$ does not contain any constant piece.

\begin{corollary} \label{wcdis-expectile} Under the condition of Theorem \ref{main_cvx},  if
$ \max_{\|e\| \le 1} \ell(\widehat{\xi}_i+\epsilon e) > \ell(\widehat{\xi}_i) $ for all $i=1,\ldots,N$, then there always exists $\alpha \in (1/2,1)$ such that the worst-case expectation problem \eqref{wcexptile} is attainable, i.e. the worst-case distribution exists.
\end{corollary}


Finally, it is clear that Theorem \ref{main_ee} can be applied, just as Theorem \ref{main_cvx}, to solve problems in Section \ref{illexample} as convex programs.

\begin{corollary} \label{reduced2}
In the case where $\ell(\xi) = f(x^\top \xi)$ and $f$ is a Lipschitz continuous convex function in $\mathbb{R}$, the worst-case expectation problem \eqref{wcee} can be solved by
\begin{align}\label{eq-expectile-19}
\min_{x \in \mathbb{X}} \frac1N  \sum_{i=1}^N \max\left\{f(x^\top \xi_i)+  {\rm Lip}(f) \|x\|_* \beta \varepsilon , ~f_1 (x^\top  \xi_i -\epsilon\|x\|_{*} ),~f_2 (x^\top  \xi_i + \epsilon\|x\|_{*}) \right\},
\end{align}
where $f_1(t)= f(t)1_{\{t<t_0\}}+ f(t_0)1_{\{t\ge t_0\}} $ and $f_2(t) = f(t_0)1_{\{t\le t_0\}}+ f(t)1_{\{t> t_0\}} $ and $t_0\in [-\infty,\infty]$ is such that $f$ is decreasing on $(-\infty,t_0)$ and increasing on $(t_0,\infty)$.
\end{corollary}

\begin{example} \label{applications} {(continue Example (ii), Section \ref{illexample})} Applying Corollary \ref{reduced2} to the robust portfolio optimization problem
$$ \min_{x \in \mathbb{X}} \sup _{\mathbb{P} \in \mathbb{B}_{(2),\varepsilon}^{\alpha}(\widehat{\mathbb{P}}_{N} )}   \mathbb{E}^{\mathbb{P}}[-u( \xi^\top x)],$$
where $u$ is Lipschitz continuous, we obtain the following convex program
$$
\min_{x \in \mathbb{X}}
 \frac1N\sum_{i=1}^N \max \left\{ -u(x^\top \xi_i) +  {\rm Lip}(u)\|x\|_{*}(1-\alpha)\epsilon,
 -u (x^\top  \xi_i -\epsilon\|x\|_{*} ) \right\}.
$$
\end{example}

\begin{example} {(continue Example (iii), Section \ref{illexample})} Applying Corollary \ref{reduced2} to the distributionally robust regression problem
$$ \min_{\beta \in {\cal B}} \sup _{\mathbb{P} \in \mathbb{B}_{(2),\varepsilon}^{\alpha}(\widehat{\mathbb{P}}_{N} )}
\mathbb{E}^\mathbb{P}[ \ell(\beta^\top \xi^x - \xi^y) ]$$
and following Example \ref{ex-ml}, we obtain
\begin{align*}
\min_{\beta \in {\cal B}}\frac{1}{N}\sum_{i=1}^{N} \max & \left\{ \begin{array}{l}
\ell(\beta^{\top}\xi_{i}^{x}-\xi_{i}^{y})+{\rm Lip}(\ell) \| (\beta,-1)\| _{*}(1-\alpha)\varepsilon,\\
\ell_{1}(\beta^{\top}\xi_{i}^{x}-\xi_{i}^{y}-\| (\beta,-1)\| _{*}\varepsilon),
\ell_{2}(\beta^{\top}\xi_{i}^{x}-\xi_{i}^{y}+\| (\beta,-1)\| _{*}\varepsilon)
\end{array}\right\},
\end{align*}
whereas in the case of distributionally robust classification problem, i.e.
$$ \min_{\beta \in {\cal B}} \sup _{\mathbb{P} \in \mathbb{B}_{(2),\varepsilon}^{\alpha}(\widehat{\mathbb{P}}_{N} )}
\mathbb{E}^\mathbb{P}[ \ell(\xi^y \cdot \beta^\top \xi^x) ], $$
we obtain
\[
\min_{\beta \in {\cal B}}\frac{1}{N}\sum_{i=1}^{N} \max\left\{ \ell(\xi_{i}^{y}\cdot\beta^{\top}\xi_{i}^{x})+{\rm Lip}(\ell)  \| \beta\| _{*}(1-\alpha)\varepsilon, ~\ell(\xi_{i}^{y}\cdot\beta^{\top}\xi_{i}^{x}-\| \beta\| _{*}\varepsilon)\right\}.
\]
\end{example}

\section{Conclusion}
In this paper, we propose a general framework for identifying families of Wasserstein metrics suited for data-driven distributionally robust optimization. We show that our framework offers a fruitful opportunity to design novel Wasserstein DRO models that can be theoretically sound and practically well-motivated. Necessary analysis is provided in this paper to facilitate tractable reformulations of the Wasserstein DRO models that adopt coherent Wasserstein metrics. We demonstrate the application of our framework using ambiguity sets constructed from CVaR- and expectile-Wasserstein metrics and provide an in-depth discussion of the connection between the new Wasserstein DRO models and the type-1 Wasserstein DRO model. The former generalizes the latter in an intuitive way, having a more enriched structure of optimization and worst-case distributions. In addition to several applications covered in this paper, the framework established in this paper shall provide an important basis for exploring further potential of Wasserstein DRO in a broader set of applications.

\section{Appendix}

\subsection{Proofs of Section \ref{sec:2}}

\noindent{\bf Proof of Proposition \ref{pro_eu}.} We show the result by proving the following two statements.
\begin{itemize}
 \item [(i)] For any $\alpha\in A$, if $\rho_\alpha$ takes finite value on $L^1$, then $\ell_\alpha$ is a Lipschitz function.
\item [(ii)] If $\ell_\alpha$ is a Lipschitz function for any $\alpha\in A$, then there does not exist a subsequence of $\rho_{\alpha}$, $\alpha\in A$,  converging to the worst-case risk measure $\esssup$.
\end{itemize}
Obviously, with statements (i) and (ii),  the {\bf robustness} and {\bf data-drivenness} of {\bf DD-DRW} could not be satisfied simultaneously.  We next show  (i) and (ii).

To see (i),  we show that $\ell_\alpha$ is a Lipschitz function by contradiction. Suppose not, i.e.,  $\ell_\alpha$ is not Lipschitz.  We will construct a random variable $X$ in $L^1$ whose risk measure $\rho_\alpha(X)$ takes infinity value.  By the increase and convexity of $\ell_\alpha$, we have $\ell_\alpha'$ is non-decreasing and unbounded, and thus, $\lim_{x\to\infty} \ell_\alpha'(x)=\infty,$ where  $\ell_\alpha'$ is the left-derivative of $\ell_\alpha$.
It then follows from the convexity of $\ell_\alpha$ that $ \ell_\alpha(x) \ge \ell_\alpha(x/2)   +  \ell_\alpha'(x/2)x/2$ for all $x>0$, and thus,
 $$\lim_{x\to\infty} \frac{\ell_\alpha(x)}{x} \ge \lim_{x\to\infty}  \frac12\ell_\alpha'\left(\frac x2\right) =\infty.$$
 Therefore, for each $n\in\N$, there exists $x_n\ge n$  such that $\ell_\alpha (x_n)> 2^n x_n$, $n\in\N$.
 Define a random variable $X$ such that $\p(X=x_n)=  \frac c{2^nx_n}$, $n\ge 1$, where $c=1/ \sum_{n=1}^\infty (2^nx_n)^{-1}$. One can verify that $X\in L^1$ as $\E[X] = \sum_{n=1}^\infty \frac c{2^n}=c<\infty$, and meanwhile
  $$\rho_\alpha(X)=\ell_\alpha^{-1}\left(\E[\ell_\alpha(X)] \right)= \ell_\alpha^{-1}\left(\sum_{n=1}^{\infty}  \frac c{2^nx_n} \ell_\alpha(x_n)\right)\ge\ell_\alpha^{-1}\left( \sum_{n=1}^{\infty}  c \right)=\infty,$$
  where the inequality follows from $\ell_\alpha (x_n)> 2^n x_n$, $n\in\N$.
 Therefore, this yields a contradiction to that $\rho_\alpha$ takes  finite value on $L^1$, and thus, (i) holds.

 To see (ii), we also show it by contradiction. Suppose that there exist $\alpha_n\in A$, $n\in\N$, such that  $\rho_{\alpha_n}$  converges to $\esssup$. We first assert that there must exist a subsequence of $\{\alpha_n,n\in\N\}$, say $\beta_n$, $n\in\N$, and  $c\in\R$ such that
 \begin{align}\label{eq:con-1}
 \ell_{\beta_n} ~{\rm converges ~to~} \infty1_{\{(\cdot)>c\}}~~{\rm or}~~ \infty1_{\{(\cdot)\ge c\}} ~~{\rm as}~ ~n\to\infty.
 \end{align}
 {\bf (Proving  \eqref{eq:con-1})}. To see \eqref{eq:con-1},  for any fixed $x>0$ and an event $A$  with $\p(A)=p\in (0,1)$, take $Y = x 1_{A}$. By  $\rho_{\alpha_n} \to \esssup$ as $n\to\infty$, we have
\begin{equation}\label{eq:220708-1}\rho_{\alpha_n}(Y) = \ell_{\alpha_n}^{-1}( \E [\ell_{\alpha_n}(Y)] ) = \ell_{\alpha_n}^{-1} ( p \ell_{\alpha_n}(x ) ) \to  x  ~~{\rm as}~ ~n\to\infty.\end{equation}
For the chosen $x $, denote $z_{\alpha_n}:=\ell_{\alpha_n}(x )$, $n\in\N$. Note that if $\{z_{\alpha_n},n\in\N\}$  is bounded, then  there exists a subsequence $\alpha_n'$ of $\alpha_n$, $n\in\N$, such that $z_{\alpha_n'} \to a$, i.e., $ \ell_{\alpha_n}(x )\to a$ for some $a\in\R$.  By the increase and convexity of $\ell_{\alpha_n'}$, \eqref{eq:220708-1} implies $  p a =a$ and thus $a=0$.
    Hence, $\lim_{n\to\infty}\ell_{\alpha_n}(x)=0.$
Take $c$ as the supremum of $x$ such that $\{\ell_{\alpha_n}(x),n\in\N\}$ is bounded.
Then we have that   $\lim_{n\to\infty}\ell_{\alpha_n}(x)=0$ for $x<c$, and $\{\ell_{\alpha_n}(x),n\in\N\}$ is not bounded for $x>c$.
We consider the following two cases.
\begin{itemize}
\item [(a)]
 If $\{\ell_{\alpha_n}(c),n\in\N\}$ is not bounded, then there exist $\alpha_n'$ such that $\ell_{\alpha_n'}(c)\to\infty$. By the increase of $\ell_{\alpha_n'}$, we have  $\lim_{n\to\infty}\ell_{\alpha_n'} (x)=\infty$ for $x\ge c$. Hence, $\ell_{\alpha_n'} (x)~{\rm converges ~to~}   \infty1_{\{(\cdot)\ge c\}}$  as $n\to\infty$, that is, \eqref{eq:con-1} holds with $\beta_n=\alpha_n'$.
\item [(b)]If  $\{\ell_{\alpha_n}(c),n\in\N\}$ is bounded, then by $\{\ell_{\alpha_n}(x),n\in\N\}$ is not bounded for $x>c$, for each $k\in \N$, we can find $\alpha_{n_k}$ such that  $\ell_{\alpha_{n_k}}(x+1/k)>k$. Then we have $\lim_{k\to\infty}\ell_{\alpha_{n_k}} (x)=\infty$ for $x>c$. Hence, $\ell_{\alpha_{n_k}}$ converges to $\infty1_{\{(\cdot)>c\}}$  as $k\to\infty$, that is, \eqref{eq:con-1} holds with $\beta_k=\alpha_{n_k}$.
\end{itemize}
Combining the above two cases, we have \eqref{eq:con-1} holds.

 Now with \eqref{eq:con-1}, let $X$ be a random variable such that $\p(X=d) = p =1-\p(X=0)$  for some  $d > c$ and $p\in (0,1)$. One can verify that  $$\lim_{n\to\infty}\rho_{\beta_n}(X) = \lim_{n\to\infty}\ell_{\beta_n}^{-1}(p \ell_{\beta_n}(d)) = \lim_{n\to\infty}\ell_{\beta_n}^{-1}(\infty) =c < d=\esssup X,$$ yielding a contradiction to the assumption that $\rho_{\alpha_n}\to\esssup$ and $\{\beta_n,n\in\N\}\subseteq\{\alpha_n,n\in\N\}$. Therefore, we have (ii) holds.

 We thus complete the proof.
\qed

The following lemma can be found in Theorem 4.2 of \cite{BM06} (see also \cite{D12}) which will be used in the proof of  Proposition \ref{pr:1-sec1} and Theorem \ref{th:main-2}.
\begin{lemma}
\label{lem:a1}
 Any  law-invariant convex (and thus coherent) risk measure that satisfies lower-semicontinuity in $L^1$ and $\rho(0)=0$ must be consistent with increasing convex order, that is, if  $X\preceq_{\rm icx} Y$\footnote{For two random variables $X$ and $Y$, $X$ is said to be smaller than $Y$ with respect to increasing convex order, denoted by $X\preceq_{\rm icx} Y$, if $\E [u(X)]\le \E u(Y)$ for any increasing convex function $u$. It is easy to see that $X\succeq_{\rm icx} Y$ if and only if $-Y \succeq_{\rm SSD} -X$, i.e., $-Y$ is smaller than $-X$ in second-order stochastic dominance.}, then $\rho(X)\le \rho(Y)$. In particular, $ \E [X] \le \rho(X)$.
\end{lemma}

\noindent{\bf Proof of Proposition \ref{pr:1-sec1}.}
(i) To show the ``if'' part, note that when $\mathbb P_1=\mathbb P_2=\mathbb P$,  the joint distribution of $(\xi,\xi)$ lies in  the set $ \Pi(\mathbb{P}_1,\mathbb{P}_2)$ where $\xi\sim \mathbb P$, and $\rho(\|\xi-\xi\|)=0$. Hence,
$d_{\rho}(\mathbb{P}_1, \mathbb{P}_2)=0$.
To show the ``only if'' part,  suppose now that $d_{\rho}(\mathbb{P}_{1}, \mathbb{P}_{2})= 0$. By Lemma \ref{lem:a1}, we have  $\rho\ge \E$, and thus, $d_{\rm W}(\mathbb{P}_{1}, \mathbb{P}_{2})= 0$. By that the Wasserstein metric satisfies identity of indiscernibles, we have     $\mathbb{P}_{1}= \mathbb{P}_{2}$.

(ii) The symmetry follows directly from  the definition.

(iii) 
Note that by the definition of $d_{\rho}$, for any $\epsilon>0$, there exist $\Pi_1\in \Pi(\mathbb P_1,\mathbb P_2)$ and $\Pi_2\in \Pi(\mathbb P_2,\mathbb P_3)$ such that
 $$
 d_{\rho}\left(\mathbb{P}_{1}, \mathbb{P}_{2} \right) \ge \rho^{\Pi_1}(\|\xi_1-\xi_2\|) -\epsilon
~~~{\rm  and }~~~
 d_{\rho}\left(\mathbb{P}_{2}, \mathbb{P}_{3} \right) \ge \rho^{\Pi_2}(\|\xi_1-\xi_2\|) -\epsilon .
 $$
By Theorem 6.10 of \cite{K97}, there exist $\xi_1^*,\xi_2^*,\xi_3^*$ such that $(\xi_1^*,\xi_2^*)$ has the joint distribution $\Pi_1$ and $(\xi_2^*,\xi_3^*)$ has the joint distribution $\Pi_2$.
It follows that
\begin{align*}
 d_{\rho}\left(\mathbb{P}_{1}, \mathbb{P}_{2} \right) + d_{\rho}\left(\mathbb{P}_{2}, \mathbb{P}_{3} \right)
& \ge  \rho (\|\xi_1^*-\xi_2^*\|) +\rho(\|\xi_2^*-\xi_3^*\|)  -2\epsilon\\
 &  \ge \rho\left(\|\xi_1^*-\xi_3^*\|\right)  -2\epsilon\ge d_{\rho}\left(\mathbb{P}_{1}, \mathbb{P}_{3} \right)-2\epsilon,
\end{align*}
where the second inequality follows from the subadditivity of $\rho$
%
and the subadditivity of $\|\cdot\|$,
and the last inequality follows from the definition of $d_{\rho}$ and ${\xi}_i^*\sim \mathbb P_i$, $i=1,3$.
  By the arbitrariness of $\epsilon$, we have  $d_{\rho}\left(\mathbb{P}_{1}, \mathbb{P}_{2} \right) + d_{\rho}\left(\mathbb{P}_{2}, \mathbb{P}_{3} \right)
\ge  d_{\rho}\left(\mathbb{P}_{1},  \mathbb{P}_{3} \right).$

(iv) The non-negativity follows from the nonnegativity of the norm $\|\cdot\|$ and $\rho(X)\ge0$ for $X\ge 0$.
We thus complete the proof.
\qed

\noindent{\bf Proof of Proposition \ref{prop:3insec2}.}
To show the ``if'' part,   for any $\alpha\in A$,  with the representation \eqref{eq:convex-kusuoka}, we have for any nonnegative random variable $X$
 \begin{align} \label{eq-0604-1}
\rho_\alpha(X)
& = \sup _{g \in {\mathcal H_{\rho_\alpha} }}
 \left\{\int_{0}^{1} \operatorname{VaR}_{\alpha}(X) \mathrm{d} g(\alpha) \right\}
   \le \sup _{g \in {\mathcal H_{\rho_\alpha} } }  \|g'\|_\infty \int_{0}^{1}  \operatorname{VaR}_{\alpha}(X) \mathrm{d} \alpha = {c_\alpha} \E  [X] ,
\end{align}
where   the inequality follows from  the H\"older inequality. This implies the robustness holds.
By $c_{\alpha_n}\to\infty$ as $n\to\infty$, there exist  $g_n\in {\mathcal H_{\rho_{\alpha_n}}}$, $n\in\N$,  such that $\|g_n'\|_\infty>c_{\alpha_n}-1/n$, and thus $\|g_n'\|_\infty =g_n'(1)\to\infty$ as $n\to\infty$. Note that $\|g_n'\|_\infty\in\R$ which implies $g_n$ is Lipschitz continuous. We have $g_n(\alpha)\to 1_{\{\alpha=1\}}$,
and thus,
\begin{align}\label{eq:conve1}
\rho_{\alpha_n}(X)
& \ge  \int_{0}^{1} \operatorname{VaR}_{\alpha}(X)  \mathrm{d}g_n(\alpha)  \to \int_{0}^{1} \operatorname{VaR}_{\alpha}(X)  1_{\{\alpha=1\}} =\esssup X~~{\rm as}~n\to\infty,
\end{align}
  where the convergence follows from  dominated convergence theorem as 
  $ \int_{0}^{1} \operatorname{VaR}_{\alpha}(X)  \mathrm{d}g_n(\alpha)\le \esssup X.$
 Also, note that $\rho(c)=c$ and $\rho$ is monotone,  $\rho(X)\le \esssup X$.  This together with \eqref{eq:conve1} implies $\rho_{\alpha_n}\to \esssup $ as $n\to\infty$.

 We next consider the ``only if'' part. First note the {\bf robustness} property implies that $\rho_\alpha$ takes finite value in $L^1$. By Corollary 2.6  and Theorem 2.9 of \cite{R13}, we have $\rho_\alpha$ must be $L^1$ continuous and thus can be represented by \eqref{eq:convex-kusuoka} with   $\mathcal H_{\rho_\alpha} $ being a set of Lipschitz continuous convex distortion functions with $c_\alpha =\sup_{g\in \mathcal H_{\rho_\alpha}}\|g'\|_\infty<\infty$.
Further, by   the property of {\bf data-drivenness}, there exist $\alpha_n\in A$, $n\in \N$, such that  $\rho_{\alpha_n} \to \esssup $ as $n\to\infty$. By \eqref{eq:convex-kusuoka}, for a random variable $X$, there exist $g_n \in \mathcal H_{\infty}^{c_{\alpha_n}}, n\in\N$,  such that
$$
 \int_{0}^{1} \operatorname{VaR}_{\alpha}(X) \mathrm{d} g_n(\alpha)   \to \esssup X~~{\rm as}~n\to\infty.
$$
This holds only if $g_n (\alpha)\to 1_{\{\alpha=1\}}$ as $n\to\infty$. Hence, we have $c_{\alpha_n}\ge  \|g'_n \|_\infty\to\infty$ as $n\to\infty$. This completes the proof.
\qed

\noindent{\bf Proof of Proposition \ref{pr-3finitesample}.}  By assumption on $\rho_\alpha$, it satisfies  Proposition \ref{prop:3insec2} with $c_\alpha=c$, and thus,  \eqref{eq-0604-1} holds  for any nonnegative random variable $X$.
That is,  $ \rho_\alpha (X)\le c\E[X]$ for all nonnegative  $X$.
  Therefore,  $ d_{\rho_\alpha} \le c d_{\mathrm{W}}$. It follows that $ d_{\rm W} \left(\mathbb{P},\widehat{\mathbb{P}}_{N}\right)\le \epsilon/c$ implies $d_{\rho_\alpha} \left(\mathbb{P},\widehat{\mathbb{P}}_{N}\right)\le \epsilon$, and thus, %
\begin{equation} \label{eq:setinclu-1}
 \mathbb{B}^{\rm W}_{\varepsilon/c}\left(\widehat{\mathbb{P}}_{N}\right) \subseteq \mathbb{B}^{\rho_\alpha}_{\varepsilon}\left(\widehat{\mathbb{P}}_{N}\right). 
\end{equation}
Recall that  Theorem 3.2 of \cite{EK18} gives
\begin{equation} \label{eq:fsg-N}
 \p\left( \mathbb{B}_{\varepsilon_{N}^{\rm EK}(\eta)}^{\rm W}(\widehat{\mathbb{P}}_{N})\right) \ge 1-\eta,
~~~{\rm where}~~  
\varepsilon_{N}^{\rm EK}(\eta)=
 \begin{cases}
\varepsilon_{0}^{1 /  m_2} & \text { if } 1 \geq \varepsilon_{0}, \\
\varepsilon_{0}^{1 / a} & \text { if } 1<\varepsilon_{0},\end{cases}
\end{equation}
for some  constants $c_1,c_2$ only depending on $a$, $A$ and $m$. Obviously, we can write $\varepsilon_N(\eta)=c \varepsilon_{N}^{\rm EK}(\eta)$, and thus,
  the set inclusion \eqref{eq:setinclu-1} and \eqref{eq:fsg-N} imply 
\begin{equation} \label{eq:fsg-N}
 \p\left( \mathbb{B}^{\rho_\alpha}_{\varepsilon_{N} (\eta)} (\widehat{\mathbb{P}}_{N})\right)\ge \p\left( \mathbb{B}_{\varepsilon_{N}^{\rm EK}(\eta)}^{\rm W}(\widehat{\mathbb{P}}_{N})\right) \ge 1-\eta.
\end{equation}
  Therefore, the finite sample guarantee \eqref{eq-fsg} holds for $\varepsilon_N(\eta)$, which completes the proof. \qed

 \noindent{\bf Proof of Proposition \ref{prop:3}.} Let $c_1,c_2$ be constants in \eqref{eq:fsg-N}.
 For $N\in\N$, define $$
 \eta_N^{(1)} =
 \begin{cases}
 c_1 e^{ -c_2 N\varepsilon_N^{m_2}}, & \varepsilon_N \le 1,\\
 c_1 e^{ -c_2 N\varepsilon_N^{a}}, & \varepsilon_N >1,
 \end{cases}
 ~~ ~~{\rm  and }~~~
 \eta_N^{(2)}=
 \begin{cases}
 c_1 e^{ -c_2 N(\varepsilon_N/c)^{m_2}}, & \varepsilon_N \le c,\\
 c_1 e^{ -c_2 N(\varepsilon_N/c)^{a}}, & \varepsilon_N >c.
 \end{cases}
 $$
 One can verify that   $\varepsilon_N \le 1\le c$ for $N$ large enough, and thus,
   $$
   \eta_N^{(i)}=c_1e^{-c_2k_N} , ~i=1,2,~~~~ \varepsilon_{N}^{\rm EK}(\eta_N^{(1)})=\varepsilon_{N},~~~~\varepsilon_{N}^{\rm EK}(\eta_N^{(2)})=\frac{\varepsilon_{N}}{c},
 $$  where $\varepsilon_{N}^{\rm EK}(\eta_N^{(i)})$ is defined by \eqref{eq:fsg-N}.  By assumptions on $k_N$ and $\varepsilon_{N}$, we have
  $$
  \sum_{N=1}^\infty\eta_N^{(i)}<\infty,~i=1,2,~~~\varepsilon_{N}^{\rm EK}(\eta_N^{(1)}) \to 0, ~~~\varepsilon_{N}^{\rm EK}(\eta_N^{(2)}) \to 0 ~~{\rm as} ~~N\to\infty.
  $$
  That is,  both $\eta_N^{(i)}$, $i=1,2$ satisfy the condition of  Theorem 3.6 of \cite{EK18}. Denote by
\begin{equation*}
\widehat{J}_{N}^{(1)}: = \inf _{x \in \mathbb{X}} \sup _{\mathbb{P} \in  \mathbb{B}^{\rm W}_{\varepsilon_N} \left(\widehat{\mathbb{P}}_{N}\right) } \mathbb{E}^{\mathbb{P}}[h(x, \xi)]~~~{\rm and}~~~\widehat{J}_{N}^{(2)} := \inf _{x \in \mathbb{X}} \sup _{ \mathbb{P} \in  \mathbb{B}^{\rm W}_{\varepsilon_N/c} \left(\widehat{\mathbb{P}}_{N}\right) } \mathbb{E}^{\mathbb{P}}[h(x, \xi)].
\end{equation*}
We have  $\mathbb{P}^{\infty}$-almost surely $\widehat{J}_{N}^{(i)} \downarrow J^{\star}$ as $N \rightarrow \infty$. By $d_{\mathrm{W}} \le d_{\rho_\alpha} \le  c d_{\mathrm{W}}$, which implies
\begin{equation} \label{eq:setinclu}
 \mathbb{B}^{\mathrm{W}}_{\varepsilon_N/c }\left(\widehat{\mathbb{P}}_{N}\right) \subseteq \mathbb{B}^{\rho_\alpha}_{\varepsilon_N}\left(\widehat{\mathbb{P}}_{N}\right)\subseteq \mathbb{B}^{\mathrm{W}}_{ \varepsilon_N}\left(\widehat{\mathbb{P}}_{N}\right)
\end{equation}
and thus,
\begin{equation*}
\widehat{J}_{N}^{(2)} \le \widehat{J}_{N} \le  \widehat{J}_{N}^{(1)},
\end{equation*}
Therefore,  we have $\mathbb{P}^{\infty}$-almost surely $\widehat{J}_{N} \downarrow J^{\star}$ as $N \rightarrow \infty$, that is, (i) holds.
The assertion (ii)  can be shown by similar arguments in the proof of  (ii) of Theorem 3.6 of \cite{EK18}.  \qed

\subsection{Proofs of Section \ref{sec:4-concave}}

\noindent{\bf Proof of  Theorem \ref{thm1}.} 
Note that the distribution $\Pi $ with $\Pi((\widehat{\xi},\xi)=(\widehat{\xi}_i,y_i)) =1/N,~i=1,\ldots,N$ satisfies $\Pi\in \Pi( \widehat{\mathbb{P}}_{N}, \mathbb{P})$. We obviously have that the problem \eqref{eq:main} is an upper bound of the   problem \eqref{eq:main-eq1-7}.   To show their equivalence, it suffices to show that the   problem \eqref{eq:main-eq1-7} is also an upper bound of the problem \eqref{eq:main}. To see this,  for any $\Pi\in\mathcal M(\Xi^2)$  satisfying $\Pi\in\Pi(\widehat{\mathbb P}_N,\mathbb P)$ and $ \rho^{\Pi}(\|\widehat{\xi}-\xi\|)\leq \epsilon$, define a new joint distribution $\Pi^*$ as
$$
\Pi^*:= \frac1N\sum_{i=1}^N \delta_{(\widehat{\xi}_i, y_i)} ~~~{\rm with }~~~y_i 
=\E^{\Pi}[\xi |\widehat{\xi}=\widehat{\xi}_i],~~i=1,\ldots,N,
$$
$\delta_{x}$ is the  degenerated distribution at  point $x$. 
 Let $(\widehat{\xi}^*, {\xi}^*)$ be a random vector having the distribution $\Pi^*$.
 Denote by $\mathbb P^*$  the marginal distribution of   $\xi^*$.   We   have the following facts.
\begin{itemize}
\item [(i)] We have $\widehat{\xi}^*\sim \widehat{\mathbb P}_N$, that is, $\Pi^*\in\Pi(\widehat{\mathbb P}_N,\mathbb P^*)$. By the convexity of $\Xi$, we have $y_i\in \Xi$, $i=1,\ldots,N$, and thus,
    $\Pi^*\in\mathcal M(\Xi^2)$.
\item [(ii)] Denote by $c(x,y)=\|x-y\|$ which is a convex function on $\R^m\times \R^m$. We have for any increasing convex function $u$, it holds that  $v:=u\circ c$ is a convex function as for any $z_1,z_2$, $\lambda\in [0,1]$,
   it holds that
   $$
     u(c(\lambda z_1+ (1-\lambda)z_2)) \le  u (\lambda c(z_1)+ (1-\lambda)c(z_2) )\le \lambda u ( c(z_1))+ (1-\lambda) u(c(z_2) ).
   $$
    Hence, we have
\begin{align*}
 \E^{\Pi^*}[u( \|\widehat{\xi}^*-\xi^*\|)]
 & =  \E^{\Pi^*}[v(\widehat{\xi}^*,\xi^* )]\\
 &  = \frac1N\sum_{i=1}^N v(\widehat{\xi}_i, y_i )]\\
  & = \frac1N \sum_{i=1}^N v(\E^{\Pi}[(\widehat{\xi},\xi) |\widehat{\xi}=\widehat{\xi}_i])  \\
 & \le \frac1N \sum_{i=1}^N \E^{\Pi} [v(\widehat{\xi}, \xi )|\widehat{\xi}=\widehat{\xi}_i]   \\
 & =  \E^{\Pi}[v(\widehat{\xi},\xi )]=\E^{\Pi}[u( \|\widehat{\xi}-\xi\|)] ,
\end{align*}
where the inequality follows from the Jensen inequality. Noting the above inequality holds for any increasing function $u$, which implies that
$$
 \|\widehat{\xi}^*-\xi ^*\| \le_{\rm icx}\|\widehat{\xi}-\xi \| .
$$
By Lemma \ref{lem:a1}, this implies $\rho^{\Pi^*}(\|\widehat{\xi}^*-\xi^*\|) \le \rho^{\Pi}(\|\widehat{\xi} -\xi \|)$, and thus, $\rho^{\Pi^*}(\|\widehat{\xi}^*-\xi^*\|)\le \epsilon$.
\item [(iii)]Denote by $\mathbb P_i$  the conditional distribution of $\xi$ given $\widehat{\xi}=\widehat{\xi}_i$ when the joint distribution of $(\widehat{\xi},\xi)$ is  ${\Pi}$.
 Noting that  $\ell$ is concave, we have  $\E^{\mathbb P_i}[\ell(\xi)] \le \ell(\widehat{\xi}_i)$, $i=1,\ldots,N$. It follows that
$$
 \E^{\Pi}[\ell(\xi)]=\frac1N\sum_{i=1}^N \E^{\mathbb P_i}[\ell(\xi)] \le  \frac1N\sum_{i=1}^N \ell(\widehat{\xi}_i)= \E^{\Pi^*}[\ell(\xi^*)].
$$
\end{itemize}
Combining the above three facts, we have that the   problem \eqref{eq:main-eq1-7} is also an upper bound of the problem \eqref{eq:main}, which completes the proof.
\qed

~~

\noindent{\bf Proof of Corollary \ref{dual_form}.}
We begin by noting that given any random variable $Y$ such that $\p(Y=y_i)=1/N$, $i=1,\ldots,N$, any law-invariant convex risk measure $\rho$ can also be written  as
\begin{equation} \label{eq-convexrm}
\rho(Y) = \sup_{z\in A_\rho} \frac1N\sum_{i=1}^N y_iz_i .
\end{equation}
 {\bf (Proving  \eqref{eq-convexrm})}.  First note that $\rho$ has the dual representation (\cite{FS16})
\begin{align} \label{eq-convexrm-1}
 \rho(Y) &  =  \sup_{Z\in \mathcal Z_\rho} \E[YZ]   ,
\end{align}
where $ \mathcal Z_\rho =\{ Z\ge 0: \E[Z]=1, \E[ZX]\le \rho(X)~{\rm for ~all}~X\}$.
It follows that
\begin{align*}
 \rho(Y)
 &  = \sup_{Z\in \mathcal Z_\rho} \left\{\frac1N\sum_{i=1}^N y_i\E[Z|Y=y_i]  \right\}  = \sup_{Z\in \mathcal Z_\rho} \left\{\frac1N\sum_{i=1}^N y_i z_i  \right\} = \sup_{z\in A_\rho^*}  \left\{\frac1N\sum_{i=1}^N y_iz_i  \right\},
\end{align*}
where  $z_i =\E^{\p}[Z|Y=y_i]$, $i=1,\ldots,N$,  and  $$A_\rho ^*=\{z\in \R_+^N: \E[Z|Y=y_i]=z_i, Z\in \mathcal Z_\rho \}.$$
It remains to show that $A_\rho^*=A_\rho$. 
 Obviously, $A_\rho \subseteq A_\rho^*$.
To see the other set inclusion, take any $z\in A_\rho^*$, i.e., there exists  $Z\in \mathcal Z_\rho$  such that $\E[Z|Y=y_i]=z_i$, $i=1,\ldots,N$.
Define  $\Delta_z =  \sum_{i=1}^Nz_i \id_{\{Y=y_i\}} $. Noting $\E[\Delta_z|Y=y_i] =z_i$, $i=1,\ldots,N$, it suffices to show $\Delta_z\in \mathcal Z_\rho$.
 For any $X$,  define $\Delta_{X} =\frac1N  \sum_{i=1}^N  {x_i}\id_{\{Y=y_i\}}$ with $x_i= \E[ X|Y=y_i] $, $i=1,\ldots,N$.  By Jensen inequality, we have for any convex function $u$,  $\E u(\Delta_X) \le \E u(X)$, that is, $\Delta_X\le_{\rm cx} X$. By Lemma \ref{lem:a1}, we have 
\begin{align*}
\rho(X) \ge \rho (\Delta_X) & \ge \E[\Delta_X Z] \\
  & = \frac1N\sum_{i=1}^N x_i\E [ Z |Y=y_i ]  \\
  & = \frac1N\sum_{i=1}^N x_i z_i \\
  & = \frac1N\sum_{i=1}^N z_i\E [ X |Y=y_i ]   = \E[\Delta_z X],
\end{align*}
where the second inequality follows from  $Z\in \mathcal Z_\rho$.  Hence, we have $\Delta_z\in \mathcal Z_\rho$, and thus, $z\in A_\rho$.   Therefore, $A_\rho = A_\rho^*$, and thus, \eqref{eq-convexrm} holds.

Using \eqref{eq-convexrm}, we can write the problem \eqref{eq:main-eq1-7} as
\begin{eqnarray}
\sup_{y_1,...,y_N}  && \frac{1}{N} \sum_{i=1}^N \ell(y_i), \nonumber \\
{\rm subject\;to}     &&  \sup_{z\in A_\rho} \left\{\frac1N\sum_{i=1}^N \|  \widehat{\xi}_i - y_i \| z_i \right\} \leq \epsilon, \label{eqn11}\\
                              && y_i \in \Xi, \;\;i=1,...,N, \nonumber
\end{eqnarray}
and its Lagrange dual by
\begin{align}
  & \inf_{\lambda \geq 0} \sup_{y_i \in \Xi} \frac{1}{N} \sum_{i=1}^N \ell(y_i) + \lambda \left(
  \epsilon -  \sup_{z\in A_\rho} \left\{\frac1N\sum_{i=1}^N \|  \widehat{\xi}_i - y_i \|  z_i  \right\}  \right) \nonumber \\
  = &  \inf_{\lambda \geq 0} \sup_{y_i \in \Xi} \inf_{z\in  A_\rho}  \frac{1}{N} \sum_{i=1}^N \ell(y_i) + \lambda \left( \epsilon -   \left\{\frac1N\sum_{i=1}^N \|  \widehat{\xi}_i - y_i \|  z_i  \right\}  \right) \nonumber \\
  = &  \inf_{\lambda \geq 0, z \in A_\rho} \sup_{y_i \in \Xi}  \lambda \epsilon  + \frac{1}{N} \sum_{i=1}^N \left(  \ell(y_i) - \lambda \|  \widehat{\xi}_i - y_i \|  z_i \right) \nonumber \\
  = & \inf_{\lambda \geq 0, z \in A_\rho}  \lambda \epsilon   + \frac{1}{N} \sum_{i=1}^N \sup_{y\in \Xi}  \left(  \ell(y) - \lambda \|  \widehat{\xi}_i - y \|  z_i \right), \label{lambda_eq}
\end{align}
where the second equality follows Sion's minimax theorem, given that the objective function is convex in $z$ and concave in $y_i$, and the set $A_\rho$ is a compact set. Strong duality holds for $\epsilon >0$,  because the slater condition can be satisfied.  Strong duality also holds for $\epsilon = 0$, because both the primal and the dual reduces to  $\frac{1}{N} \sum_{i=1}^N \ell(\widehat{\xi}_i)$.

Taking $p = \lambda z$, we can write the above problem equivalently as
\begin{align*}
   & \inf_{\lambda \geq 0,  \frac{p}{\lambda} \in A_\rho}  \lambda \epsilon + \frac{1}{N} \sum_{i=1}^N \sup_{y\in \Xi}  \left(  \ell(y) - \|  \widehat{\xi}_i - y \|  p_i \right) \\
= & \left\{ \begin{array}{l}
\inf_{\lambda \geq 0, \frac{p}{\lambda} \in A_\rho, s_i}\;\;\; \lambda \epsilon + \frac{1}{N} \sum_{i=1}^N  s_i \\
{\rm subject\;to}\;\;\; \sup_{y\in \Xi}  \left(  \ell(y) - \|  \widehat{\xi}_i - y \|  p_i \right) \leq s_i,\;\;i=1,...,N
\end{array}
 \right\}.
\end{align*}
In the case $\lambda =0$, $\frac{p}{0}$ is defined as an infeasible solution for any $p \neq {\bf 0}$, denoted by $\frac{p}{0} \notin  A_{\rho}$, and $\frac{\bf{0}}{0}$ is defined as a feasible solution, denoted by $\frac{\bf{0}}{0} \in A_{\rho}$. One can easily verify that this definition is consistent with the fact that the problem \eqref{lambda_eq} reduces to $\frac{1}{N} \sum_{i=1}^N \sup_{y\in \Xi} \ell(y)$ when $\lambda =0$.

Let $\chi_{\Xi}$ denote the characteristic function of $\Xi$ and $f(y):= \|  \widehat{\xi}_i - y \|  p_i $. We have from the Fenchel duality that the left-hand-side of the above constraint, i.e. $\sup_{y}  \left(  [\ell(y) - \chi_{\Xi}(y)] - f(y) \right)$, can be replaced by its dual and thus the constraint can be equivalently written as
\begin{align*}
     &\inf_{u_i}\; [-\ell + \chi_{\Xi} ]^*(u_i) + f^*(-u_i) \leq s_i, \;\; i=1,...,N\\
\Leftrightarrow & \exists u_i,\;\;   [-\ell + \chi_{\Xi} ]^*(u_i) + f^*(-u_i) \leq s_i, \;\; i=1,...,N,
\end{align*}
where
$$
[-\ell + \chi_{\Xi} ]^*(u_i) = \inf_{\nu}\;\; [-\ell]^*(u_i-\nu) + \sigma_{\Xi}(\nu)
$$
and
$$
f^*(-u_i) = \begin{cases}
 -\widehat{\xi}_i ^\top u_i, &\| u_i\| _* \leq p_i, \\
\infty, &{\rm o.w.}.
\end{cases}
$$
That is, the constraint can also be equivalently written as
$$ \exists u_i, \nu_i, ~~ [-\ell]^*(u_i-\nu_i) + \sigma_{\Xi}(\nu_i) - \widehat{\xi}_i ^\top u_i \leq s_i,
~~ \| u_i\| _* \leq p_i.
\;\; i=1,...,N.$$
Finally, the fact that $A_\rho$ must be a subset of a probability simplex, i.e. satisfying $\frac{p}{\lambda} \geq 0$ and $\vec{1}^\top \frac{p}{\lambda} = 1$ implies that $\sum_{i=1}^N p_i = \lambda$ and $p_i \geq 0$ must hold.

\qed

~~

\subsection{Proofs of Section  \ref{sec:4convex}}
\noindent{\bf Proof of  Theorem \ref{th:main-2}.}
Denote by $\Theta_j =\max\{x\in\R^m: \ell_j(x) = \ell(x)\}$, $j=1,\ldots,K$.  Without loss of generality assume that $\Theta_j$, $j=1,\ldots,K$, are disjoint. Similar as Theorem \ref{thm1}, we only show that  the problem \eqref{eq:gene-2-31} is also an upper bound of the problem \eqref{eq:main}.
For any $\Pi\in\Pi(\widehat{\mathbb P}_N,\mathbb P)$  satisfying $\Pi\in\mathcal M(\Xi^2)$ and $ \rho^{\Pi}(\|\widehat{\xi}-\xi\|)\leq \epsilon$, define a new joint distribution $\Pi^*$ as
 \begin{align}
   \label{eq:Qstar-general}
 \Pi^*= \frac1N\sum_{i=1}^N\sum_{j=1}^K p_{ij}\delta_{(\widehat{\xi}_i,\xi_{ij})}, ~~{\rm with}~~ p_{ij}=\mathbb Q_i(\Theta_j),~~  \xi_{ij}
 =\E^{\mathbb Q_i}[\xi|\Theta_j] =  \E^{\Pi}[\xi|\Theta_j, \widehat{\xi}=\widehat{\xi}_i],
 \end{align}
 where $\mathbb Q_i$ is the conditional distribution of $\xi$ given $\widehat{\xi} =\widehat{\xi}_i$ when the joint distribution of $(\widehat{\xi},\xi )$ is  ${\Pi}$.
   Let $ (\widehat{\xi}^*, {\xi}^*)\in\R^m\times\R^m$ be a random vector having the distribution $\Pi^*$.  Denote by $\mathbb Q_i^*$   the marginal distributions of  $\xi^*$ conditional on $\widehat{\xi}=\widehat{\xi}_i$, $i=1,\ldots,N$.
  \begin{itemize}
\item [(i)]  We have $\widehat{\xi}^*\sim \widehat{\mathbb P}_N$, that is, $\Pi^*\in\Pi(\widehat{\mathbb P}_N,\mathbb P^*)$. By the convexity of $\Xi$, we have $\xi_{ij}\in \Xi$, $i=1,\ldots,N$, $j=1,\ldots,K$, and thus,
    $\Pi^*\in\mathcal M(\Xi^2)$.
\item [(ii)]Similar as (ii) in the proof of Theorem \ref{thm1}, for any increasing convex function $u$, define $v:=u\circ c$ is a convex function, where $c(x,y)=\|x-y\| $.
    Hence, we have
\begin{align*}
 \E^{\Pi^*}[u( \|\widehat{\xi}^*-\xi^*\|)]
 &  =  \E^{\Pi^*}[v(\widehat{\xi}^*,\xi^* )]\\
 &  = \frac1N\sum_{i=1}^N\sum_{j=1}^K p_{ij} v(\widehat{\xi}_i,\xi_{ij}) \\
 &  =  \frac1N\sum_{i=1}^N\sum_{j=1}^K p_{ij} v(\widehat{\xi}_i,\E^{\Pi}[\xi|\Theta_j, \widehat{\xi}=\widehat{\xi}_i])  \\
 &  =  \frac1N\sum_{i=1}^N\sum_{j=1}^K p_{ij} v( \E^{\Pi}[(\widehat{\xi},\xi)|\Theta_j, \widehat{\xi}=\widehat{\xi}_i])  \\
 & \le \frac1N \sum_{i=1}^N\sum_{j=1}^K p_{ij} \E^{\Pi} [v(\widehat{\xi}, \xi )|\Theta_j, \widehat{\xi}=\widehat{\xi}_i]   \\
 &  =  \E^{\Pi}[v(\widehat{\xi},\xi )]=\E^{\Pi}[u( \|\widehat{\xi}-\xi\|)] ,
\end{align*}
where the inequality follows from the Jensen inequality.
Therefore, we have $\rho^{\Pi^*}(\|\widehat{\xi}^*-\xi^*\|)\le \epsilon$.
\item [(iii)]
 We have
 \begin{align*}
\E^{\mathbb Q_i}[\ell(\xi)]
& = \sum_{j=1}^K  \int_{\Theta_j} \ell_j(\xi) d{ \mathbb Q_i}=\sum_{j=1}^K  p_{ij} \ell_j(\xi_{ij})  = \sum_{j=1}^K  \int_{\Theta_j} \ell_j(\xi) d {\mathbb Q_i^*} \le \sum_{j=1}^K  \int_{\Theta_j} \ell(\xi) d {\mathbb Q_i^*}= \E^{\mathbb Q_i^*}[\ell(\xi)] ,
\end{align*}
where the second equality follows from the linearity of $\ell_j$, and the inequality follows from  $\ell_j\le \ell$, $j=1,\ldots,K.$
Therefore, we have
 \begin{align*}
\E^{\Pi}[\ell(\xi)]=\frac1N\sum_{i=1}^N \E^{\mathbb Q_i}[\ell(\xi)]
 \le \frac1N\sum_{i=1}^N  \E^{\mathbb Q_i^*}[\ell(\xi)]= \E^{\Pi^*}[\ell(\xi)].
\end{align*}
\end{itemize}
Combining the above three facts, we have that the   problem \eqref{eq:main} is equivalent to 
\begin{align}\label{eq:gene-31-1}
\sup_{p_{ij},x_{ij}}~~ &\frac1N\sum_{i=1}^N\sum_{j=1}^K p_{ij} \ell(\xi_{ij}) \\
{\rm subject~ to}~~& \rho^{\Pi^*}(\|\widehat{\xi}-\xi\|) \le \epsilon,\nonumber\\
& \Pi^*(\{(\widehat{\xi}_i,\xi_{ij})\})=p_{ij}\ge 0,~\xi_{ij}\in\Xi, ~\forall i,j,\nonumber\\
& \sum_{j=1}^K p_{ij} =1,~~i=1,\ldots,N.\nonumber
\end{align}
Note that we have shown that the original optimization problem \eqref{eq:main} is bounded by the problem \eqref{eq:gene-2-31} which is again bounded by the problem \eqref{eq:gene-31-1} as $\ell_i\le \ell$ for $i=1,\ldots,K$. Therefore, by the equivalence between the problems \eqref{eq:main} and \eqref{eq:gene-31-1}, we have the optimization problem \eqref{eq:main} is equivalent to the problem \eqref{eq:gene-2-31}. This completes the proof.
\qed

~~

 \noindent{\bf Proof of Theorem \ref{main_cvx}.}
We first consider the case that the objective function $\ell$ is a piecewise linear function, that is,   $\ell=\max_{k=1,\ldots,K} \ell_k$, where  $\ell_k(x)=a_k^\top x+b_k$, $k=1,\ldots,K$. Without loss of generality, assume   $$\|a_{1}\|_*\le \cdots\le \|a_{K}\|_* ~~{\rm and ~denote ~by} ~~z_{ij}:=\ell_j(\widehat{\xi}_i)=a_j^\top \widehat{\xi}_i+b_j,~\forall~i,j.$$  By Theorem \ref{th:main-2}, we have the optimization problem  \eqref{wcvar} is equivalent to 
\begin{align}\label{eq:gene-2}
\sup_{t,\;p_{ij}\in\R,\;x_{ij}\in\R^m }~~ &\frac1N\sum_{i=1}^N\sum_{j=1}^K p_{ij} (a_j^\top (x_{ij} +\widehat{\xi}_i)+b_j) \\
{\rm subject~ to}~~& t+\frac1{1-\alpha}\frac1N\sum_{i=1}^N\sum_{j=1}^K p_{ij}(\|x_{ij}\| -t)_+ \le \epsilon,\nonumber\\
& \sum_{j=1}^Kp_{ij} =1,~~i=1,\ldots,N,~~p_{ij}\ge 0, ~\forall ~i,j.\nonumber
\end{align}
Note that for any $x_{ij}\in \R^m $, by taking $\widehat{x}_{j}  = \arg\max_{x\in\R^m :\|x\| =1}  a_j^\top x $ and $x_{ij}^* = \| x_{ij}\| \widehat{x}_{j}$, one can verify that $\|x_{ij}^*\| =\|x_{ij}\| $ and
$$
a_j^\top x_{ij}^*  =\| x_{ij}\| a_j^\top  \widehat{x}_{j} =\| x_{ij}\| \|a_j\|_*\ge  a_j^\top  x_{ij}.
$$
Therefore, it suffices to consider the $x_{ij}$ that takes the form of $x_{ij}=y_{ij} \widehat{x}_j$, $y_{ij}\in\R_+$. By taking $x_{ij}=y_{ij} \widehat{x}_j$,  we have the   problem  \eqref{eq:gene-2} is equivalent to
 \begin{align}
\sup_{ t,p_{ij},y_{ij}\in\R_+}~~ &
\frac1N\sum_{i=1}^N\sum_{j=1}^Ky_{ij}\|a_j\|_* +\frac1N\sum_{i=1}^N\sum_{j=1}^K p_{ij} z_{ij} \label{eq-220311-1}\\
{\rm subject~ to}~~& t+\frac1{1-\alpha}\frac1N\sum_{i=1}^N\sum_{j=1}^K(y_{ij}-t p_{ij})_+ \le\epsilon,\label{eq-220311-2}\\
& \sum_{j=1}^Kp_{ij} =1,~~i=1,\ldots,N.\label{eq-220311-3}
\end{align}
For any $p_{ij}\ge 0$ and $y_{ij}\ge0$ satisfying  conditions \eqref{eq-220311-2} and \eqref{eq-220311-3},
if  for some $i$, there exists $j\neq K$ such that $y_{ij} > t p_{ij}$, then by taking $y_{ij}^*:= t p_{ij} $ and $y_{iK}^*:=y_{iK}+y_{ij}-t p_{ij}> y_{iK}$, one can verify that
$$t+\frac1{1-\alpha}\frac1N\sum_{i=1}^N\sum_{j=1}^K (y_{ij}^* -t p_{ij})_+ \le \epsilon$$
and the objective function becomes larger with $y_{ij}^*$. Therefore, it suffices to consider the case that  
$$y_{ij} \le t p_{ij},~~j\neq K ,~~~ t+\frac1{1-\alpha}\frac1N\sum_{i=1}^N (y_{iK } -t p_{iK })_+ \le \epsilon.$$
In this case, we have the problem \eqref{eq-220311-1} is equivalent to
 \begin{align*}
\sup_{t,p_{ij},y_{ij}\in\R_+}~~ &
\frac1N\sum_{i=1}^N\sum_{j=1}^K y_{ij}\|a_j\|_* +\frac1N\sum_{i=1}^N\sum_{j=1}^K p_{ij} z_{ij}~~{\rm with}~~z_{ij}=(a_j^\top \widehat{\xi}_i+b_j)  \\
{\rm subject~ to}~~&y_{ij} \le tp_{ij},~~j< K ,~~~ t+\frac1{1-\alpha}\frac1N\sum_{i=1}^N (y_{iK } -t p_{iK })_+ \le \epsilon  ~\text{and }\eqref{eq-220311-3},
\end{align*}
that is,
 \begin{align*}
\sup_{t,p_{ij},y_{ij}\in\R_+}~~ &
\frac1N\sum_{i=1}^N \left( \sum_{j=1}^{ K-1 }t p_{ij}\|a_j\|_* + y_{iK }\|a_{K }\|_*\right) +\frac1N\sum_{i=1}^N\sum_{j=1}^K p_{ij} z_{ij}  \\
{\rm subject~ to}~~& t +\frac1{1-\alpha}\frac1N\sum_{i=1}^N (y_{iK } -t p_{iK })_+ \le \epsilon   ~\text{and }\eqref{eq-220311-3}.
\end{align*}
Take $y_{iK }= x_{i}p_{iK }$, $i=1,\ldots,N$. We can rewrite it as
 \begin{align*}
\sup_{t,p_{ij},x_i\in\R_+}~~ &
\frac1N\sum_{i=1}^N \left( \sum_{j=1}^{ K-1 }t p_{ij}\|a_j\|_*+ x_i p_{iK }\|a_{K }\|_*\right) +\frac1N\sum_{i=1}^N\sum_{j=1}^K p_{ij} z_{ij} \\
{\rm subject~ to}~~& t+\frac1{1-\alpha}\frac1N\sum_{i=1}^N p_{iK }(x_{i} -t)_+ \le \epsilon    ~\text{and }\eqref{eq-220311-3}.
\end{align*}
For any feasible solution $x_i,p_{ij},t$, $i=1,\ldots,N$, $j=1,\ldots,K$, define $x_{i}^*:=\sum_{i=1}^N p_{iK } x_{i}/(\sum_{i=1}^N p_{iK })=:\overline{x}$, $i=1,\ldots,N$. By Jensen inequality, we have
$t +\frac1{1-\alpha}\frac1N\sum_{i=1}^N p_{iK }(\overline{x}-t)_+ \le\epsilon$ and
the objective function remains unchanged. Therefore,
the above optimization problem  is equivalent to
 \begin{align*}
\sup_{t,p_{ij},\overline{x}\in\R_+}~~ &
t \frac1N\sum_{i=1}^N    \sum_{j\neq K }p_{ij}\|a_j\|_*+\overline{x}  \|a_{K }\|_* \frac1N\sum_{i=1}^Np_{iK } +\frac1N\sum_{i=1}^N\sum_{j=1}^K p_{ij} z_{ij} \\
{\rm subject~ to}~~& t+\frac1{1-\alpha}\frac1N\sum_{i=1}^N p_{iK }(\overline{x} -t)_+ \le \epsilon    ~\text{and }\eqref{eq-220311-3}.
\end{align*}
Noting that the objective function is increasing in $\overline{x}$, we have $\overline{x}\ge x$ holds automatically.
 Taking $\delta=\overline{x}-t$, e
 the above optimization problem  is equivalent to
 \begin{align*}
\sup_{t,p_{ij},\overline{x}\in\R_+}~~ &
t \frac1N\sum_{i=1}^N  \sum_{j=1}^Kp_{ij}\|a_j\|_*  +
\delta \|a_{K }\|_*\frac1N \sum_{i=1}^Np_{iK } +\frac1N\sum_{i=1}^N\sum_{j=1}^K p_{ij} z_{ij} \\
{\rm subject~ to}~~& t +\frac1{1-\alpha}\frac1N\sum_{i=1}^N p_{iK }\delta \le \epsilon    ~\text{and }\eqref{eq-220311-3}.
\end{align*}
Denote by $u= \frac1{1-\alpha}\frac1N\sum_{i=1}^N p_{iK }\delta$. We have the problem \eqref{eq:gene-2} is equivalent to
\begin{align*}
\sup_{t,p_{ij},u\in\R_+}~~ &
t \frac1N\sum_{i=1}^N  \sum_{j=1}^Kp_{ij}\|a_j\|_*  +
 \|a_{K }\|_*(1-\alpha) u+\frac1N\sum_{i=1}^N\sum_{j=1}^K p_{ij} z_{ij} \\
{\rm subject~ to}~~& t+u \le \epsilon    ~\text{and }\eqref{eq-220311-3}.
\end{align*}
Since the objective function is increasing in both $t$ and $u$,  the constraint $ t+u \le \epsilon$  can be replaced by $t+u=\epsilon$ without loss of generality. 
 Letting $u=\epsilon -t$, we have it is equivalent to
  \begin{align*}
\sup_{t,p_{ij}}~~ &
 \frac1N\sum_{i=1}^N  \sum_{j=1}^Kp_{ij}(\|a_j\|_* t +z_{ij} ) - t \|a_{K }\|_*(1-\alpha) + \|a_{K }\|_*(1-\alpha)\epsilon~~{\rm subject~ to}~~  \eqref{eq-220311-3}\\
& =\sup_{t\in [0,\epsilon]}\frac1N\sum_{i=1}^N \max_{j}(\|a_j\|_*t +z_{ij} ) - t \|a_{K }\|_*(1-\alpha) + \|a_{K }\|_*(1-\alpha)\epsilon.\end{align*}
Note that the objective function is convex in $t$. We therefore have that the optimal $t$ is either $0$ or $\epsilon$, and the optimal value of the problem  \eqref{eq:gene-2}  is
\begin{align*}
 \lefteqn{\max\left\{\frac1N\sum_{i=1}^N \max_{j}\{\|a_j\|_*\epsilon +z_{ij} \},~\frac1N\sum_{i=1}^N \max_{j} z_{ij}  + \|a_{K }\|_*(1-\alpha)\epsilon  \right\}}\\
 & =\max\left\{\frac1N\sum_{i=1}^N  \max_{e: \|e\|=1} \ell(\widehat{\xi}_i+\epsilon e),~\frac1N\sum_{i=1}^N \ell(\widehat{\xi}_i) + \|a_{K }\|_*(1-\alpha)\epsilon  \right\}.
\end{align*}
  Moreover, one can verify the following statements about the worst-case distribution.
\begin{itemize}
\item [(1)] If $\frac1N\sum_{i=1}^N  \max_{e: \|e\| =1} \ell(\widehat{\xi}_i+\epsilon e)\ge \frac1N\sum_{i=1}^N \ell(\widehat{\xi}_i) + L (1-\alpha)\epsilon$, then the optimal solution to the problem  \eqref{eq:gene-2}  is  $t^*=\epsilon$,
$$
 x_{ij}^* = \widehat{\xi}_i + \epsilon e^*_i ~~{\rm with}~~e^*_i=\arg \max_{e: \|e\| =1} \ell(\widehat{\xi}_i+\epsilon e),~~~p_{ij_i}^*=1 ~~{\rm for}~j_i ~{\rm such ~that}~ e^*_i = \frac{a_{j_i}}{\|a_{j_i}\| }.
$$
\item [(2)] If $\frac1N\sum_{i=1}^N \ell(\widehat{\xi}_i) + \|a_{K }\|_*(1-\alpha)\epsilon > \frac1N\sum_{i=1}^N  \max_{e: \|e\| =1} \ell(\widehat{\xi}_i+\epsilon e)$, then the optimal value can be approached by the following feasible solutions:  $t^*=0$,
    $$
     x_{ij}=\widehat{\xi}_i~~{\rm for}~j\neq K ,~~x_{iK }  = \widehat{\xi}_i  + (1-\alpha) \frac{ a_{K }}{\|a_{K }\| }\frac{ \epsilon}{ p_{iK }},~~p_{iK } \downarrow 0.
    $$
\end{itemize}

Now, consider the general convex function $\ell$.
 By assumption on $\ell$, there exist  $\ell_n,\ell_n^*$, $n\in\N$ such that $\ell_n\le \ell\le \ell_n^*$, $\ell_n$ and $\ell_n^*$ are piecewise linear   functions,  both $\ell_n$ and $\ell_n^*$ converge to $\ell$, and  $ \lim_{n\to\infty} \overline{\|\nabla \ell_n\|_*}=\lim_{n\to\infty}  \overline{\|\nabla \ell_n^*\|_*}=L$.   Denote by $\overline{\ell}$ the worst-case value of the problem  \eqref{wcvar}.
By the result for piecewise linear functions, we have
\begin{align*}
\lefteqn{ \max\left\{\frac1N\sum_{i=1}^N  \max_{e: \|e\| =1} \ell_n(\widehat{\xi}_i+\epsilon e),~\frac1N\sum_{i=1}^N \ell_n(\widehat{\xi}_i)  +  \overline{\|\nabla \ell_n\|_*}  (1-\alpha)\epsilon  \right\} }\\
& \le   \overline{\ell} \le
\max\left\{\frac1N\sum_{i=1}^N  \max_{e: \|e\|=1} \ell_n^*(\widehat{\xi}_i+\epsilon e),~\frac1N\sum_{i=1}^N\ell_n^*(\widehat{\xi}_i)   + \overline{\|\nabla \ell_n^*\|_*} (1-\alpha)\epsilon  \right\}.
\end{align*}
Letting $n\to\infty$, we obtain that
$$
\overline{\ell} =
\max\left\{\frac1N\sum_{i=1}^N  \max_{e: \|e\|=1} \ell(\widehat{\xi}_i+\epsilon e),~\frac1N\sum_{i=1}^N \ell(\widehat{\xi}_i) + L(1-\alpha)\epsilon  \right\}.
$$
 Moreover, we have the following statements about the worst-case distribution.
\begin{itemize}
\item [(1)] If $\frac1N\sum_{i=1}^N  \max_{e: \|e\|=1} \ell(\widehat{\xi}_i+\epsilon e)\ge \frac1N\sum_{i=1}^N \ell(\widehat{\xi}_i) + L(1-\alpha)\epsilon$, then the optimal solution to the problem  \eqref{eq:gene-2}  is  $x^*=\epsilon$,
$$
 {\xi}_i^* = \widehat{\xi}_i + \epsilon e^*_i ~~{\rm with}~~e^*_i=\arg \max_{e: \|e\|=1} \ell(\widehat{\xi}_i+\epsilon e) ~~{\mathbb P}_i\textrm{-a.s.},~~i=1,\ldots,N.
$$
\item [(2)] If $\frac1N\sum_{i=1}^N \ell(\widehat{\xi}_i) +L(1-\alpha)\epsilon > \frac1N\sum_{i=1}^N  \max_{e: \|e\|=1} \ell(\widehat{\xi}_i+\epsilon e)$, then the optimal value can be approached by the following feasible solutions:  $x^*=0$,
    $$
      {\xi}_i^*=\widehat{\xi}_i, ~{\rm for}~i=2,\ldots,N,~~\mathbb P_1(\xi_1^*  = \widehat{\xi}_1  + e_n )= \frac{(1-\alpha)\epsilon}{\|e_n\| },~~\mathbb P_1(\xi_1^*  = \widehat{\xi}_1 )=1-\frac{(1-\alpha)\epsilon}{\|e_n\| },~~n\in\N.
    $$
    where $e_n\in\R^m $, $n\in\N$, satisfy   $\lim_{n\to\infty} \frac{\ell(\widehat{\xi}_1+ e_n)- \ell(\xi_1)}{ \|e_n\| }=L $.
\end{itemize}

We thus complete the proof.
\qed

~~

\noindent{\bf Proof of Corollary \ref{wcdis}.}
Since  $\frac1N\sum_{i=1}^N  \max_{e: \|e\|=1} \ell(\widehat{\xi}_i+\epsilon e) =\frac1N\sum_{i=1}^N \max_{\|e\|\le 1} \ell(\widehat{\xi}_i+\epsilon e) > \frac1N\sum_{i=1}^N \ell(\widehat{\xi}_i)$ and $L(1-\alpha)\epsilon$ in \eqref{wass_exp} is strictly decreasing in $\alpha\in [0,1]$, there must exist $\alpha \in (0,1)$ large enough such that
\begin{equation*}
\sup _{\mathbb{P} \in \mathbb{B}_{\varepsilon}^{{\rm wc}} (\widehat{\mathbb{P}}_{N} )}\mathbb{E}^{\mathbb{P}}[\ell(\xi)] = \frac1N\sum_{i=1}^N  \max_{e: \|e\|=1} \ell(\widehat{\xi}_i+\epsilon e)
 > \frac1N\sum_{i=1}^N \ell(\widehat{\xi}_i) +  L(1-\alpha)\epsilon = \sup _{\mathbb{P} \in \mathbb{B}^{\rm W}_{(1-\alpha)\varepsilon} (\widehat{\mathbb{P}}_{N} )} \mathbb{E}^{\mathbb{P}}[\ell(\xi)].
\end{equation*}
In the case,  the optimal value  \eqref{connection} is $\frac1N\sum_{i=1}^N  \max_{e: \|e\| =1} \ell(\widehat{\xi}_i+\epsilon e)$ which is attainable by the distribution $\mathbb P^* = \frac1N\sum_{i=1}^N \delta_{\widehat{\xi}_i+\epsilon e_i^*}$, where $e_i^*=\arg\max_{\{e\in\R^m : \|e\| =1\}}\ell(\widehat{\xi}_i+\epsilon e)$, $i=1,\ldots,N$. This completes the proof.
\qed

~~

The following result shows that the worst-case expectation problem formulated based on the Wasserstein ambiguity sets $\mathbb{B}_{\varepsilon} (\widehat{\mathbb{P}}_{N} )$ is not attainable.
\begin{proposition}\label{thm:7}
In the case where the loss function $\ell$ is a convex function satisfying $L:=\sup_{x\in\R^m}\|\partial \ell(x)\|_*$ $<\infty$,  and the sample satisfies $\|\partial \ell(\widehat{\xi}_i)\|_*<L$, $i=1,\ldots,N$, the worst-case expectation problem
\begin{equation} \label{wasser}
\sup _{\mathbb{P} \in \mathbb{B}^{\rm W}_{\varepsilon} (\widehat{\mathbb{P}}_{N} )} \mathbb{E}^{\mathbb{P}}[\ell(\xi)]
\end{equation}
is not attainable, i.e. the worst-case distribution does not exist.
\end{proposition}
\begin{proof}  By Theorem  \ref{main_cvx},   the optimal value of the problem  \eqref{wasser} is $\frac1N\sum_{i=1}^N \ell(\widehat{\xi}_i) +L\epsilon$.
We can rewrite the problem \eqref{wasser} as
\begin{align}
\sup_{\mathbb Q_1,\ldots,\mathbb Q_N}~~ & \frac1N\sum_{i=1}^N \E^{\mathbb Q_i}[\ell(\widehat{\xi}_i +\xi_i)]\label{eq:gene-2-expectation} \\
{\rm subject~ to}~~&  \frac1N\sum_{i=1}^N \E^{\mathbb Q_i}[\|\xi_i\| ] = \epsilon,\label{eq:gene-2-expectation-1}
\end{align}
where $\mathbb Q_i$ is the distribution of $\xi_i$, $i=1,\ldots,N$.
By  the convexity of $\ell$ and the definition of $L$, we have for any   $\mathbb Q_1,\ldots,\mathbb Q_N$,  there exist random variables $y_1,\ldots,y_N$ such that
 $$
 \ell(\widehat{\xi}_i +\xi_i)
 = \ell(\widehat{\xi}_i) + \partial  \ell(y_i)^\top \xi_i  \le \ell(\widehat{\xi}_i) +\|\partial  \ell(y_i)\|_*\| \xi_i \| \le  
\ell(\widehat{\xi}_i) +L\| \xi_i \| ~~\mathbb Q_i~{\rm a.s.}.
 $$ 
 Therefore, for any $\mathbb Q_1,\ldots,\mathbb Q_N$ satisfying \eqref{eq:gene-2-expectation-1}, we have $$
 \E^{\mathbb Q_i}[\ell(\widehat{\xi}_i +\xi_i)]
 \le \ell(\widehat{\xi}_i) + L\E^{\mathbb Q_i}[ \| \xi_i \| ]
 $$ and the inequality reduces to equality only if   $\ell(\widehat{\xi}_i +\xi_i) = \ell(\widehat{\xi}_i) +L\| \xi_i \|$ a.s. $\mathbb Q_i$, $i=1,\ldots,N$.
Therefore, we have
\begin{align*}
 \frac1N\sum_{i=1}^N \E^{\mathbb Q_i}[\ell(\widehat{\xi}_i +\xi_i)]
 & \le \frac1N\sum_{i=1}^N \ell(\widehat{\xi}_i) +L\frac1N\sum_{i=1}^N \E^{\mathbb Q_i}[\| \xi_i \| ]= \frac1N\sum_{i=1}^N \ell(\widehat{\xi}_i) +L\epsilon,
\end{align*}
and the first inequality reduces to equality only if  $\ell(\widehat{\xi}_i +\xi_i) = \ell(\widehat{\xi}_i) +L\| \xi_i \|$ a.s. $\mathbb Q_i$, $i=1,\ldots,N$. By the assumption that the sample satisfies $\|\partial  \ell(\widehat{\xi}_i)\|_*<L$, $i=1,\ldots,N$, we have
$\ell(\widehat{\xi}_i +\xi_i) = \ell(\widehat{\xi}_i) +L\| \xi_i \|$ a.s. $\mathbb Q_i$, $i=1,\ldots,N$, can not happen simultaneously. Therefore, we have
 \begin{align*}
 \frac1N\sum_{i=1}^N \E^{\mathbb Q_i}[\ell(\widehat{\xi}_i +\xi_i)]
<\frac1N\sum_{i=1}^N \ell(\widehat{\xi}_i) +L\epsilon,
\end{align*}
 which means that   $(\mathbb Q_1,\ldots,\mathbb Q_N)$ is not the optimal solution and thus, any feasible solution  is not the optimal solution.
\end{proof}

~~

\noindent{\bf Proof of Corollary \ref{reduced1}.}
By assumption, we have that  $f_1$  is a non-increasing convex function on $\R$, $f_2$ is non-decreasing convex function on $\R$, and one can verify that the value of \eqref{worst_exp} with $ \ell(\xi)=f(x^\top \xi)$ is
 $$
  \max_{\|e\|=1}f(x^\top (\widehat{\xi}_i+\epsilon e)) = \max\left\{ f_1 (x^\top  \widehat{\xi}_i - \epsilon\|x\|_{*}), ~f_2 (x^\top  \widehat{\xi}_i +\epsilon\|x\|_{*} ) \right\},
 $$
 which is a convex function in $x\in\R^m$.
 Further, we have the value of \eqref{wass_exp} with $ \ell(\xi)=f(x^\top \xi)$ is
 $$ \sup _{\mathbb{P} \in \mathbb{B}^{\rm W}_{(1-\alpha)\varepsilon} (\widehat{\mathbb{P}}_{N} )} \mathbb{E}^{\mathbb{P}}[f(x^\top \xi)] = \frac1N\sum_{i=1}^N f(x^\top \widehat{\xi}_i) +  {\rm Lip}(f)\|x\|_{*}(1-\alpha)\epsilon.$$
Substituting them to \eqref{connection}, we obtain the result.
\qed

~~

\noindent{\bf Proof of Theorem \ref{ex-thm1}.}
By Theorem \ref{th:main-2}, we have the problem \eqref{wcexptile} is equivalent to
\begin{align}
\sup_{p_{ij},\xi_{ij}}~~
& \frac1N \sum_{i=1}^N \sum_{j=1}^K p_{ij}(a_j^\top \xi_{ij} + b_j)\label{eq:expectile-3} \\
{\rm subject~ to}~~
& e_\alpha^{\Pi}(\|\widehat{\xi} - \xi\|) \le \epsilon,  \nonumber\\
& \Pi((\widehat{\xi},\xi)=(\widehat{\xi}_i,\xi_{ij}))=p_{ij}\ge 0, ~ \xi_{ij}\in  \Xi,~\forall i,j,  \nonumber\\
& \sum_{j=1}^K p_{ij} =1,~~i=1,\ldots,N.  \nonumber
\end{align}
Note that $e_\alpha(X) \le \epsilon$ is equivalent to $e_\alpha(X-\epsilon) \le0$ as $e_\alpha$ satisfies translation invariance. By the monotonicity of $x\mapsto\alpha\E[(X-x)_+]-(1-\alpha)\E[(X-x)_-]$,  we have $ e_\alpha(X-\epsilon) \le0 $  if and only if $\alpha\E (X-\varepsilon)_+ \le  (1-\alpha)\E (\varepsilon-X)_+$, that is,
$ (2\alpha-1) \E (X-\varepsilon)_+ \le  (1-\alpha)(\varepsilon-\E X).$
We have the constraint $e_\alpha^{\Pi}(\|\widehat{\xi} - \xi\|) \le \epsilon$ is equivalent to
$$   \frac1N \sum_{i=1}^{N} \sum_{j=1}^K p_{ij} (\|\xi_{ij} -\widehat{\xi}_i\|-\varepsilon)_+
\le  \frac{1-\alpha}{2\alpha - 1}\left[\varepsilon- \frac1N \sum_{i=1}^{N} \sum_{j=1}^Kp_{ij}\|\xi_{ij}-\widehat{\xi}_i\|\right].$$
Therefore, we have the problem \eqref{eq:expectile-3} is equivalent to
\begin{align}
\sup_{p_{ij},\xi_{ij}}~~ &\frac1N\sum_{i=1}^N\sum_{j=1}^K p_{ij} (a_j^\top \xi_{ij} +b_j) \label{eq:expectile-1} \\
{\rm subject~ to}~~&  \frac1N \sum_{i=1}^{N} \sum_{j=1}^K p_{ij} (\|\xi_{ij} -\widehat{\xi}_i\|-\varepsilon)_+
\le  \frac{1-\alpha}{2\alpha - 1}\left[\varepsilon- \frac1N \sum_{i=1}^{N} \sum_{j=1}^Kp_{ij}\|\xi_{ij}-\widehat{\xi}_i\|\right], \nonumber\\
& \sum_{j=1}^Kp_{ij} =1,~i=1,\ldots,N, ~ ~ \xi_{ij}\in \Xi,~\forall i,j.\nonumber
\end{align}
Substituting 
$y_{ij} =p_{ij} (\xi_{ij}-\widehat{\xi}_i)$ for $i=1,\ldots,N$, $j=1,\ldots,K$, into the  problem \eqref{eq:expectile-1} yields  \eqref{eq:expectile-2} by standard computation.

To derive the alternative minimization formulation, we begin by dualizing the first constraint in  \eqref{eq:expectile-2}
\begin{eqnarray} \label{dual_expectile}
\inf_{\lambda\geq0}\left\{ \begin{array}{ll}
\sup_{p_{ij},y_{ij}} & \frac{1}{N}\sum_{i=1}^{N}\sum_{j=1}^{K}(a_{j}^{\top}y_{ij}+\bar{z}_{ij}p_{ij})+\lambda\left(\beta\varepsilon-\frac{1}{N}\sum_{i=1}^{N}\sum_{j=1}^{K} \left( \left(\| y_{ij}\| -\varepsilon p_{ij}\right)_{+}+\beta\| y_{ij}\| \right) \right)\\
{\rm subject\;to} & \sum_{j=1}^{K}p_{ij}=1,\;i=1,...,N,\;p_{ij}\geq0,\;\widehat{\xi}_{i}+\frac{y_{ij}}{p_{ij}}\in\Xi,\;\;\forall i,j.
\end{array}\right\}
\end{eqnarray}
where $\beta=\frac{1-\alpha}{2\alpha-1}$ and $\bar{z}_{ij} = (a_j^\top \widehat{\xi}_i+b_j)$. Strong duality holds, i.e. \eqref{dual_expectile} = \eqref{eq:expectile-2},
because \eqref{eq:expectile-2} is a convex optimization problem that satisfies the slater condition when $\varepsilon > 0$. Strong duality holds also for $\varepsilon  = 0$, because both the primal and the dual reduces to the same problem. 
Replacing $y_{ij}$ with $p_{ij}(\xi_{ij}-\widehat{\xi}_{i})$, we have
\begin{align*}
 & \inf_{\lambda\geq0}\lambda\beta\varepsilon+\frac{1}{N}\sum_{i=1}^{N}\left\{ \begin{array}{ll}
\sup_{p_{ij},\xi_{ij}} & \sum_{j=1}^{K}p_{ij}\left[\left(a_{j}^{\top}\xi_{ij}+b_{j}\right)-\lambda\left(\left(\| \xi_{ij}-\widehat{\xi}_{i}\| -\varepsilon\right)_{+}+\beta\| \xi_{ij}-\widehat{\xi}_{i}\| \right)\right]\\
{\rm subject\;to} & \sum_{j=1}^{K}p_{ij}=1,\;p_{ij}\geq0,\;\xi_{ij}\in\Xi,\;\;\forall j.
\end{array}\right\} \\
= & \inf_{\lambda\geq0}\lambda\beta\varepsilon+\frac{1}{N}\sum_{i=1}^{N}\left\{ \begin{array}{ll}
\sup_{\xi_{ij}} & \max_{j\in\{1,...,K\}}\left[\left(a_{j}^{\top}\xi_{ij}+b_{j}\right)-\lambda\left(\left(\| \xi_{ij}-\widehat{\xi}_{i}\| -\varepsilon\right)_{+}+\beta\| \xi_{ij}-\widehat{\xi}_{i}\| \right)\right]\\
{\rm subject\;to} & \xi_{ij}\in\Xi,\;\;\forall j.
\end{array}\right\} \\
= & \left\{ \begin{array}{ll}
\inf_{\lambda\geq0,s_{i}} & \lambda\beta\varepsilon+\frac{1}{N}\sum_{i=1}^{N}s_{i}\\
{\rm subject\;to} & \max_{j\in\{1,...,K\}}\left[\sup_{\xi\in\Xi}\left(a_{j}^{\top}\xi+b_{j}\right)-\lambda\left(\left(\| \xi-\widehat{\xi}_{i}\| -\varepsilon\right)_{+}+\beta\| \xi-\widehat{\xi}_{i}\| \right)\right]\le s_{i},\;i=1,...,N
\end{array}\right\}.
\end{align*}
The constraint is equivalent to
\begin{equation} \label{qqq}
\sup_{\xi\in\Xi}\left(a_{j}^{\top}\xi+b_{j}\right)-\lambda\left(\left(\| \xi-\widehat{\xi}_{i}\| -\varepsilon\right)_{+}+\beta\| \xi-\widehat{\xi}_{i}\| \right)\leq s_{i},\;i=1,...,N,\;j=1,...,K.
\end{equation}
Let $f(\xi)=f_{1}(\xi)+f_{2}(\xi)$, where $f_{1}(\xi):=\lambda \left(\| \xi-\widehat{\xi}_{i}\| -\varepsilon\right)_{+}$ and
$f_{2}(\xi)=\lambda \beta\| \xi-\widehat{\xi}_{i}\| $.
Using Fenchel duality, we can substitute the left-hand-side of \eqref{qqq} by its dual and arrive at
\begin{equation} \label{dd2}
\exists u_{ij}: \; \; [-\ell_{j}+\chi_{\Xi}]^{*}(u_{ij})+f^{*}(-u_{ij})\leq s_{i},\;i=1,...,N,\;j=1,...,K,
\end{equation}
where $\ell_j (\xi) = a_j^\top \xi + b_j$ and
\[
[-\ell_{j}+\chi_{\Xi}]^{*}(u_{ij})=b_{j}+\sigma_{\Xi}(u_{ij}+a_{j}). 
\]
To derive the conjugate $f^*$, we derive first the conjugate $f_1^*$ and $f_2^*$. Note that $f_{1}(\xi)=h\left(g(\xi)\right)$ where $h(u)=\max(0,u)$ is nondecreasing, and $g(\xi)=\lambda(\| \xi-\widehat{\xi}_{i}\| -\varepsilon$). Assuming first that $u_1 \neq 0$ and $\lambda >0$, we have
\begin{align}
f_{1}^{*}(u_{1})= & \inf_{z\geq0}\left\{ \left(zg\right)^{*}(u_{1})+h^{*}(z)\right\} \label{eq01}\\
= & \inf_{0\leq z\leq1}(zg)^{*}(u_{1}) \label{eq02}\\
= & \inf_{0\leq z\leq1}\sup_{\xi}\;u_{1}^{\top}\xi-z\lambda(\| \xi-\widehat{\xi}_{i}\| -\varepsilon) \nonumber\\
= & \inf_{0\leq z\leq1}\sup_{y}\;u_{1}^{\top}(y+\widehat{\xi}_{i})-z\lambda(\| y\| -\varepsilon)  \label{eq04}\\
= & u_{1}^{\top}\widehat{\xi}_{i}+\inf_{0 < z\leq1}z\lambda \left\{\sup_{y}\left(\frac{1}{z\lambda}u_{1}\right)^{\top}y-\| y\| \right\}+ z\lambda\varepsilon \label{eq05}\\
= & u_{1}^{\top}\widehat{\xi}_{i}+\inf_{0< z\leq1,\;z\lambda\geq\| u_{1}\| _{*}}z\lambda\varepsilon \label{eq06} \\
= & \begin{cases}
u_{1}^{\top}\widehat{\xi}_{i}+\| u_{1}\| _{*}\varepsilon, &\| u_{1}\| _{*}\leq\lambda,\\
\infty, &{\rm o.w.},
\end{cases} \label{eq07}
\end{align}
where \eqref{eq01} follows the property of conjugate functions, \eqref{eq02} follows
\[
h^{*}(z)=\begin{cases}
0, & 0\leq z\leq1,\\
\infty, & {\rm o.w.},
\end{cases}
\]
\eqref{eq05} invokes the observation that $z\lambda > 0$ can be assumed without the loss of generality because $\lambda > 0$ and $z=0$ cannot be optimal in \eqref{eq04} (in which case \eqref{eq04} equals to $\infty$ unless $u_1=0$),
 \eqref{eq06} follows the conjugate of the norm $\| y\| $, and finally \eqref{eq07} is because the optimal $z$ must satisfy $z \lambda = \| u_1\| _*$ for any $\| u_1\| _* \leq \lambda$ and does not exist otherwise. From \eqref{eq04}, one can see that in the case $u_1=0$, we have $f_1^*(0) = \inf_{0\leq z \leq 1} \sup_{y} -z\lambda
 (\| y\| -\varepsilon) = \inf_{0\leq z \leq 1} z\lambda \varepsilon = 0$, which is consistent with \eqref{eq07}.

We also have
\[
f_{2}^{*}(u_{2})=
\begin{cases}
u_{2}^{\top}\widehat{\xi}_{i}, &\| u_{2}\| _{*}\leq\lambda\beta,\\
\infty, &{\rm o.w.},
\end{cases}
\]

We thus have the conjugate $f^*(-u)$
\begin{align*}
f^{*}(-u) & =\inf_{u_{1}+u_{2}=-u}f_{1}^{*}(u_{1})+f_{2}^{*}(u_{2})\\
 & =\left\{ \begin{array}{ll}
\inf_{u_{1},u_{2}} & u_{1}^{\top}\widehat{\xi}_{i}+\| u_{1}\| _{*}\varepsilon+u_{2}^{\top}\widehat{\xi}_{i}\\
 & u_{1}+u_{2}=-u,\;\| u_{1}\| _{*}\leq\lambda,\;\| u_{2}\| _{*}\leq\lambda\beta
\end{array}\right\} \\
 & =\left\{ \begin{array}{ll}
\inf_{u_{1}} & \;(-u)^{\top}\widehat{\xi}_{i}+\| u_{1}\| _{*} \varepsilon\\
 & \| u_{1}\| _{*}\leq\lambda,\;\| u+u_{1}\| _{*}\leq\lambda\beta
\end{array}\right\} .
\end{align*}

Thus, the constraint \eqref{dd2} can be equivalently formulated as
\begin{align*}
 & \exists u_{ij} : \begin{array}{l}
b_{j}+\sigma_{\Xi}(u_{ij}+a_{j})+\left\{ \begin{array}{ll}
\inf_{u_{1}} & (-u_{ij})^{\top}\widehat{\xi}_{i}+\| u_{1}\| _{*}\varepsilon\\
 & \| u_{1}\| _{*}\leq\lambda,\;\| u_{ij}+u_{1}\| _{*}\leq\lambda\beta
\end{array}\right\} \end{array} \leq s_{i},\;i=1,...,N,\;j=1,...,K.\\
\Leftrightarrow & \exists u_{ij},v_{ij} : \;  \begin{array}{ll}
b_{j}+\sigma_{\Xi}(u_{ij}+a_{j})+(-u_{ij})^{\top}\widehat{\xi}_{i}+\| v_{ij}\| _{*}\varepsilon\leq s_{i}, & \\
\| v_{ij}\| _{*}\leq\lambda,\;\| u_{ij}+v_{ij}\| _{*}\leq\lambda\beta, &
\end{array} \;i=1,...,N,\;j=1,...,K.
\end{align*}
We thus arrive at the final formulation. Note that in the case $\lambda = 0$, \eqref{qqq} can be directly reduced to $ [-\ell_{j}+\chi_{\Xi}]^{*}(0)\leq s_{i},\;i=1,...,N,\;j=1,...,K$ by the property of convex conjugate, which is consistent with the final formulation.
\qed

~~

\noindent{\bf Proof of Theorem \ref{main_ee}.}
   We first show \eqref{eq-expectile-17} for piecewise linear function, that is, $\ell(x)=\max_{j=1,\ldots,K}\{ a_j^\top x + b_j\}$. Without loss of generality, assume that $\|a_1\|_*\le \cdots\le \|a_K\|_*$ and Denote by $z_{ij}:=\ell_j(\widehat{\xi}_i)=a_j^\top \widehat{\xi}_i+b_j,~\forall~i,j$.
By Theorem \ref{ex-thm1} for $\Xi=\R^m$, letting $x_{ij}=\|y_{ij}\|$ $\forall i,j$,  we have the problem \eqref{eq:expectile-2} is equivalent to
\begin{align}
\sup_{p_{ij},x_{ij}\in\R_+}~~ &\frac1N\sum_{i=1}^N \sum_{j=1}^K  p_{ij}(\|a_j\|_*x_{ij} +z_{ij}) \label{eq:expectile-5} \\
{\rm subject~ to}~~&  \alpha \sum_{i=1}^{N}\sum_{j=1}^Kp_{ij}( x_{ij}-\varepsilon)_+ \le (1-\alpha) \sum_{i=1}^{N}\sum_{j=1}^K p_{ij}(x_{ij}-\varepsilon)_- ,\nonumber\\
& \sum_{j=1}^Kp_{ij} =1,~~i=1,\ldots,N, ~~p_{ij}\ge0,~\forall i,j.\nonumber
\end{align}
For each  $i=1,\ldots,N$, we have the following two observations.
\begin{itemize}\item[(a)]
 If $x_{ij}<\epsilon $ and $x_{ik}< \epsilon$ for some $j<k$, then taking $x_{ij}^*  = x_{ij}  - \delta \ge0 $ and $x_{ik}^* = x_{ik} +\delta p_{ij}/p_{ik}\le \epsilon  $ for any $\delta>0$ yields a feasible solution and a larger value of the objective function.
\item [(b)]If  $x_{ij} \ge \epsilon  $  and $x_{ik} \ge \epsilon $ for some $j<k$, then taking $x_{ij}^*  = \epsilon   $ and $x_{ik}^*= x_{ik} + (x_{ij} - \epsilon )p_{ij}/p_{ik}$ yields a feasible solution and a larger value of the objective function.
\end{itemize}
 Combining the above two observations, for each $i=1,\ldots,N$, we have either one of the following two cases holds:
 \begin{itemize}\item[(i)]
    $x_{ij}\le  \epsilon  $ for all $j=1,\ldots,K.$ In this case, there exists $j_i$ such that $x_{i1}=\ldots=x_{ij_i}=0\le x_{ij_i+1}\le x_{ij_i+2}=\cdots=x_{iK}=\epsilon$.

\item [(ii)] $x_{iK} >  \epsilon  $  and  $x_{ij}\le  \epsilon  $ for all $j=1,\ldots,K-1.$ In this case, there exists $j_i$ such that $x_{i1}=\ldots=x_{ij_i}=0\le x_{ij_i+1}\le x_{ij_i+2}=\cdots=x_{iK-1}=\epsilon$.
We then have the following two cases. \begin{itemize} \item [(ii.a)] If $\alpha \|a_K\|\ge  (1-\alpha) \|a_{j_i+1}\|$, then letting $x_{iK}^* = x_{iK} +  (1-\alpha) \delta p_{ij_i+1}$ and $x_{i{j_i+1}}^* = x_{i {j_i+1}} - \alpha \delta p_{iK}\ge 0$ for any $\delta>0$  yields a feasible solution and a larger value of the objective function.
\item [(ii.b)] If $\alpha \|a_K\|<(1-\alpha) \|a_{j_i+1}\|$, then letting $x_{iK}^* = x_{iK} -  (1-\alpha) \delta p_{ij_i+1}$ and $x_{i{j_i+1}}^* = x_{i {j_i+1}} + \alpha \delta p_{iK}\le \epsilon$ for any $\delta>0$  yields a feasible solution and a larger value of the objective function.
    \end{itemize}
    Therefore, combining the above two cases (ii.a) and (ii.b), we have that  the case (ii)  reduces to  either case (i) or $x_{ij_i+1}=0$.
    \end{itemize}
We can rephrase the above two cases (i) and (ii) as follows: For each $i=1,\ldots,N$, we have either one of the following two cases holds.
\begin{itemize}\item[(i)]
    $x_{ij}\le  \epsilon  $ for all $j=1,\ldots,K.$ In this case, there exists $j_i$ such that $x_{i1}=\ldots=x_{ij_i}=0\le x_{ij_i+1}\le x_{ij_i+2}=\cdots=x_{iK}=\epsilon$.

\item [(ii)] $x_{iK} >  \epsilon  $  and  $x_{ij}\le  \epsilon  $ for all $j=1,\ldots,K-1.$ In this case, there exists $j_i$ such that $x_{i1}=\ldots=x_{ij_i}=0\le x_{ij_i+1}=\cdots=x_{iK-1}=\epsilon$.
    \end{itemize}
Moreover, note that for  any feasible $x_{ij}$, $i=1,\ldots,N$, $j=1,\ldots,K,$ 
define
 $$x_{ij}^*=\begin{cases} \frac{\sum_{i=1}^N p_{ij}x_{ij} 1_{\{x_{ij}\le \epsilon \}} }{\sum_{i=1}^N p_{ij} 1_{\{x_{ij}\le \epsilon \}}} =:y_j~~ &{\rm if} ~x_{ij} \le \epsilon,\\
   \frac{\sum_{i=1}^N p_{ij}x_{ij} 1_{\{x_{ij}>\epsilon\} } }{\sum_{i=1}^Np_{ij} 1_{\{x_{ij}>\epsilon \}}}=:z_j ~~ &{\rm if} ~x_{ij} > \epsilon.\end{cases}
  $$
 We have
 $$
     \sum_{i=1}^{N}\sum_{j=1}^Kp_{ij}(x_{ij}^*-\varepsilon)_+  =  \sum_{i=1}^{N}\sum_{j=1}^Kp_{ij}( x_{ij}-\varepsilon)_+ ~~{\rm and}~~  \sum_{i=1}^{N}\sum_{j=1}^K p_{ij}(x_{ij}^*-\varepsilon)_-=\sum_{i=1}^{N}\sum_{j=1}^K p_{ij}(x_{ij}-\varepsilon)_-,
 $$
 and
 $$
   \sum_{i=1}^N \sum_{j=1}^K  p_{ij}(\|a_j\|_*x_{ij}^* +z_{ij})=\sum_{i=1}^N \sum_{j=1}^K  p_{ij}(\|a_j\|_*x_{ij} +z_{ij}).
 $$
  That is,  $p_{ij}, x_{ij}^*$'s satisfy the constraint and the objective function remains unchanged. 
This combined with the above observation, i.e., $x_{ij}\le \epsilon$ for all $j<K$,
  we have the optimal solution $x_{ij}$ must satisfy:   
 $x_{iK}  \ge  \epsilon  $ for  all $i=1,\ldots,N$,   $x_{ij}\le  \epsilon  $ for all $i=1,\ldots,N$ and $j=1,\ldots,K-1,$ and moreover, there exists $k$ which is independent of $i$ such that $x_{i1}=\cdots=x_{ik}=0\le x_{ik+1}=\cdots=x_{K-1}=\epsilon$.
Therefore, the problem \eqref{eq:expectile-5} is equivalent to 
\begin{align}
\sup_{p_{ij},x_{ij} \in\R_+,k}~~ &\frac1N  \sum_{i=1}^N \left( \sum_{j=k+1}^{K-1} p_{ij} \|a_{j}\|_*  \epsilon  +  p_{iK} \|a_K\|_* x_{iK} + \sum_{j=1}^K p_{ij} z_{ij} \right) \label{eq-expectile-11} \\
{\rm subject~ to}~~&  \alpha   \sum_{i=1}^{N} p_{iK} (x_{iK}-\varepsilon)   \le (1-\alpha) \sum_{i=1}^{N} \sum_{j=1}^{k}  \varepsilon p_{ij} ,\label{eq-expectile-111} \\
& \sum_{j=1}^Kp_{ij} =1,~~i=1,\ldots,N, ~~k=1,\ldots,K.\label{eq-expectile-112}
\end{align}
Since the objective function is increasing in $x_{iK}$,  it suffices to consider the case that the inequality of \eqref{eq-expectile-111} is an equality.
Substituting $\sum_{i=1}^{N} p_{iK} x_{iK} = \sum_{i=1}^{N} \varepsilon p_{iK}+ \beta \sum_{i=1}^{N} \sum_{j=1}^{k}  \varepsilon p_{ij}$ into the objective function, we have the problem (\ref{eq-expectile-11}) is equivalent to
\begin{align}\label{eq-expectile-14}
\sup_{p_{ij}\in\R_+,~k}~~ &\frac1N  \sum_{i=1}^N \left[ \sum_{j=1}^{k} \beta p_{ij} \|a_{K}\|_* \varepsilon +\sum_{j=k+1}^{K-1} p_{ij} \|a_{j}\|_* \epsilon   +  p_{iK} \|a_{K}\|_*\varepsilon + \sum_{j=1}^K p_{ij} z_{ij}  \right] ~~{\rm s.t.}~~  \eqref{eq-expectile-112}.
\end{align}
Note that for any feasible $p_{ij}$, the optimal $k$ must be the largest index of $j$ such that $ \|a_j\|_*\le \beta\|a_K\|_* $, i.e.,  $$k^*=\max\{j=1,\ldots,K-1: \|a_j\|_*\le \beta\|a_K\|_* \}.$$ Hence, the problem \eqref{eq-expectile-14} is equivalent to
\begin{align}
\sup_{p_{ij}\in\R_+}~~ &\frac1N  \sum_{i=1}^N \left[\sum_{j=1}^{k^*} p_{ij}\big( \beta \|a_{K}\|_* \varepsilon  +z_{ij}\big) +\sum_{j=k^*+1}^{K}p_{ij}( \|a_{j}\|_* \epsilon+z_{ij})   \right] ~~
{\rm s.t.} ~\sum_{j=1}^K p_{ij} =1,~\forall~i,\nonumber
\end{align}
whose optimal value equals to
\begin{align}\label{eq-expectile-15}
 \frac1N  \sum_{i=1}^N \max\left\{ \big( \beta \|a_{K}\|_* \varepsilon  +z_{ij}\big)_{j=1}^{k^*},~ ( \|a_{j}\|_* \epsilon+z_{ij}) _{j=k^*+1}^{K}  \right\}   =: \frac1N  \sum_{i=1}^N   {\mathcal E}_i.
\end{align}
Note that  $\|a_j\|_*\varepsilon  +z_{ij}\le \beta \|a_{K}\|_* \varepsilon  +z_{ij}$, $  j=1,\ldots,k^*$,
and thus,
\begin{align*}
 {\mathcal E}_i & \ge \max_{j=1,\ldots,K} \{\|a_{j}\|_* \epsilon+z_{ij}\} = \max_{\|e\|=1} \ell(\widehat{\xi}_i+\epsilon e).
\end{align*}
Also, note that
$\|a_j\|_*\varepsilon  +z_{ij}> \beta \|a_{K}\|_* \varepsilon  +z_{ij}$, $  j=k^*+1,\ldots,K,$
and thus,
\begin{align*}
 {\mathcal E}_i\ge  \beta \|a_{K}\|_* \varepsilon + \max_{j=1,\ldots,K} \{\|a_{j}\|_* \epsilon+z_{ij}\} =  \beta \|a_{K}\|_* \varepsilon+\ell(\widehat{\xi}_i) .
\end{align*}
Therefore,  we have
\begin{align*}
\eqref{eq-expectile-15} \ge  \frac1N  \sum_{i=1}^N \max\left\{ \ell(\widehat{\xi}_i)+ \beta \|a_{K}\|_* \varepsilon ,~\max_{\|e\|=1} \ell(\widehat{\xi}_i+\epsilon e)\right\}.
\end{align*}
On the other hand, obviously we have  \begin{align*}
 &\max_{j=1,\ldots,k^*}\big( \beta \|a_{K}\|_* \varepsilon  +z_{ij}\big)   \le \beta \|a_{K}\|_* \varepsilon+\ell(\widehat{\xi}_i)
~~{\rm and}~~\max_{j=k^*+1,\ldots,K}\big( \|a_{j}\|_* \epsilon+z_{ij} \big) \le   \max_{\|e\|=1} \ell(\widehat{\xi}_i+\epsilon e).\end{align*} Hence, we have
\begin{align*}
\eqref{eq-expectile-15}= \frac1N  \sum_{i=1}^N \max\left\{ \ell(\widehat{\xi}_i)+ \beta \|a_{K}\|_* \varepsilon ,~\max_{\|e\|=1} \ell(\widehat{\xi}_i+\epsilon e)\right\}.
\end{align*}
Therefore, we have shown \eqref{eq-expectile-17} for  piecewise linear functions.
For general convex loss function $\ell$, by similar arguments in  Theorem \ref{main_cvx}, we have the optimal value is \eqref{eq-expectile-17}. \qed

~~

\noindent{\bf Proof of Corollary \ref{wcdis-expectile}.}
Since $\beta=(1-\alpha)/\alpha\to 0$ as $\alpha\to 1$ and $ \ell(\widehat{\xi}_i) < \max_{\|e\|=1} \ell(\widehat{\xi}_i+\epsilon e)$ for all $i=1,\ldots,N$, there exists $\alpha\in (1/2,1)$ large enough such that $ \ell(\widehat{\xi}_i)+ \beta L \varepsilon < \max_{\|e\|=1} \ell(\widehat{\xi}_i+\epsilon e)$ for $i=1,\ldots,N$.
In this case,  \eqref{eq-expectile-17} in Theorem \ref{main_ee} reduces to
\begin{align*}
 \frac1N  \sum_{i=1}^N \max_{\|e\|=1} \ell(\widehat{\xi}_i+\epsilon e).
\end{align*}
One can verify that this optimal value  is attained by $$
\mathbb P_{\epsilon}=\frac1N \sum_{i =1}^N \delta_{ \widehat{\xi}_i+ \epsilon e_i }~~{\rm with}~~ e_i=\arg\max_{\|e\|=1} \ell(\widehat{\xi}_i+\epsilon e),~~i=1,\ldots,N.
$$
 We thus complete the proof.
\qed

\noindent{\bf Proof of Corollary \ref{reduced2}.} The proof is similar to that of Corollary \ref{reduced1} and thus omitted. \qed

\end{document}